\documentclass[a4paper,11pt,fleqn]{article} 

\usepackage[mathscr]{eucal}
\usepackage{amsmath,amssymb,amsthm}
\usepackage{hyperref}
\hypersetup{colorlinks, linkcolor=blue, citecolor=blue, urlcolor=blue}

\allowdisplaybreaks

\makeatletter
\long\def\@makecaption#1#2{%
  \vskip\abovecaptionskip\footnotesize
  \sbox\@tempboxa{#1. #2}%
  \ifdim \wd\@tempboxa >\hsize
    #1. #2\par
  \else
    \global \@minipagefalse
    \hb@xt@\hsize{\hfil\box\@tempboxa\hfil}%
  \fi
  \vskip\belowcaptionskip}
\makeatother

\newcommand{\todo}[1][\null]{\ensuremath{\clubsuit}}

\newcommand{\noprint}[1]{}
\newcommand{\checked}[1][\null]{\ensuremath{\boldsymbol{\surd}}}

\newcommand{\p}{\partial}

\newcommand{\sgn}{\mathop{\rm sgn}\nolimits}

\newcommand{\lsemioplus}{\mathbin{\mbox{$\lefteqn{\hspace{.72ex}\rule{.4pt}{1.2ex}}{\in}$}}}

\newtheorem{theorem}{Theorem}
\newtheorem{lemma}[theorem]{Lemma}
\newtheorem{corollary}[theorem]{Corollary}
\newtheorem{proposition}[theorem]{Proposition}
\newtheorem*{problem*}{Problem}
{\theoremstyle{definition}
\newtheorem{definition}[theorem]{Definition}

\newtheorem{remark}[theorem]{Remark}
\newtheorem*{remark*}{Remark}
\newtheorem*{definition*}{Definition}
}

\begin{document}

\par\noindent {\LARGE\bf
Mapping method of group classification
\par}

\vspace{4mm}\par\noindent {\large Stanislav Opanasenko$^\dag$ and Roman O.\ Popovych$^\ddag$}

\vspace{4mm}\par\noindent {\it
$^\dag$\,Department of Mathematics and Statistics, Memorial University of Newfoundland,\\
$\phantom{^\dag}$~St.\ John's (NL) A1C 5S7, Canada}

\vspace{3mm}\par\noindent {\it
$^\ddag$\,Mathematical Institute, Silesian University in Opava, Na Rybn\'\i{}\v{c}ku 1, 746 01 Opava,\\
$\phantom{^\ddag}$\,Czech Republic\\
$\phantom{^\ddag}$\,Fakult\"at f\"ur Mathematik, Universit\"at Wien, Oskar-Morgenstern-Platz 1, A-1090 Wien, Austria\\
$\phantom{^\ddag}$\,Institute of Mathematics of NAS of Ukraine, 3 Tereshchenkivska Str., 01024 Kyiv, Ukraine}

{\vspace{3mm}\par\noindent {\it
\textup{E-mail:} sopanasenko@mun.ca, rop@imath.kiev.ua
}\par}

\vspace{8mm}\par\noindent\hspace*{10mm}\parbox{140mm}{\small
We revisit the entire framework of group classification of differential equations.
After introducing the notion of weakly similar classes of differential equations,
we develop the mapping method of group classification for such classes,
which generalizes all the versions of this method that have been presented in the literature.
The mapping method is applied to group classification of various classes of Kolmogorov equations and of Fokker--Planck equations in the case of space dimension one.
The equivalence groupoids and the equivalence groups of these classes are computed.
The group classification problems for these classes with respect to the corresponding equivalence groups
are reduced to finding all inequivalent solutions of heat equations with inequivalent potentials
admitting Lie-symmetry extensions.
This reduction allows us to exhaustively solve the group classification problems
for the classes of Kolmogorov and Fokker--Planck equations with time-independent coefficients.
}\par\vspace{5mm}

\noprint{
}

\section{Introduction}

Problems of group classification of classes of differential equations were first considered by Sophus Lie himself.
Thus, in~\cite{Lie1881} he solved the group classification problem
for the entire class of second-order linear partial differential equations in two complex independent variables.
The research interest in such classifications was renewed by Ovsiannikov late 1950s,
who in particular revisited the above Lie's classification in invariant terms~\cite{Ovsiannikov1982}.
In the past decades the field of group classification has been steadily developed,
but until this very day the consideration of complicated classes of differential equations,
e.g., of those appearing in problems of invariant parameterization~\cite{PopovychBihlo2010},
within the framework of group classification is mostly avoided.
The reason is partly in a lack of methods to handle such classes.
The applicability of the powerful algebraic method of group classification \cite{basa2001a,gaze1992a,gung2004a,popo2010a}
is rather limited to classes with nice transformational properties,
which includes various kinds of normalized classes \cite{KurujyibwamiBasarabHorwathPopovych2018,kuru2020a,popo2010a},
see~\cite{VaneevaBihloPopovych2020} for a recent breakthrough in extending the algebraic method to specific non-normalized classes.
(Definitions of related notions are presented in Section~\ref{FP:sec:TheorBack} below.)
For generic classes, the simplest version of the direct method of group classification is still heavily used
although there are better alternatives such as
the partition of the class under study into its normalized subclasses~\cite{kuru2020a,popo2010a,VaneevaPosta2017},
the furcate splitting method~\cite{bihl2020a,NikitinPopovych2001,OpanasenkoBoykoPopovych2020}
and the mapping method~\cite{VaneevaPopovychSophocleous2009,OpanasenkoBihloPopovych2020}.

The generalization of the mapping method of group classification is the primary objective of the present paper.
For this purpose, we update basic notions of the field of group classification
and formalize the generalized group classification problem for a general class of differential equations.
The efficiency of the developed method is demonstrated by means of its application to group classification of various classes
of (1+1)-dimensional second-order linear evolution equations, including classes of Kolmogorov and Fokker--Planck equations.

The history of the group analysis of (1+1)-dimensional second-order linear evolution equations dates
back to the aforementioned seminal paper~\cite{Lie1881} of Sophus Lie,
where the group classification problem for these equations was solved
as a part of the group classification problem for the entire class of second-order linear partial differential equations
in two complex independent variables.
After Ovsiannikov's revision~\cite{Ovsiannikov1982},
there were a number of papers restating, specifying or developing the above Lie's result.
We refer the reader to~\cite{PopovychKunzingerIvanova2008} for
a list of these papers and a modern treatment of the problem.
The equivalence problem for (1+1)-dimensional second-order linear evolution equations
was initiated in~\cite{Bluman1980,Cherkasov1957,Lie1881} and
further considered in~\cite{Gungor2018,JohnpillaiMahomed2001,Morozov2003}.
Darboux transformations between such equations were studied, e.g.,
in \cite{BlumanShtelen2004,BlumanYuzbasi2020,MatveevSalle1991,PopovychKunzingerIvanova2008}.

In general, a group classification list for a class~$\mathcal L$ of differential equations
may not be contained in a reasonable group classification list for its superclass~$\bar{\mathcal L}$.
Thus, although the general point equivalence within the class~$\bar{\mathcal L}$
is obviously consistent with that within the subclass~$\mathcal L$ with respect to the inclusion $\mathcal L\subseteq\bar{\mathcal L}$,
no system from~$\mathcal L$ may belong to the known classification list for~$\bar{\mathcal L}$ at all.
If the group classification problem for the class~$\bar{\mathcal L}$ is solved up to the equivalence generated by its equivalence group or a subgroup thereof,
then this equivalence may even be inconsistent with the analogous equivalence in~$\mathcal L$ with respect to the inclusion $\mathcal L\subseteq\bar{\mathcal L}$,
and the known classification list for~$\bar{\mathcal L}$ cannot be easily adjusted to obtain a group classification list for the class~$\mathcal L$.
For example, the group classification problem for the class~$\mathcal E$ of (1+1)-dimensional second-order linear evolution equations
\begin{gather}\label{eq:FPInhomoLinearEquation}
u_t=A(t,x)u_{xx}+B(t,x)u_x+C(t,x)u+D(t,x)
\end{gather}
 with $A\ne0$ is traditionally solved by reducing it to that for the class~$\mathcal P$ of heat equations with potentials,
\begin{gather*}
u_t=u_{xx}+C(t,x)u,
\end{gather*}
which is actually a simple incarnation of the mapping method of group classification.
Then all the listed equations representing Lie-symmetry extensions in the class~$\mathcal E$ belong to~$\mathcal P$
and thus, this classification list is not appropriate for various classes of Kolmogorov or Fokker--Planck equations.
Moreover, the equivalence groups of these classes
are not consistent with the equivalence groups of the classes~$\mathcal E$ and~$\mathcal P$ in several senses,
e.g., with respect to the relations of inclusion and of weak similarity.

The Fokker--Planck equations are prominent equations in stochastic analysis
by virtue of the fact that they describe the evolution of the transition probability density in diffusion processes~\cite{Feller1968}.
In the case of one spatial variable, the entire class~$\bar{\mathcal F}$ of the Fokker--Planck equations consists of equations of the form
\begin{gather*}
u_t=(A(t,x)u)_{xx}+(B(t,x)u)_x,
\end{gather*}
where~$A$ and~$B$ are functions of~$(t,x)$, called the diffusion coefficient and the drift, respectively, with $A\ne0$.
Besides stochastic analysis, Fokker--Planck equations play an important role in various natural sciences,
including biology and physics~\cite{Risken1989}.

The study of equations from the class~$\bar{\mathcal F}$ within the framework of the group analysis of differential equations
commenced in the middle of the 1980s,
when the Fokker--Planck equations describing the Ornstein--Uhlenbeck and the Rayleigh processes
were considered in~\cite{AnChenGuo1984} and~\cite{SastriDunn1985}, respectively.
Point transformations that reduce these and a number of other physically important Fokker--Planck equations
to the heat equation were constructed in~\cite{ShtelenStogny1989}.
Lie symmetries of general (1+1)-dimensional Fokker--Planck equations
were discussed in~\cite{CicognaVitoli1990,Kozlov2013,SpichakStognii1999}.

The adjoint class~$\bar{\mathcal K}$ is constituted by Kolmogorov equations,
\begin{gather*}
u_t=A(t,x)u_{xx}+B(t,x)u_x
\end{gather*}
with $A\ne0$, which are formally adjoint to Fokker--Planck equations if $A$ is replaced by~$-A$.
This fact reflects on the similarity of transformational properties of the classes~$\bar{\mathcal F}$ and~$\bar{\mathcal K}$.
In particular, the maximal Lie invariance algebras of Kolmogorov and Fokker--Planck equations
with the arbitrary-element tuples~$(-A,B)$ and $(A,B)$, respectively, are isomorphic.
Furthermore, there is a connection between the equivalence groups of these classes,
see~\cite{PopovychKunzingerIvanova2008} and Section~\ref{sec:FPEquivalenceGroup}.
Thus, we can switch between considerations of these classes whenever we feel this simplifies a problem at hand~\cite[Section~5]{PopovychKunzingerIvanova2008}.

Our personal interest in the class~$\bar{\mathcal K}$ lies
in the fact that reaction--diffusion equations of the form $u_t=f(u_x)u_{xx}+g(u)$ with $f\ne0$
are linearized by the hodograph transformation~$\tilde t=t$, $\tilde x=u$, $\tilde u=x$
to Kolmogorov equations if $f(u_x)=cu_x^{-2}$, where $c$ is an arbitrary nonzero constant.
In~\cite{OpanasenkoBoykoPopovych2020}, the class of equations of the above form with arbitrary $f\ne0$ was represented as a
union of four subclasses, and three of them were classified up to the equivalences generated by their equivalence groups.
At the same time, the fourth subclass, which consists of linearizable equations, was classified only up to the general point equivalence.
The group classification of the last subclass modulo the equivalence generated by its equivalence group
is reduced to the analogous classification of the similar class~$\mathcal K'^{+1}$ of the reduced ($A=1$) Kolmogorov equations with stationary drifts.
Unfortunately, \emph{there are no group classifications of classes of Kolmogorov and Fokker--Planck equations
with respect to the corresponding equivalence groups in the literature}, including the class~$\mathcal K'^{+1}$.
We fill this gap in the present paper.

The classes~$\bar{\mathcal K}$ and~$\bar{\mathcal F}$ are not normalized, and their group classification problems cannot be tackled with the algebraic method.
Moreover, the action groupoids of their equivalence groups constitute only small portions of the corresponding equivalence groupoids.
In other words, within each of the classes~$\bar{\mathcal K}$ and~$\bar{\mathcal F}$ there
are too many admissible transformations that are not generated by equivalence transformations.
Firstly, this suggests us that a partition of any of these classes into normalized subclasses is unlikely,
and secondly, it means that the number of inequivalent cases of Lie-symmetry extensions
in both~$\bar{\mathcal K}$ and~$\bar{\mathcal F}$ modulo the corresponding equivalence groups
is significantly higher than the analogous number up to the general point equivalence.
At the same time, we show that the group classification problems for~$\bar{\mathcal K}$ and~$\bar{\mathcal F}$
modulo the corresponding equivalence groups are equivalent to the analogous problems
for the subclasses~$\mathcal K^\epsilon\subset\bar{\mathcal K}$ and~$\mathcal F^\epsilon\subset\bar{\mathcal F}$ of reduced Kolmogorov and Fokker--Planck equations,
which are singled out by the constraint $A=\epsilon$.
Here and in what follows $\epsilon$ is a fixed number in $\{-1,1\}$.
Moreover, the subclasses~$\mathcal K^\epsilon$ and~$\mathcal F^\epsilon$ are weakly similar to the class~$\mathcal P^\epsilon$ of heat equations with potentials,
see Definition~\ref{FP:def:WeaklySimilarClasses}.
The class~$\mathcal P^\epsilon$ is disjointedly semi-normalized with respect to the linear superposition of solutions~\cite{PopovychKunzingerIvanova2008}.
Hence the equivalences generated by its equivalence group~$G^\sim_{\mathcal P^\epsilon}$ and by its equivalence groupoid are the same.
The group classification of this class is quite simple,
and the well-known classification list contains only three inequivalent cases of essential Lie-symmetry extensions~\cite{Lie1881,Ovsiannikov1982}.
In~\cite{PopovychKunzingerIvanova2008} with the help of the standard version of the mapping method of group classification,
this list was mapped to the group classification lists for the classes~$\mathcal K^\epsilon$ and~$\mathcal F^\epsilon$ up to the general point equivalence,
each of which is comprised of four cases of inequivalent essential Lie-symmetry extensions.

In the present paper, we apply the advanced version of the mapping method to the group classification of the classes~$\mathcal K^\epsilon$ and~$\mathcal F^\epsilon$
modulo the equivalences generated by the corresponding equivalence groups.
Here, the image of any equation from~$\mathcal P^\epsilon$ is a set of equations from~$\mathcal K^\epsilon$ (resp.\ $\mathcal F^\epsilon$)
that necessarily contains equations that are inequivalent with respect to the corresponding equivalence group.
Thus, the mapping method is applied for the first time to classes with so general weak similarity.
It reduces the group classification problems for the classes~$\mathcal K^\epsilon$ and~$\mathcal F^\epsilon$,
and hence for the classes~$\bar{\mathcal K}$ and~$\bar{\mathcal F}$,
to finding all solutions of the $G^\sim_{\mathcal P^\epsilon}$-inequivalent heat equations with potentials
that are inequivalent with respect to the essential point symmetry groups thereof.
This means that these classification problems are in fact no-go problems 
in the sense of explicitly listing inequivalent Lie-symmetry extensions.
Nonetheless, the mapping method allows us to solve, in a closed form,
the group classification problems for the classes~$\bar{\mathcal K}'$ and~$\bar{\mathcal F}'$
of Kolmogorov and Fokker--Planck equations with time-independent coefficients and their reduced subclasses~$\mathcal K'^\epsilon$ and~$\mathcal F'^\epsilon$
modulo the equivalences generated by the corresponding equivalence groups.
The basis for this application of the mapping method is the well-known group classification of the subclass~$\mathcal P'$ of~$\mathcal P$
constituted by the heat equations with time-independent potentials~$C$.
It is necessary to note that Fokker--Planck equations arising in applications,
in particular those describing the Ornstein--Uhlenbeck and the Rayleigh processes,
have time-independent coefficients.

The rest of the paper is organized as follows.
The advanced theory of point transformations between differential equations is given in Section~\ref{FP:sec:TheorBack}.
Thus, Sections~\ref{FP:sec:AdmEqTransf}, \ref{FP:sec:GroupClassProblem} and~\ref{FP:sec:MapMethod}
are devoted to the notions of admissible and equivalence transformations within classes of differential equations,
a revisited formulation of the group classification problem and developing the mapping method of group classification, respectively.
In Section~\ref{sec:FPEquivalenceGroup}, we compute the equivalence groupoids and the equivalence groups
of the class~$\mathcal E$, of its reparameterization~$\breve{\mathcal E}$
and of various subclasses of~$\mathcal E$ and~$\breve{\mathcal E}$.
In particular, we first construct a minimal self-consistent generating set of admissible transformations
for the class~$\mathcal P'$ of the heat equations with time-independent potentials~$C$.
Using the known group classification of the class~$\mathcal P^\epsilon$ and the developed version of the mapping method,
in Section~\ref{sec:FPGCArbitrary} we revisit the solution of the group classification problems
for the classes~$\bar{\mathcal K}$, $\mathcal K^\epsilon$, $\bar{\mathcal F}$ and~$\mathcal F^\epsilon$
up to the general point equivalence
and exhaustively study the analogous problems up to the equivalences
generated by the corresponding equivalence groups.
Applying the same approach, in Section~\ref{sec:FPGCTimeIndependent}
we completely solve the group classification problems, up to the equivalences
generated by the corresponding equivalence groups, for the subclasses of the above classes
that consists of the equations with time-independent coefficients.
The most interesting results of the present paper are discussed in Section~\ref{sec:FPConclusion}.

\section[Point transformations in classes of differential equations and the mapping method of group classification]
{Point transformations in classes of differential equations\\ and the mapping method of group classification}
\label{FP:sec:TheorBack}

Let~$\mathcal L=\{\mathcal L_\theta\mid \theta\in\mathcal S\}$ be a class (of systems) of differential equations, $\mathcal L_\theta$, of order~$r$
in independent variables~$x=(x_1,\dots,x_n)$ and dependent variables $u=(u^1,\dots,u^m)$
with the arbitrary-element tuple $\theta=\big(\theta^1(x,u_{(r)}),\dots,\theta^k(x,u_{(r)})\big)$
running through the solution set~$\mathcal S$ of an auxiliary system~$\mathsf S$ of differential equations and differential inequalities. 
(The notation of variables in this section differs from that in the other sections.)
The systems $\mathcal L_\theta$ have the form $L_\theta(x,u_{(r)}):=L\big(x,u_{(r)},\theta_{(q)}(x,u_{(r)})\big)=0$,
where $L$ is a fixed tuple of $r$th order differential functions in~$u$ parameterized by derivatives of~$\theta$.
The short-hand notations $u_{(r)}$ and $\theta_{(q)}$ are used
for the tuples of derivatives of the corresponding~functions with respect to their arguments up to orders~$r$ and~$q$, respectively.
In the auxiliary system~$\mathsf S$, the jet variables $(x,u_{(r)})$ and the arbitrary elements~$\theta$
play the role of the independent and the dependent variables, respectively.

\subsection{Admissible and equivalence transformations}\label{FP:sec:AdmEqTransf}

Admissible transformations of the class~$\mathcal L$ are triples~$\mathcal T$ of the form~$(\theta_1,\Phi,\theta_2)$,
where $\theta_1,\theta_2\in\mathcal S$ and $\Phi$ is a point transformation%
\footnote{%
In the case of one dependent variable, one can also consider various structures consisting of contact transformations instead of point ones.
}
relating the systems $\mathcal L_{\theta_1}$ and~$\mathcal L_{\theta_2}$,
$\Phi_*\mathcal L_{\theta_1}=\mathcal L_{\theta_2}$.
Here $\Phi_*\mathcal L_{\theta_1}$ is the pushforward of the system~$\mathcal L_{\theta_1}$ by the transformation~$\Phi$,
and the equality \mbox{$\Phi_*\mathcal L_{\theta_1}=\mathcal L_{\theta_2}$} should be understood as
\smash{$(\mathop{\rm pr}_{(r)}\Phi)^* L_{\theta_2}|^{}_{\mathcal L_{\theta_1}}=0$},
where $\mathop{\rm pr}_{(r)}\Phi$ is the $r$th prolongation of the transformation~$\Phi$.
In this case, we also say $\Phi_*\theta_1=\theta_2$.
For an admissible transformation $\mathcal T=(\theta_1,\Phi,\theta_2)$,
$\theta_1$~and~$\theta_2$ are respectively called the source and the target of~$\mathcal T$,
${\rm s}(\mathcal T):=\theta_1$ and ${\rm t}(\mathcal T):=\theta_2$.
The composition~$\mathcal T\star\mathcal T'$ of admissible transformations
$\mathcal T=(\theta_1,\Phi,\theta_2)$ and $\mathcal T'=(\theta_1',\Phi',\theta_2')$
is defined if and only if $\theta_1'={\rm s}(\mathcal T')={\rm t}(\mathcal T)=\theta_2$,
and it is equal to $\mathcal T\star\mathcal T':=(\theta_1,\Phi'\circ\Phi,\theta_2')$,
which gives a partial multiplication for admissible transformations of~$\mathcal L$.
For any $\theta\in\mathcal S$, the unit at~$\theta$ is \smash{${\rm id}_\theta:=(\theta,{\rm id}_{(x,u)},\theta)$},
where \smash{${\rm id}_{(x,u)}$} is the identity transformation of~$(x,u)$.
The inverse of~$\mathcal T$ is $\mathcal T^{-1}:=(\theta_2,\Phi^{-1},\theta_1)$, where
$\Phi^{-1}$ is the inverse of~$\Phi$.
This makes the set
\[
\mathcal G^\sim_{\mathcal L}=\{(\theta_1,\Phi,\theta_2)\mid\theta_1,\theta_2\in\mathcal S,\ \Phi\in\mathrm{Diff}^{\rm loc}_{(x,u)}\colon
\Phi_*\mathcal L_{\theta_1}=\mathcal L_{\theta_2}\}
\]
of admissible transformations of the class~$\mathcal L$
a groupoid, $\mathcal G^\sim_{\mathcal L}\rightrightarrows\mathcal S$, called the \emph{equivalence groupoid} of this class,
where the set of objects is ${\rm s}(\mathcal G^\sim_{\mathcal L})={\rm t}(\mathcal G^\sim_{\mathcal L})=\mathcal S$
and the base groupoid is $\mathcal S\rightrightarrows\mathcal S:=\{{\rm id}_\theta\mid\theta\in\mathcal S\}$.

\emph{Usual equivalence transformations} of the class~$\mathcal L$ 
are the point transformations (i.e., local diffeomorphisms) in the space coordinatized by $(x,u_{(r)},\theta)$
that are projectable to the spaces coordinatized by~$(x,u)$ and~$(x,u_{(r)})$,
are consistent with the contact structure of the jet space $\mathrm J^r(\mathbb R^n_x\times\mathbb R^m_u)$
and leave~$\mathcal L_\theta$ in~$\mathcal L$ for any value of the arbitrary-element tuple~$\theta$ of~$\mathcal L$.
Such transformations constitute the \emph{usual equivalence (pseudo)group}~$G^\sim_{\mathcal L}$ of the class~$\mathcal L$.
In what follows, we omit the attribute ``usual'' for such groups since we do not deal with other types of equivalence transformations in the present paper. 
Most objects that are called groups in the present paper are in fact pseudogroups of local diffeomorphisms 
but we omit the attribute ``pseudo'' as well. 

Let $\pi$ be the natural projection from the space coordinatized by $(x,u_{(r)},\theta)$ onto the space coordinatized by $(x,u)$.
There are two particularly important subgroupoids of~$\mathcal G^\sim_{\mathcal L}$,
the action groupoid of the equivalence group~$G^\sim_{\mathcal L}$ and the fundamental groupoid~$\mathcal G^{\rm f}_{\mathcal L}$ of the class~$\mathcal L$,
\begin{gather*}
\mathcal G^{G^\sim_{\mathcal L}}=\mathcal G(G^\sim_{\mathcal L})
:=\{(\theta,\pi_*\mathscr T,\mathscr T_*\theta)\mid\theta\in\mathcal S,\ \mathscr T\in G^\sim_{\mathcal L}\},
\\
\mathcal G^{\rm f}_{\mathcal L}
:=\{\mathcal T\in\mathcal G^\sim_{\mathcal L}\mid{\rm s}(\mathcal T)={\rm t}(\mathcal T)\}
=\{(\theta,\Phi,\theta)\mid \theta\in\mathcal S, \Phi\in G_\theta\}
=\smash{\bigsqcup_{\theta\in\mathcal S} \mathcal G_\theta}.
\\[-2.2ex]
\end{gather*}
Here for each $\theta\in\mathcal S$, $\mathcal G_\theta$ is
the vertex group of the system~$\mathcal L_\theta$,
$\mathcal G_\theta=\{\mathcal T\in\mathcal G^\sim_{\mathcal L}\mid{\rm s}(\mathcal T)={\rm t}(\mathcal T)=\theta\}$.
It is associated with the complete point symmetry group~$G_\theta$ of~$\mathcal L_\theta$:
$(\theta,\Phi,\theta)\in\mathcal G_\theta$ if and only if $\Phi\in G_\theta$.
The fundamental groupoid~$\mathcal G^{\rm f}_{\mathcal L}$ of~$\mathcal L$ is a normal subgroupoid of~$\mathcal G^\sim_{\mathcal L}$.

Any wide%
\footnote{%
A subgroupoid~$\mathcal U$ of a groupoid~$\mathcal G$ is called \emph{wide}
if the sets of objects of~$\mathcal U$ and~$\mathcal G$ coincide, ${\rm s}(\mathcal U)={\rm s}(\mathcal G)$.
More generally, a subset~$\mathcal M$ of a groupoid~$\mathcal G$ is called \emph{wide}
if the set ${\rm s}(\mathcal M)$ coincides with the set of objects of~$\mathcal G$, ${\rm s}(\mathcal M)={\rm s}(\mathcal G)$.
}
subgroupoid~$\mathcal U$ of the equivalence groupoid~$\mathcal G^\sim_{\mathcal L}$ of the class~$\mathcal L$
generates an equivalence relation for elements of~$\mathcal G^\sim_{\mathcal L}$.
More specifically, $\mathcal T,\tilde{\mathcal T}\in\mathcal G^\sim_{\mathcal L}$ are called \emph{$\mathcal U$-equivalent}
if there exist $\mathcal T_1,\mathcal T_2\in\mathcal U$ such that $\tilde{\mathcal T}=\mathcal T_1\star\mathcal T\star\mathcal T_2$.%
\footnote{%
One can also introduce other kinds of equivalence with respect to~$\mathcal U$ for elements of~$\mathcal G^\sim_{\mathcal L}$,
see~\cite{VaneevaBihloPopovych2020}.
}
We call $\mathcal T\in\mathcal G^\sim_{\mathcal L}$ composable with $\mathcal T'\in\mathcal G^\sim_{\mathcal L}$
up to the $\mathcal U$-equivalence
if an admissible transformation that is $\mathcal U$-equivalent to~$\mathcal T$ is composable with~$\mathcal T'$ or,
equivalently, if $\mathcal T$ is composable with an admissible transformation that is $\mathcal U$-equivalent to~$\mathcal T'$.
A set $\mathcal B=\{\mathcal T_\gamma\in\mathcal G^\sim_{\mathcal L}\mid\gamma\in\Gamma\}$,
where $\Gamma$ is an index set, is called
a \emph{generating set of admissible transformations for the class~$\mathcal L$ up to the $\mathcal U$-equivalence}
if any admissible transformation of this class can be represented
as the composition of a finite number of elements of the set~$\mathcal B\cup\hat{\mathcal B}\cup\mathcal U$.
Here \smash{$\hat{\mathcal B}=\{\mathcal T^{-1}\mid \mathcal T\in\mathcal B\}$}.
A subset~$\mathcal B$ of $\mathcal G^\sim_{\mathcal L}$ is \emph{self-consistent with respect to the $\mathcal U$-equivalence}
if the composability of elements of \smash{$\mathcal B \cup \hat{\mathcal B}$} up to the $\mathcal U$-equivalence implies their usual composability.
The standard choice for~$\mathcal U$ is $\mathcal U=\mathcal G^{G^\sim_{\mathcal L}}$~\cite{VaneevaBihloPopovych2020},
and then one can use the term `$G^\sim_{\mathcal L}$-equivalence' instead of `$\mathcal G^{G^\sim_{\mathcal L}}$-equivalence'.
For specific classes of differential equations, e.g., for classes of linear homogeneous systems of differential equations,
other choices of~$\mathcal U$ can be more convenient, see Section~\ref{sec:FP:HeatEqsWithPotsEquivAndAdmTrans}.

The class~$\mathcal L$ is called \emph{normalized} in the usual sense if $\mathcal G^\sim_{\mathcal L}=\mathcal G^{G^\sim_\mathcal L}$
and \emph{semi-normalized} if the groupoid~$\mathcal G^\sim_{\mathcal L}$
coincides with the Frobenius product of~$\mathcal G^{G^\sim_{\mathcal L}}$ and~$\mathcal G^{\rm f}_{\mathcal L}$,
\[
\mathcal G^\sim_{\mathcal L}=\mathcal G^{G^\sim_{\mathcal L}}\star\mathcal G^{\rm f}_{\mathcal L}
:=\big\{\mathcal T\star\mathcal T'\mid
\mathcal T\in\mathcal G^{G^\sim_{\mathcal L}},\,\mathcal T'\in\mathcal G^{\rm f}_{\mathcal L},\,{\rm t}(\mathcal T)={\rm s}(\mathcal T')\big\}.
\]

More generally, consider a subgroup $H$ of~$G^\sim_{\mathcal L}$
and a family $N_{\mathcal S}:=\{N_\theta<G_\theta\mid\theta\in\mathcal S\}$ of subgroups of the groups~$G_\theta$
that is \emph{uniform} with respect to~$H$,
i.e., $\mathcal N_{{\rm s}(\mathcal T)}\star\mathcal T=\mathcal T\star\mathcal N_{{\rm t}(\mathcal T)}$
for any $\mathcal T\in\mathcal G^H$.
Here $\mathcal G^H:=\{(\theta,\pi_*\mathscr T,\mathscr T_*\theta)\mid\theta\in\mathcal S,\ \mathscr T\in H\}$ is the action groupoid of~$H$,
and for any $\theta\in\mathcal S$,
$\mathcal N_\theta:=\{(\theta,\Phi,\theta)\mid\theta\in\mathcal S,\,\Phi\in N_\theta\}$ is the subgroup of the vertex group~$\mathcal G_\theta$
that is associated with~$N_\theta$.
Denote $\mathcal N^{\rm f}:=\sqcup_{\theta\in\mathcal S}\mathcal N_\theta$.
Then the Frobenius product
\[
\mathcal N^{\rm f}\star\mathcal G^H:=\big\{\mathcal T\star\mathcal T'\mid
\mathcal T\in\mathcal N^{\rm f},\,\mathcal T'\in\mathcal G^H,\,{\rm t}(\mathcal T)={\rm s}(\mathcal T')\big\},
\]
is a subgroupoid of~$\mathcal G^\sim_{\mathcal L}$.
We call the class~$\mathcal L|_{\mathcal S}$
\emph{semi-normalized with respect to the subgroup $H$ of~$G^\sim_{\mathcal L}$ and the family $N_{\mathcal S}$ of subgroups of the point symmetry groups}
if $\mathcal N^{\rm f}\star\mathcal G^H=\mathcal G^\sim_{\mathcal L}$,
and we add the modifier ``disjointedly'' to ``semi-normalized'' if additionally $\mathcal N^{\rm f}\cap\mathcal G^H=\mathcal S\rightrightarrows\mathcal S$.
In~particular,
if the class~$\mathcal L$ consists of linear homogeneous systems of differential equations,
$H=G^\sim_{\mathcal L}$,
and  for any $\theta\in\mathcal S$, $N_\theta$ is the subgroup of~$G_\theta$ that is constituted by
the point symmetry transformations of the linear superposition of solutions,
$N_\theta=\{\Phi\in G_\theta\mid \Phi\colon\tilde x=x,\,\tilde u=u+f(x)\}$
with $f(x)=\left(f^1(x),\dots,f^m(x)\right)$ running through the solution set of~$\mathcal L_\theta$,
then the class~$\mathcal L$ is called
\emph{(disjointedly) semi-normalized with respect to the linear superposition of solutions}.

See~\cite{VaneevaBihloPopovych2020} and references therein for more definitions of structures
related to point or contact transformations within classes of differential equations.

A \emph{subclass}~$\mathcal L'$ of the class~$\mathcal L$ is a specific subset of~$\mathcal L$,
$\mathcal L'=\{\mathcal L_\theta\mid\theta\in\mathcal S'\subseteq\mathcal S\}\subseteq\mathcal L$,
where $\mathcal S'$ is the solution set of an auxiliary system~$\mathsf S'$
obtained  by supplementing the system~$\mathsf S$
with the system~$\check{\mathsf S}$ of additional auxiliary differential equations and/or differential inequalities.
We say that subclass~$\mathcal L'$ is singled out from the class~$\mathcal L$
by the system~$\check{\mathsf S}$ of differential constraints on the arbitrary-element tuple~$\theta$.
Among such constraints,
we distinguish those that can be imposed by relating the class~$\mathcal L$ via wide subsets of admissible transformations
(usually, from $\mathcal G^{G^\sim_{\mathcal L}}$) to its proper subclasses.
In contrast to the general constraints on the arbitrary-element tuple, these specific constraints are called \emph{gauges}.
More precisely, let $\mathcal L'$ be a subclass of~$\mathcal L$ that is singled out from~$\mathcal L$
by the additional auxiliary system~$\check{\mathsf S}$ on the arbitrary-element tuple~$\theta$.
We say that a wide subset~$\mathcal M$ of~$\mathcal G^\sim_{\mathcal L}$ gauges
the class~$\mathcal L$ to its subclass~$\mathcal L'$ or, equivalently,
the set~$\mathcal S$ to its subset~$\mathcal S'$ if ${\rm t}(\mathcal M)=\mathcal S'$.
Then the system~$\check{\mathsf S}$ is called the gauge on~$\theta$ associated with~$\mathcal M$.

Suppose that the system~$\mathsf S'$ contains algebraic equations $S^i(x,u_{(r)},\theta)=0$ in~$\theta$,
where the index~$i$ runs through a finite index set~$I$, $|I|\leqslant\#\theta=k$, and
the function tuple $S:=(S^i,i\in I)$ is of the maximal rank with respect to~$\theta$.
Here the attribute ``algebraic'' means ``not being genuinely differential''.
The arbitrary-element tuple~$\theta$ can (locally) be split into
``parametric'' and ``leading'' parts~$\hat\theta$ and~$\check\theta$
with $\#\check\theta=|I|$ and $\det(\p S/\p \check\theta)\ne0$.%
\footnote{%
Throughout the paper, any inequality of the form $f\ne0$ with a function~$f$ means that
this function is nonzero for each point in its domain.
}
Solving the system $S=0$ with respect to~$\check\theta$,
we can assume without loss of generality that \smash{$S=\check\theta-\Theta(x,u_{(r)},\hat\theta)$}.
Excluding $\check\theta$ from the system~$L_\theta$ due to the substitution $\check\theta=\Theta(x,u_{(r)},\hat\theta)$ for~$\check\theta$
and considering~$\hat\theta$ as the only arbitrary elements,
we reparameterize the class~$\mathcal L'$ to a class~$\mathcal L'_{\rm r}$ with a less number of arbitrary elements.
A~benefit of the reparameterization is that it allows one to avoid
the consideration of insignificant equivalence transformations~\cite[Appendix~A]{BoykoLokazyukPopovych2021},
whose presence complicates the computation of equivalence transformations of the class~$\mathcal L'$.
At the same time, in contrast to~$\mathcal L'$,
the reparameterized class~$\mathcal L'_{\rm r}$ is not formally a subclass of~$\mathcal L$,
although it can be embedded in~$\mathcal L$ via naturally identifying with its counterpart~$\mathcal L'$,
$\mathcal L'_{\rm r}\hookrightarrow\mathcal L$.
This is why we can still informally say that the class~$\mathcal L'_{\rm r}$ is a subclass~of~$\mathcal L$.

If the system~$\mathsf S'$ contains genuinely differential constraints that can algebraically be solved with respect to
a subtuple~$\check\theta$ of~$\theta$, \smash{$\check\theta=\Theta(x,u_{(r)},\hat\theta_{(q)})$},
where $\hat\theta$ is the complement of~$\check\theta$ in~$\theta$,
i.e., the subtuple~$\check\theta$ can be expressed in terms of derivatives of~$\hat\theta$,
the class~$\mathcal L'$ can be reparameterized in an analogous way.
Nevertheless, this is not so important in this case as in the previous one
since such constraints are not related to the existence of insignificant equivalence transformations.

\subsection{Revisited statement of group classification problem}\label{FP:sec:GroupClassProblem}

Loosely speaking, the group classification problem for the class~$\mathcal L$
is to classify the maximal Lie invariance algebras~$\mathfrak g_\theta$
of the systems~$\mathcal L_\theta$ from the class~$\mathcal L$ up to an appropriate equivalence
related to point transformations between such systems.
To solve the group classification problem for the class~$\mathcal L$ is to list all inequivalent values of~$\theta$
for each of which the algebra~$\mathfrak g_\theta$ is wider
than a Lie invariance algebra of~$\mathcal L_\theta$
that is in a certain sense common or general for all systems in the class~$\mathcal L$.
Thus, for a rigorous formulation of the group classification problem,
one should precisely specify the equivalence to be used, invariance algebras to be assumed ``common'' or ``general''
and in which sense these algebras are ``common'' or ``general''.

\looseness=1
There are two standard equivalences used in the course of group classification of the class~$\mathcal L$,
the $\mathcal G^\sim_{\mathcal L}$-equivalence and the $G^\sim_{\mathcal L}$-equivalence.
The classification with respect to the equivalence groupoid~$\mathcal G^\sim_{\mathcal L}$ is the classification
up to all point transformations within the class~$\mathcal L$ or, in other words,
up to the general point equivalence called the similarity of systems of differential equations.
The $\mathcal G^\sim_{\mathcal L}$-equivalence coincides with the equivalence with respect to the equivalence group~$G^\sim_{\mathcal L}$
if the class~$\mathcal L$ is semi-normalized
but for a general class of differential equations, the latter equivalence is weaker.
If the class~$\mathcal L$ is not semi-normalized, then one may try to partition it into semi-normalized subclasses
that are invariant under the action of~$\mathcal G^\sim_{\mathcal L}$
and to solve the group classification problem for each of these subclasses with respect to its own equivalence group.
If such a partition is not feasible, one should carry out the group classification of the class~$\mathcal L$ with respect to its equivalence group
and additionally indicate which $G^\sim_{\mathcal L}$-inequivalent cases of Lie-symmetry extensions
are $\mathcal G^\sim_{\mathcal L}$-equivalent.
Such equivalences are called additional, and are realized by admissible transformations of the class~$\mathcal L$
that do not belong to the action groupoid~$\mathcal G^{G^\sim_{\mathcal L}}$ of the equivalence group~$G^\sim_{\mathcal L}$.
One can also use the equivalence with respect to the generalized, the extended or the extended generalized equivalence groups
instead of the usual one~\cite{ivan2010a}.
In some classification problems, e.g.,
in partial preliminary group classification problems~\cite{bihl2012b}
or in group classifications of potential systems or potential equations~\cite{popo2005a},
the corresponding equivalence groups are replaced by their proper subgroups.

Likewise, there are two approaches to assigning which invariance algebras are regarded ``common'' or ``general''
for systems in the class~$\mathcal L$.
The traditional approach is to take the common part
$\mathfrak g^{\cap}_{\mathcal L}=\cap_{\theta\in\mathcal S} \mathfrak g_\theta$
of all the algebras~$\mathfrak g_\theta$ as basic for the further classification of its extensions~\cite{Ovsiannikov1982}.
It is called the kernel of maximal Lie invariance algebras of systems from the class~$\mathcal L$ or, briefly, the \emph{kernel algebra} of~$\mathcal L$.
Although this approach is applicable to a generic class of differential equations, there may be better alternatives for some classes.
In fact, the algebras~$\mathfrak g_\theta$, $\theta\in\mathcal S$,
may have more in common than merely the common part~$\mathfrak g^{\cap}_{\mathcal L}$.
Within the more general second approach,
one considers a family $\{\mathfrak g^{\rm unf}_\theta\subseteq\mathfrak g_\theta\mid\theta\in\mathcal S\}$
of \emph{uniform} subalgebras of the algebras~$\mathfrak g_\theta$, $\theta\in\mathcal S$.%
\footnote{%
For each $\theta\in\mathcal S$, the Lie algebra~$\mathfrak g_\theta$ is defined
by a homogeneous linear system of differential equations $\Theta^\theta[\mathrm v]=0$
in components of vector fields~$\mathrm v$ on the space with coordinates~$(x,u)$.
Here~$\Theta^\theta=\Theta^\theta[\mathrm v]$ is a tuple of homogeneous linear differential functions
in components of~$\mathrm v$ that is parameterized by derivatives of~$\theta$.
In other words, a vector field~$\mathrm v$ belongs to~$\mathfrak g_\theta$ if and only if $\Theta^\theta[\mathrm v]=0$.
We assume that the subalgebra~$\mathfrak g^{\rm unf}_\theta$ of~$\mathfrak g_\theta$ is defined
by an analogous system $\bar\Theta^\theta[\mathrm v]=0$, which contains the system $\Theta^\theta[\mathrm v]=0$ as a subsystem.
}
The uniformity of these subalgebras means
that their parameterization by~$\theta$ is consistent with the equivalence used
in the course of the group classification of the class~$\mathcal L$.
More specifically,
$(\pi_*\mathscr T)_*\mathfrak g^{\rm unf}_\theta=\mathfrak g^{\rm unf}_{\mathscr T_*\theta}$
for any $\theta\in\mathcal S$ and any $\mathscr T\in G^\sim_{\mathcal L}$ in the case of the $G^\sim_{\mathcal L}$-uniformity
and
$\Phi_*\mathfrak g^{\rm unf}_{\theta_1}=\mathfrak g^{\rm unf}_{\theta_2}$
for any $\mathcal T=(\theta_1,\Phi,\theta_2)\in\mathcal G^\sim_{\mathcal L}$ in the case of the $\mathcal G^\sim_{\mathcal L}$-uniformity.
The uniform subalgebras from the same family usually have a similar structure,
e.g., they are simultaneously abelian or their dimensions are the same or, more generally, they have the same arbitrariness of their parameters.
Note that, under classifying up to the $G^\sim_{\mathcal L}$-equivalence,
the first (traditional) approach can be considered as a particular case of the second approach,
where each $G^\sim_{\mathcal L}$-uniform subalgebra coincides with the kernel algebra of~$\mathcal L$,
$\mathfrak g^{\rm unf}_\theta=\mathfrak g^{\cap}_{\mathcal L}$ for any $\theta\in\mathcal S$. 
This is due to the fact that the kernel algebra~$\mathfrak g^\cap_{\mathcal L}$ is invariant under the $G^\sim_{\mathcal L}$-action,
\[
(\pi_*\mathscr T)_*\mathfrak g^\cap_{\mathcal L}=(\pi_*\mathscr T)_*\left(\cap_{\theta\in\mathcal S}\mathfrak g_\theta\right)=
\cap_{\theta\in\mathcal S}\mathfrak g_{\mathscr T_*\theta}=\cap_{\theta\in\mathcal S}\mathfrak g_{\theta}=\mathfrak g^\cap_{\mathcal L}
\quad\mbox{for any}\ \mathscr T\in G^\sim_{\mathcal L}.
\]

\begin{proposition}\label{pro:GroupoidUniformSubalgebras}
Given a family $\{\mathfrak g^{\rm unf}_\theta\subseteq\mathfrak g_\theta\mid\theta\in\mathcal S\}$
of $\mathcal G^\sim_{\mathcal L}$-uniform subalgebras, for each $\theta\in\mathcal S$,
the subalgebra~$\mathfrak g^{\rm unf}_\theta$ is invariant under the adjoint action
of the point symmetry group~$G_\theta$ of the system~$\mathcal L_\theta$ on~$\mathfrak g_\theta$
and, moreover, it is an ideal of~$\mathfrak g_\theta$.
\end{proposition}

\begin{proof}
We fix an arbitrary $\theta\in\mathcal S$ and an arbitrary $\Phi\in G_\theta$.
Then $\mathcal T:=(\theta,\Phi,\theta)$ belongs to the vertex group $\mathcal G_\theta\subset\mathcal G^\sim_{\mathcal L}$
and thus $\Phi_*\mathfrak g^{\rm unf}_\theta=\mathfrak g^{\rm unf}_\theta$.
In other words, the subalgebra~$\mathfrak g^{\rm unf}_\theta$ is invariant under the adjoint action
of the point symmetry group~$G_\theta$ of the system~$\mathcal L_\theta$ on~$\mathfrak g_\theta$.

To prove the second claim of the proposition, we follow the proof of~\cite[Corollary~25]{OpanasenkoBihloPopovych2020}.
Let $\mathrm v\in\mathfrak g_\theta$ and thus $\{\exp(\varepsilon\mathrm v)\}$,
where the parameter~$\varepsilon$ runs through a neighborhood of zero in~$\mathbb R$,
is a local one-parameter point symmetry group of the system~$\mathcal L_\theta$, $\{\exp(\varepsilon\mathrm v)\}\subseteq G_\theta$.
For any $\mathrm w\in\mathfrak g^{\rm unf}_\theta$ we have $(\exp(\varepsilon\mathrm v))_*\mathrm w\in\mathfrak g^{\rm unf}_\theta$.
Recall that
\[
[\mathrm v,\mathrm w]
=\mathscr L_{\mathrm v}\mathrm w
:=\frac{\mathrm d\mathrm w_\varepsilon}{\mathrm d\varepsilon}\Big|_{\varepsilon=0}
=\lim\limits_{\varepsilon\to0}\frac{\mathrm w_\varepsilon-\mathrm w_0}\varepsilon,
\quad\mbox{where}\quad
\mathrm w_\varepsilon\big|_p:=\big(\exp(-\varepsilon\mathrm v)_*\mathrm w\big)\big|_{\exp(\varepsilon\mathrm v)(p)}
\]
for any point $p$ in the space coordinatized by $(x,u)$.
The vector field~$\mathrm u_\varepsilon=(\mathrm w_\varepsilon-\mathrm w_0)/\varepsilon$
clearly belongs to the algebra~$\mathfrak g^{\rm unf}_\theta$ for any~$\varepsilon\ne0$,
and thus it satisfies the system defining this algebra, $\bar\Theta^\theta[\mathrm u_\varepsilon]=0$.
Recalling that components of $\mathrm u_\varepsilon$ are smooth in $(x,u,\varepsilon)$,
we use the continuity of~$\bar\Theta^\theta$ in components of vector fields to show
\[
\bar\Theta^\theta\big[[\mathrm v,\mathrm w]\big]=\bar\Theta^\theta\left[\,\lim\limits_{\varepsilon\to0}\mathrm u_\varepsilon\right]=
\lim\limits_{\varepsilon\to0}\bar\Theta^\theta[\mathrm u_\varepsilon]=0,
\]
implying that $[\mathrm v,\mathrm w]\in\mathfrak g^{\rm unf}_\theta$.
Therefore, $\mathfrak g^{\rm unf}_\theta$ is an ideal of~$\mathfrak g_\theta$.
\end{proof}

In practice, uniform subalgebras are either known a priori from the very form of the systems~$\mathcal L_\theta$
or can be easily computed at the preliminary step of group classification of~$\mathcal L$.
From a known family $\{\mathfrak g^{\rm unf}_\theta\subseteq\mathfrak g_\theta\mid\theta\in\mathcal S\}$
of $G^\sim_{\mathcal L}$-uniform subalgebras one can easily construct a family of $G^\sim_{\mathcal L}$-uniform subalgebras
each of which contains the kernel algebra~$\mathfrak g^\cap_{\mathcal L}$.

\begin{proposition}\label{pro:AttachingKernelAlgebraToUniformSubAlgebras}
Given a family $\{\mathfrak g^{\rm unf}_\theta\subseteq\mathfrak g_\theta\mid\theta\in\mathcal S\}$
of $G^\sim_{\mathcal L}$-uniform subalgebras,
for an arbitrary $\theta\in\mathcal S$, let $\bar{\mathfrak g}^{\rm unf}_\theta$ be the minimal subalgebra of~$\mathfrak g_\theta$
that contains both~$\mathfrak g^\cap_{\mathcal L}$ and~$\mathfrak g^{\rm unf}_\theta$.
Then $\{\bar{\mathfrak g}^{\rm unf}_\theta\mid\theta\in\mathcal S\}$
is a family of $G^\sim_{\mathcal L}$-uniform subalgebras of the algebras~$\mathfrak g_\theta$,~$\theta\in\mathcal S$.
\end{proposition}

\begin{proof}
We fix an arbitrary $\mathscr T\in G^\sim_{\mathcal L}$ and an arbitrary $\theta\in\mathcal S$
and recall that the kernel algebra~$\mathfrak g^\cap_{\mathcal L}$ is invariant under the $G^\sim_{\mathcal L}$-action,
$(\pi_*\mathscr T)_*\mathfrak g^\cap_{\mathcal L}=\mathfrak g^\cap_{\mathcal L}$.
Since in addition $(\pi_*\mathscr T)_*\mathfrak g^{\rm unf}_\theta=\mathfrak g^{\rm unf}_{\mathscr T_*\theta}$,
we have that $(\pi_*\mathscr T)_*\bar{\mathfrak g}^{\rm unf}_{\theta}\supseteq\mathfrak g^{\rm unf}_{\mathscr T_*\theta}$ and
$(\pi_*\mathscr T)_*\bar{\mathfrak g}^{\rm unf}_{\theta}\supseteq\mathfrak g^\cap_{\mathcal L}$.
The fact that $\bar{\mathfrak g}^{\rm unf}_\theta$ is a subalgebra of $\mathfrak g_\theta$
directly implies that $(\pi_*\mathscr T)_*\bar{\mathfrak g}^{\rm unf}_\theta$ is a subalgebra of $(\pi_*\mathscr T)_*\mathfrak g_\theta=
\mathfrak g_{\mathscr T_*\theta}$.
Therefore, $(\pi_*\mathscr T)_*\bar{\mathfrak g}^{\rm unf}_\theta\supseteq\bar{\mathfrak g}^{\rm unf}_{\mathscr T_*\theta}$.
Analogously, we derive $(\pi_*\mathscr T^{-1})_*\bar{\mathfrak g}^{\rm unf}_{\mathscr T_*\theta}\supseteq\bar{\mathfrak g}^{\rm unf}_\theta$
and thus $\bar{\mathfrak g}^{\rm unf}_{\mathscr T_*\theta}\supseteq(\pi_*\mathscr T)_*\bar{\mathfrak g}^{\rm unf}_\theta$.
Hence $(\pi_*\mathscr T)_*\bar{\mathfrak g}^{\rm unf}_\theta=\bar{\mathfrak g}^{\rm unf}_{\mathscr T_*\theta}$.
\end{proof}

On the other hand, attaching the kernel algebra~$\mathfrak g^\cap_{\mathcal L}$ to each of the chosen uniform subalgebras may backfire.
First of all, Proposition~\ref{pro:AttachingKernelAlgebraToUniformSubAlgebras} does in general not hold
in the case of the $\mathcal G^\sim_{\mathcal L}$-uniformity
since the kernel algebra~$\mathfrak g^\cap_{\mathcal L}$ is not necessarily $\mathcal G^\sim_{\mathcal L}$-invariant,
cf.~\cite[Example~1]{CardosoBihloPopovych2011}.
Moreover, the family $\{\bar{\mathfrak g}^{\rm unf}_\theta\mid\theta\in\mathcal S\}$
may not have important properties of the family $\{\mathfrak g^{\rm unf}_\theta\mid\theta\in\mathcal S\}$.
In a best-case scenario,
the class~$\mathcal L$ is disjointedly semi-normalized
with respect to a family of subgroups~$H_\theta$ of the complete point symmetry groups~$G_\theta$,
and we can choose the family of Lie algebras of the subgroups~$H_\theta$
as a family of uniform subalgebras~$\mathfrak g^{\rm unf}_\theta$ of the algebras~$\mathfrak g_\theta$,
see~\cite[Definition~1]{KurujyibwamiBasarabHorwathPopovych2018}.
In this case, for each $\theta\in\mathcal S$,
the subalgebra~$\mathfrak g^{\rm unf}_\theta$ is an ideal of~$\mathfrak g_\theta$
and has a complement subspace in~$\mathfrak g_\theta$ that is a subalgebra of~$\mathfrak g_\theta$,
called the \emph{essential Lie invariance algebra}~$\mathfrak g^{\rm ess}_\theta$ of~$\mathcal L_\theta$,
$\mathfrak g_\theta=\mathfrak g^{\rm ess}_\theta\lsemioplus\mathfrak g^{\rm unf}_\theta$,
and moreover $\mathfrak g^{\rm ess}_\theta=\mathfrak g_\theta\cap\pi_*\mathfrak g^\sim_{\mathcal L}$.
Therefore, the group classification of~$\mathcal L$ reduces
to the classification of appropriate subalgebras of~$\mathfrak g^\sim_{\mathcal L}$.
Even if the family of uniform subalgebras of~$\mathfrak g_\theta$ cannot be chosen in the above way,
one may want to ensure at least some of the aforementioned properties.

To summarize, the statement of group classification problem can be revisited as follows.

\emph{To carry out the group classification of a class~$\mathcal L$ of systems of differential equations with respect to
its equivalence groupoid~$\mathcal G^\sim_{\mathcal L}$ (resp.\ its equivalence group~$G^\sim_{\mathcal L}$)
is to list all $\mathcal G^\sim_{\mathcal L}$-inequivalent (resp.\ $G^\sim_{\mathcal L}$-inequivalent) values
of its arbitrary-element tuple $\theta\in\mathcal S$
such that the maximal Lie invariance algebras~$\mathfrak g_\theta$ of the systems~$\mathcal L_\theta$
are wider than the corresponding elements of the chosen family
$\{\mathfrak g^{\rm unf}_\theta\mid\theta\in\mathcal S\}$ of $\mathcal G^\sim_{\mathcal L}$-uniform
(resp.\ $G^\sim_{\mathcal L}$-uniform) subalgebras of the algebras~$\mathfrak g_\theta$.}

The problem can be reformulated in the case when the essential Lie invariance algebra~$\mathfrak g^{\rm ess}_\theta$
is defined for each $\theta\in\mathcal S$ in a canonical way.

\emph{Let~$\mathcal L$ be a class of systems of differential equations,
where for any value of its arbitrary-element tuple $\theta\in\mathcal S$
one can canonically define the essential subalgebra~$\mathfrak g^{\rm ess}_\theta$ of~$\mathfrak g_\theta$.
The group classification of this class with respect to
its equivalence groupoid~$\mathcal G^\sim_{\mathcal L}$ (resp.\ its equivalence group~$G^\sim_{\mathcal L}$)
reduces to listing all $\mathcal G^\sim_{\mathcal L}$-inequivalent (resp.\ $G^\sim_{\mathcal L}$-inequivalent) values $\theta\in\mathcal S$
such that the essential Lie invariance algebras~$\mathfrak g^{\rm ess}_\theta$ of the systems~$\mathcal L_\theta$
are wider than the kernel $\mathfrak g^{\cap\rm ess}_{\mathcal L}=\cap_{\theta\in\mathcal S}\mathfrak g^{\rm ess}_\theta$
of the essential Lie invariance algebras of the class~$\mathcal L$.}

As an example, consider a class~$\mathcal L$ of linear homogeneous systems of differential equations.
Any system~$\mathcal L_\theta\in\mathcal L$ admits, in view of the linear superposition of solutions,
the point symmetry transformations of the form $\tilde x=x$, $\tilde u=u+f(x)$,
where~$f(x)=\left(f^1(x),\dots,f^m(x)\right)$ is an arbitrary solution of~$\mathcal L_\theta$.
This reflects on the structure of both the complete point symmetry group~$G_\theta$
and the maximal Lie invariance algebra~$\mathfrak g_\theta$ thereof.
In fact, if the admissible transformations of the class~$\mathcal L$ are fiber-preserving and their $u$-components are affine with respect to~$u$,%
\footnote{%
These properties are quite common for classes of linear homogeneous systems of differential equations.
At the same time, see~\cite[Section~4]{BoykoLokazyukPopovych2021} for an example of a class of such systems,
where admissible transformations are not fiber-preserving,
and for any value of the corresponding arbitrary-element tuple~$\theta$,
$\mathfrak g^{\rm lin}_\theta$ is not an ideal of~$\mathfrak g_\theta$
and any complement subspace to~$\mathfrak g^{\rm lin}_\theta$ in~$\mathfrak g_\theta$ contains 
vector fields of a different form although such a subspace can be chosen to be a subalgebra of~$\mathfrak g_\theta$. 
This is the class~$\mathcal L_0$ in the notation of~\cite{BoykoLokazyukPopovych2021}, 
which consists of the linear homogeneous systems of second-order ordinary differential equations that are similar 
with respect to point transformations to the elementary (free particle) system. 
}
then the maximal Lie invariance algebra~$\mathfrak g_\theta$ of the system~$\mathcal L_\theta$
can be represented as $\mathfrak g_\theta:=\mathfrak g^{\rm ess}_\theta\lsemioplus\mathfrak g^{\rm lin}_\theta$.~Here
\begin{gather*}
\mathfrak g^{\rm lin}_\theta=\Bigg\{\sum_{i=1}^m f^i(x)\p_{u^i}\ \Big|\ f \text{ is a solution of }\mathcal L_\theta\Bigg\},\\
\mathfrak g^{\rm ess}_\theta=\Bigg\{Q=\sum_{a=1}^n\xi^a(x)\p_{x_a}+\sum_{i,j=1}^m\eta^{ij}(x)u^j\p_{u^i}\ \Big|\ Q\in\mathfrak g_\theta\Bigg\}
\end{gather*}
are an ideal of~$\mathfrak g_\theta$ corresponding to the linear superposition of solutions
and the \emph{essential} subalgebra of~$\mathfrak g_\theta$, respectively.
The family $\{\mathfrak g^{\rm lin}_{\theta}\mid\theta\in\mathcal S\}$ is taken here
to be a family of uniform subalgebras of the algebras~$\mathfrak g_\theta$,
and thus it is natural to classify extensions
not of the algebra $\mathfrak g^{\rm \cap}_{\mathcal L}$ but of the algebra~$\mathfrak g^{\cap\rm ess}_{\mathcal L}$.
If additionally the class~$\mathcal L$ is disjointedly semi-normalized with respect to the linear superposition of solutions,
then $\mathfrak g^{\rm ess}_\theta=\mathfrak g_\theta\cap\pi_*\mathfrak g^\sim_{\mathcal L}$ for any $\theta\in\mathcal S$.
Note that every system~$\mathcal L_\theta\in\mathcal L$ admits the vector field $I=\sum_{i=1}^m u^i\p_{u^i}$ as its Lie symmetry,
i.e., $I\in\mathfrak g^{\cap\rm ess}_{\mathcal L}\subseteq\mathfrak g^\cap_{\mathcal L}$.%
\footnote{%
Under the above conditions, $\mathfrak g^\cap_{\mathcal L}=\mathfrak g^{\cap\rm ess}_{\mathcal L}\lsemioplus\mathfrak g^{\cap\rm lin}_{\mathcal L}$,
where $\mathfrak g^{\cap\rm lin}_{\mathcal L}=\{\sum_{i=1}^m f^i(x)\p_{u^i}\mid
f \text{ is a solution of }\mathcal L_\theta\ \text{for any}$ $\theta\in\mathcal S\}$.
In other words, $\mathfrak g^{\cap\rm lin}_{\mathcal L}\ne\{0\}$ if and only if
the systems from the class~$\mathcal L$ have a common nonzero solution.
This is the case, e.g., for the class~$\bar{\mathcal K}$ of Kolmogorov equations
whose set of common solutions consists of the constant functions, and hence  $\mathfrak g^{\cap\rm lin}_{\bar{\mathcal K}}=\langle\p_u\rangle$.
}
Therefore, it may be tempting to consider the family of the wider uniform subalgebras~$\bar{\mathfrak g}^{\rm lin}_{\theta}$
of the algebras~$\mathfrak g_\theta$ than the subalgebras~$\mathfrak g^{\rm lin}_{\theta}$
that are obtained by attaching~$I$ to~$\mathfrak g^{\rm lin}_{\theta}$ for each $\theta\in\mathcal S$,
$\bar{\mathfrak g}^{\rm lin}_{\theta}=\langle I\rangle\lsemioplus\mathfrak g^{\rm lin}_{\theta}$.
Although the subalgebra $\bar{\mathfrak g}^{\rm lin}_{\theta}$ may share, with~$\mathfrak g^{\rm lin}_{\theta}$,
the property to be ideal of~$\mathfrak g_\theta$,
it may not have a complement subspace in~$\mathfrak g_\theta$ that is closed with respect to the Lie bracket of vector fields.
This leads to artificial complications in applying the algebraic method of group classification.
As a result, the family $\{\mathfrak g^{\rm lin}_{\theta}\mid\theta\in\mathcal S\}$
is a natural and convenient choice for a family of uniform subalgebras of the algebras~$\mathfrak g_\theta$
when solving the group classification problem for the class~$\mathcal L$.

\subsection{Extension of the mapping method of group classification}\label{FP:sec:MapMethod}

We develop the mapping method of group classification proposed in~\cite{VaneevaPopovychSophocleous2009}
to solve the group classification problem for the classes~$\bar{\mathcal K}$ and~$\bar{\mathcal F}$ and their subclasses.

Let ${\mathcal L}$ and~$\tilde{\mathcal L}$ be related classes of differential equations
and let the group classification problem for the class~$\mathcal L$ be already solved.
The underlying idea of the mapping method is to derive a group classification list for the class~$\tilde{\mathcal L}$ from
the known group classification list for the class~$\mathcal L$.
A~natural choice for the relation between classes of differential equations
within the framework of the mapping method is their weak similarity
although there are also other options~\cite{VaneevaBihloPopovych2020};
see also~\cite{VaneevaPopovychSophocleous2009} and Definition~\ref{FP:def:Similarity} below
for the more restricted notion of (usual) similarity of classes of differential equations.

\begin{definition}\label{FP:def:WeaklySimilarClasses}
Classes~$\mathcal L$ and~$\tilde{\mathcal L}$ of systems of differential equations
with the same number of independent variables, the same number of dependent variables,
the same number of equations in systems and the same order of systems
are called \emph{weakly similar} if for each system~$\mathcal L_\theta$ from the class~$\mathcal L$
there exists a system~$\tilde{\mathcal L}_{\tilde\theta}$ from the class~$\tilde{\mathcal L}$
that is similar to~$\mathcal L_\theta$, and vice versa.
\end{definition}

Weak similarity is an equivalence relation for classes of differential equations.

The mapping method straightforwardly works for weakly similar classes of differential equations
in the case of group classification up to the general point equivalence
if one looks for extensions of the subalgebras of~$\mathfrak g_\theta$ that are chosen as uniform with respect to this equivalence.
Let~${\rm CL}$ be a known classification list for the class~$\mathcal L$ up to the general point equivalence
with respect to a family $\{\mathfrak g^{\rm unf}_\theta\mid\theta\in\mathcal S\}$
of $\mathcal G^\sim_{\mathcal L}$-uniform subalgebras of~$\mathfrak g_\theta$.
The list~${\rm CL}$ is a complete set of the $\mathcal G^\sim_{\mathcal L}$-inequivalent values $\theta\in\mathcal S$
such that \smash{$\mathfrak g_\theta\varsupsetneq\mathfrak g^{\rm unf}_\theta$}.
Each $\theta\in{\rm CL}$ should be
supplemented with the corresponding maximal Lie invariance algebra~$\mathfrak g_\theta$.
Then for any \smash{$\tilde\theta\in\tilde{\mathcal S}$}, we define the candidate for the uniform subalgebra of~$\mathfrak g_{\tilde\theta}$ as
\smash{$\mathfrak g^{\rm unf}_{\tilde\theta}:=\Phi_*\mathfrak g^{\rm unf}_\theta$},%
\footnote{%
The tildes over~$\theta$ in $\mathfrak g_{\tilde\theta}$ and \smash{$\mathfrak g^{\rm unf}_{\tilde\theta}$} indicate that
these are
the maximal Lie invariance algebra of the system~$\tilde{\mathcal L}_{\tilde\theta}$
and a uniform subalgebra of this algebra, respectively.
}
where the value $\theta\in\mathcal S$ and the point transformation~$\Phi$ are such that \smash{$\tilde{\mathcal L}_{\tilde\theta}=\Phi_*\mathcal L_\theta$}.%
\footnote{%
At the same time, the kernel algebras~$\mathfrak g^\cap_{\mathcal L}$ and~$\mathfrak g^\cap_{\tilde{\mathcal L}}$
of the classes~$\mathcal L$ and~$\tilde{\mathcal L}$ are not directly related in the general case,
and thus the algebra~$\mathfrak g^\cap_{\tilde{\mathcal L}}$ should be computed independently.
}
This candidate is well-defined.
Indeed, if there exist another $\theta'\in\mathcal S$ and a point transformation~$\Phi'$
such that $\tilde{\mathcal L}_{\tilde\theta}=\Phi'_*\mathcal L_{\theta'}$, then
\[
\Phi'_*\mathfrak g^{\rm unf}_{\theta'}=(\Phi\circ\Phi^{-1}\circ\Phi')_*\mathfrak g^{\rm unf}_{\theta'}=
\Phi_*\big((\Phi^{-1}\circ\Phi')_*\mathfrak g^{\rm unf}_{\theta'}\big)=\Phi_*\mathfrak g^{\rm unf}_{\theta}=\mathfrak g^{\rm unf}_{\tilde\theta}
\]
since $(\theta',\Phi^{-1}\circ\Phi',\theta)\in\mathcal G^\sim_{\mathcal L}$ and thus
$(\Phi^{-1}\circ\Phi')_*\mathfrak g^{\rm unf}_{\theta'}=\mathfrak g^{\rm unf}_{\theta}$
in view of the $\mathcal G^\sim_{\mathcal L}$-uniformity of the family $\{\mathfrak g^{\rm unf}_{\theta}\mid\theta\in\mathcal S\}$.
Moreover, the family $\{\mathfrak g^{\rm unf}_{\tilde\theta}\mid\tilde\theta\in\tilde{\mathcal S}\}$ is $\mathcal G^\sim_{\tilde{\mathcal L}}$-uniform.
Indeed, for any \smash{$(\tilde\theta_1,\tilde\Phi,\tilde\theta_2)\in\mathcal G^\sim_{\tilde{\mathcal L}}$}
there exist $\theta_1,\theta_2\in\mathcal S$ and point transformations $\Phi_1$ and $\Phi_2$ such that
$(\Phi_i)_*\mathcal L_{\theta_i}=\tilde{\mathcal L}_{\tilde\theta_i}$, $i=1,2$.
Then the triple $(\theta_1,\Phi,\theta_2)$ with $\Phi=\Phi_2^{-1}\circ\tilde\Phi\circ\Phi_1$
belongs to~$\mathcal G^\sim_{\mathcal L}$.
Representing $\tilde\Phi=\Phi_2\circ\Phi\circ\Phi_1^{-1}$,
in view of the $\mathcal G^\sim_{\mathcal L}$-uniformity of the family $\{\mathfrak g^{\rm unf}_{\theta}\mid\theta\in\mathcal S\}$ we derive
\looseness=-1
\[
\tilde\Phi_*\mathfrak g^{\rm unf}_{\tilde\theta_1}=(\Phi_2\circ\Phi\circ\Phi_1^{-1})_*\big((\Phi_1)_*\mathfrak g^{\rm unf}_{\theta_1}\big)=
(\Phi_2\circ\Phi)_*\mathfrak g^{\rm unf}_{\theta_1}=(\Phi_2)_*\mathfrak g^{\rm unf}_{\theta_2}=\mathfrak g^{\rm unf}_{\tilde\theta_2}.
\]
We call the family \smash{$\{\mathfrak g^{\rm unf}_{\tilde\theta}\mid\tilde\theta\in\tilde{\mathcal S}\}$}
to be associated with the family $\{\mathfrak g^{\rm unf}_\theta\mid\theta\in\mathcal S\}$.
It is obvious that the \smash{$\mathcal G^\sim_{\mathcal L}$}- and \smash{$\mathcal G^\sim_{\tilde{\mathcal L}}$}-equivalences
are consistent with respect to the weak similarity of~$\mathcal L$ and~$\tilde{\mathcal L}$.%
\footnote{%
Given a binary relation~$\mathfrak R$ over sets~$M$ and~$M'$,
equivalence relations~$\mathfrak E$ and~$\mathfrak E'$ defined on~$M$ and~$M'$, respectively,
are called \emph{consistent with respect to~$\mathfrak R$}
if for any $m_1,m_2\in M$ and any $m'_1,m'_2\in M'$ with $m_1\mathfrak Rm'_1$ and $m_2\mathfrak Rm'_2$,
the equivalence $m_1\sim_{\mathfrak E}m_2$ implies $m'_1\sim_{\mathfrak E'}m'_2$, and vice versa.
}
Thus, given a complete list~$\rm CL$ of \smash{$\mathcal G^\sim_{\mathcal L}$}-inequivalent Lie-symmetry extensions
of \smash{$\mathcal G^\sim_{\mathcal L}$}-uniform Lie invariance algebras~\smash{$\mathfrak g^{\rm unf}_\theta$} in the class~$\mathcal L$,
an analogous list~$\widetilde{\rm CL}$ for the class~$\tilde{\mathcal L}$ is obtained by taking,
for each $\theta\in{\rm CL}$, a single value $\tilde\theta\in\tilde{\mathcal S}$
such that the system~$\tilde{\mathcal L}_{\tilde\theta}$ is similar to~$\mathcal L_\theta$.
The list $\widetilde{\rm CL}$ provides \smash{$\mathcal G^\sim_{\tilde{\mathcal L}}$}-inequivalent Lie-symmetry extensions
of the \smash{$\mathcal G^\sim_{\tilde{\mathcal L}}$}-uniform Lie invariance algebras~\smash{$\mathfrak g^{\rm unf}_{\tilde\theta}$}
in the class~$\tilde{\mathcal L}$.
The list~$\widetilde{\rm CL}$ contains only \smash{$\mathcal G^\sim_{\tilde{\mathcal L}}$}-inequivalent values of~$\tilde{\theta}$
since otherwise the list~$\rm CL$ would contain $\mathcal G^\sim_{\mathcal L}$-equivalent values of~$\theta$.
The list~\smash{$\widetilde{\rm CL}$} is exhaustive because if there exists $\tilde\theta_0\in\tilde{\mathcal S}$
that is \smash{$\mathcal G^\sim_{\tilde{\mathcal L}}$}-inequivalent to any of the values in~\smash{$\widetilde{\rm CL}$},
then its associated value $\theta_0\in\mathcal S$
would be $\mathcal G^\sim_{\mathcal L}$-inequivalent to any value in the list~$\rm CL$, which is absurd.

The transition between the group classifications of the weakly similar classes~$\mathcal L$ and~$\tilde{\mathcal L}$
with respect to the corresponding equivalence groups,
except for some special cases where the direct relation between these group classifications exists,
is more challenging.
Firstly, one can require strong normalization properties for such a transition between weakly similar classes.

\begin{proposition}\label{pro:GroupClassificationOfSeminormClassesByMappingMethod}
Let $\mathcal L$ and $\tilde{\mathcal L}$ be weakly similar semi-normalized classes of differential equations,
and there is a family $\{\mathfrak g^{\rm unf}_\theta\mid\theta\in\mathcal S\}$
of $G^\sim_{\mathcal L}$-uniform subalgebras of the algebras~$\mathfrak g_\theta$
that is invariant with respect to the fundamental groupoid~$\mathcal G^{\rm f}_{\mathcal L}$ of~$\mathcal L$.
Then there is an associated family \smash{$\{\mathfrak g^{\rm unf}_{\tilde\theta}\mid\tilde\theta\in\tilde{\mathcal S}\}$}
of \smash{$G^\sim_{\tilde{\mathcal L}}$}-uniform subalgebras of the algebras~$\mathfrak g_{\tilde\theta}$,
and the group classifications of~$\mathcal L$ and~$\tilde{\mathcal L}$ with respect to the corresponding equivalence groups
and the above subalgebra families are equivalent.
\end{proposition}

\begin{proof}
Consider arbitrary $\theta_1,\theta_2\in\mathcal S$ and arbitrary $\tilde\theta_1,\tilde\theta_2\in\tilde{\mathcal S}$
such that the system~\smash{$\tilde{\mathcal L}_{\tilde\theta_i}$} is similar to~$\mathcal L_{\theta_i}$, $i=1,2$.
Suppose that the systems~$\mathcal L_{\theta_1}$ and~$\mathcal L_{\theta_2}$ are $G^\sim_{\mathcal L}$-equivalent.
Then the systems~\smash{$\tilde{\mathcal L}_{\tilde\theta_1}$ and~$\tilde{\mathcal L}_{\tilde\theta_2}$} are similar to each other,
and thus they are \smash{$G^\sim_{\tilde{\mathcal L}}$}-equivalent since the class~$\tilde{\mathcal L}$ is semi-normalized.
Switching the roles of the classes~$\mathcal L$ and $\tilde{\mathcal L}$ shows
that the systems~$\mathcal L_{\theta_1}$ and~$\mathcal L_{\theta_2}$ are $G^\sim_{\mathcal L}$-equivalent
if and only if the systems~\smash{$\tilde{\mathcal L}_{\tilde\theta_1}$} and~\smash{$\tilde{\mathcal L}_{\tilde\theta_2}$}
are \smash{$G^\sim_{\tilde{\mathcal L}}$}-equivalent,
i.e., the $G^\sim_{\mathcal L}$- and \smash{$G^\sim_{\tilde{\mathcal L}}$}-equivalences are consistent
with respect to the weak similarity of~$\mathcal L$ and~$\tilde{\mathcal L}$.

Since the family $\{\mathfrak g^{\rm unf}_\theta\mid\theta\in\mathcal S\}$ of $G^\sim_{\mathcal L}$-uniform subalgebras of~$\mathfrak g_\theta$
is in addition invariant with respect to the groupoid~$\mathcal G^{\rm f}_{\mathcal L}$,
then it is a family of $\mathcal G^\sim_{\mathcal L}$-uniform subalgebras of~$\mathfrak g_\theta$ as well
in view of the semi-normalization of~$\mathcal L$,
which means the representation $\mathcal G^\sim_{\mathcal L}=\mathcal G^{G^\sim_{\mathcal L}}\star\mathcal G^{\rm f}_{\mathcal L}$.
Therefore, there is a well-defined family \smash{$\{\mathfrak g^{\rm unf}_{\tilde\theta}\mid\tilde\theta\in\tilde{\mathcal S}\}$}
of \smash{$\mathcal G^\sim_{\tilde{\mathcal L}}$}-uniform subalgebras of~$\mathfrak g_{\tilde\theta}$
that is associated with the family $\{\mathfrak g^{\rm unf}_\theta\mid\theta\in\mathcal S\}$
via the weak similarity of the classes~$\mathcal L$ and~$\tilde{\mathcal L}$.
Any $\mathcal G^\sim_{\tilde{\mathcal L}}$-uniform subalgebra is $G^\sim_{\tilde{\mathcal L}}$-uniform.
\end{proof}

Proposition~\ref{pro:GroupClassificationOfSeminormClassesByMappingMethod}
is particularly important when the class~$\tilde{\mathcal L}$ is a subclass of the class~$\mathcal L$.
In this case, we can reduce the group classification of the class~$\mathcal L$ to that of its subclass~$\tilde{\mathcal L}$.

\begin{corollary}
Group classifications of weakly similar normalized classes~$\mathcal L$ and~$\tilde{\mathcal L}$
with respect to associated families of $G^\sim_{\mathcal L}$- and $G^\sim_{\tilde{\mathcal L}}$-uniform subalgebras
of the algebras~$\mathfrak g_\theta$ and~$\mathfrak g_{\tilde\theta}$, respectively, are equivalent.
\end{corollary}

Alternatively, instead of considering classes with strong normalization properties,
one can require stronger similarity of classes of differential equations
to ensure the equivalence of their group classifications with respect to the corresponding equivalence groups.
More precisely, we consider similar classes that are related by a fixed point transformation.
Let $\pi_p$ be the projection from the space with coordinates~$(x,u_{(r)},\theta)$
on the space with coordinates $(x,u_{(p)})$, $p=0,\dots,r$,
and thus $\pi_0:=\pi$.

\begin{definition}\label{FP:def:Similarity}
A class $\mathcal L$ of differential equations is called \emph{similar} to a class $\tilde{\mathcal L}$
if these classes have the same number of independent variables, the same number of dependent variables,
the same number of equations in systems and the same order of systems,
and there exists a point transformation $\Psi\colon(x,u_{(r)},\theta)\mapsto(\tilde x,\tilde u_{(r)},\tilde\theta)$
that is projectable on the space of $(x,u_{(p)})$ for any $0\leqslant p\leqslant r$,
whereas $(\pi_p)_*\Psi$ is the $p$th order prolongation of $\pi_*\Psi$,
$\Psi_* \mathcal{S}=\tilde{\mathcal S}$ and $\Psi_*\mathcal L_\theta=\tilde{\mathcal L}_{\Psi_*\theta}$
for any $\theta\in\mathcal S$.
The transformation~$\Psi$ is called a \emph{similarity transformation} between classes~$\mathcal L$ and~$\tilde{\mathcal L}$.
\end{definition}

Since the similarity of classes of differential equations is an equivalence relation,
the classes~$\mathcal L$ and~$\tilde{\mathcal L}$ are said to be \emph{similar}.

\begin{proposition}\label{pro:GroupClassificationOfSimilarClassesByMappingMethod}
Let $\mathcal L$ and $\tilde{\mathcal L}$ be similar classes of differential equations.
Then their group classifications with respect to the corresponding equivalence groups are equivalent.
\end{proposition}

\begin{proof}
Let $\Psi$ be a similarity transformation between the classes~$\mathcal L$ and~$\tilde{\mathcal L}$.
Then by construction,
$\Psi_* G^\sim_{\mathcal L}=\{\Psi\circ \mathscr T\circ\Psi^{-1}\mid \mathscr T\in G^\sim_{\mathcal L}\}$
is the equivalence group~$G^\sim_{\tilde{\mathcal L}}$ of the class~$\tilde{\mathcal L}$.
This implies that the $G^\sim_{\mathcal L}$- and $G^\sim_{\tilde{\mathcal L}}$-equivalences are consistent
with respect to the similarity of $\mathcal L$ and $\tilde{\mathcal L}$.

The kernel algebra~$\mathfrak g^\cap_{\mathcal L}$ of the class~$\mathcal L$ is mapped by~$\Psi$ onto
the kernel algebra of the class~$\tilde{\mathcal L}$,
\smash{$\mathfrak g^\cap_{\tilde{\mathcal L}}=(\pi_*\Psi)_*\mathfrak g^\cap_{\mathcal L}$}.
Indeed, if $\mathrm v$ is a Lie-symmetry vector field of any system $\mathcal L_\theta\in\mathcal L$,
then $(\pi_*\Psi)_*\mathrm v$ is a Lie-symmetry vector field of the system $\tilde{\mathcal L}_{\Psi_*\theta}\in\tilde{\mathcal L}$,
and $\Psi_*\theta$ runs through~$\tilde{\mathcal S}$ when $\theta$ runs through~$\mathcal S$,
i.e., $\mathfrak g^\cap_{\tilde{\mathcal L}}\supset(\pi_*\Psi)_*\mathfrak g^\cap_{\mathcal L}$. Analogously,
$(\pi_*\Psi^{-1})_*\mathfrak g^\cap_{\tilde{\mathcal L}}\subset\mathfrak g^\cap_{\mathcal L}$, and thus
$\mathfrak g^\cap_{\tilde{\mathcal L}}=(\pi_*\Psi)_*\mathfrak g^\cap_{\mathcal L}$.

If $\{\mathfrak g^{\rm unf}_\theta\mid\theta\in\mathcal S\}$
is a family of $G^\sim_{\mathcal L}$-uniform subalgebras of the algebras~$\mathfrak g_\theta$,
then one can define a family \smash{$\{\mathfrak g^{\rm unf}_{\tilde\theta}\mid\tilde\theta\in\tilde{\mathcal S}\}$}
of \smash{$G^\sim_{\tilde{\mathcal L}}$}-uniform subalgebras of the algebras~$\mathfrak g_{\tilde\theta}$
as \smash{$\mathfrak g^{\rm unf}_{\tilde\theta}:=(\pi_*\Psi)_*\mathfrak g^{\rm unf}_{\theta}$}, where $\tilde\theta=\Psi_*\theta$.
For any \smash{$\tilde{\mathscr T}\in G^\sim_{\tilde{\mathcal L}}$},
there exists a unique transformation $\mathscr T\in G^\sim_{\mathcal L}$ such that $\tilde{\mathscr T}\circ\Psi=\Psi\circ\mathscr T$,
which is obviously defined by $\mathscr T:=\Psi^{-1}\circ\tilde{\mathscr T}\circ\Psi$.
The \smash{$G^\sim_{\tilde{\mathcal L}}$}-uniformity of the family \smash{$\{\mathfrak g^{\rm unf}_{\tilde\theta}\mid\tilde\theta\in\tilde{\mathcal S}\}$}
follows from the $G^\sim_{\mathcal L}$-uniformity of the family $\{\mathfrak g^{\rm unf}_{\theta}\mid\theta\in\mathcal S\}$,
\begin{align*}
(\pi_*\tilde{\mathscr T})_*\mathfrak g^{\rm unf}_{\tilde\theta}
={}&(\pi_*\tilde{\mathscr T})_*\big((\pi_*\Psi)_*\mathfrak g^{\rm unf}_{\theta}\big)
=\big((\pi_*\tilde{\mathscr T})\circ(\pi_*\Psi)\big)_*\mathfrak g^{\rm unf}_{\theta}
=\big(\pi_*(\tilde{\mathscr T}\circ\Psi)\big)_*\mathfrak g^{\rm unf}_{\theta}
\\
={}&\big(\pi_*(\Psi\circ\mathscr T)\big)_*\mathfrak g^{\rm unf}_{\theta}
=(\pi_*\Psi)_*\big((\pi_*\mathscr T)_*\mathfrak g^{\rm unf}_{\theta}\big)
=(\pi_*\Psi)_*\mathfrak g^{\rm unf}_{\mathscr T_*\theta}
=\mathfrak g^{\rm unf}_{\tilde{\mathscr T}_*\tilde\theta}.
\end{align*}
The last equality follows from the identity
$\tilde{\mathscr T}_*\tilde\theta=\tilde{\mathscr T}_*(\Psi_*\theta)
=(\tilde{\mathscr T}\circ\Psi)_*\theta=(\Psi\circ\mathscr T)_*\theta=\Psi_*(\mathscr T_*\theta)$.
\end{proof}

Lastly, when neither normalization properties of weakly similar classes nor their specific similarity are strong enough to satisfy
the conditions of Propositions~\ref{pro:GroupClassificationOfSeminormClassesByMappingMethod}
or~\ref{pro:GroupClassificationOfSimilarClassesByMappingMethod},
one can still ensure the consistency of the equivalences with respect to the corresponding equivalence groups of such classes
in certain special cases~\cite[Propositions~4 and~5]{VaneevaPopovychSophocleous2009}.
For general weakly similar classes~$\mathcal L$ and~$\tilde{\mathcal L}$, one needs to check
the consistence of the $G^\sim_{\mathcal L}$- and $G^\sim_{\tilde{\mathcal L}}$-equivalences by a direct computation.
Thus, the mapping method was successfully applied to weakly similar classes of (1+1)-dimensional reaction--diffusion equations
in~\cite{VaneevaPopovychSophocleous2009} to show the equivalence of their group classifications.
These classes are related by a family of point transformations parameterized by the arbitrary-element tuple of the source class,
and  there is a closed expression of the arbitrary elements of the target class via those of the source class.
At the same time, it is not the most general case of weak similarity of classes of differential equations.
For example, the classes~$\mathcal K^\epsilon$ (resp.\ $\mathcal F^\epsilon$) and~$\mathcal P^\epsilon$ are weakly similar,
and each equation~$\mathcal P^\epsilon_C$ from the class~$\mathcal P^\epsilon$
is related to a set of equations in the class~$\mathcal K^\epsilon$ (resp.\ $\mathcal F^\epsilon$)
by a family of point transformations parameterized by an arbitrary solution of the equation~$\mathcal P^\epsilon_C$
(resp.\ of the adjoint equation to~$\mathcal P^\epsilon_C$), see Section~\ref{sec:GroupClassificationsWrtEquivGroupsKF}.
In contrast to~\cite{VaneevaPopovychSophocleous2009},
there is no closed form for values of the arbitrary element of~$\mathcal K^\epsilon$ (resp.\ $\mathcal F^\epsilon$)
in terms of associated values of the arbitrary element of~$\mathcal P^\epsilon$.
As a result, a more general version of the mapping method than that in~\cite{VaneevaPopovychSophocleous2009}
is required to tackle the group classification problems for the classes~$\mathcal K^\epsilon$ and~$\mathcal F^\epsilon$.\looseness=-1

To obtain the group classification of the class~$\tilde{\mathcal L}$ modulo the $G^\sim_{\tilde{\mathcal L}}$-equivalence from that
of its weakly similar class~$\mathcal L$ with the help of the mapping method,
one should start with a classification list~$\rm{CL}$ for the class~$\mathcal L$ modulo the $\mathcal G^\sim_{\mathcal L}$-equivalence,
which coincides with the $G^\sim_{\mathcal L}$-equivalence if the class~$\mathcal L$ is semi-normalized.
First of all, given a family $\{\mathfrak g^{\rm unf}_{\theta}\mid\theta\in\mathcal S\}$
of $\mathcal G^\sim_{\mathcal L}$-uniform subalgebras of the algebras~$\mathfrak g_\theta$,
one can construct as above the family \smash{$\{\mathfrak g^{\rm unf}_{\tilde{\theta}}\mid\tilde\theta\in\tilde{\mathcal S}\}$}
of \smash{$\mathcal G^\sim_{\tilde{\mathcal L}}$}-uniform subalgebras of the algebras~$\mathfrak g_{\tilde\theta}$,
which are thus \smash{$G^\sim_{\tilde{\mathcal L}}$}-uniform as well.
Let $\tilde{\mathfrak L}_\theta$ be the set of systems in~$\tilde{\mathcal L}$ that are similar to~$\mathcal L_\theta$.
In other words, for each \smash{$\tilde{\mathcal L}_{\tilde\theta}\in\tilde{\mathfrak L}_\theta$} there exists a point transformation~$\Phi$
such that \smash{$\tilde{\mathcal L}_{\tilde\theta}=\Phi_*\mathcal L_\theta$}.
Since the system similarity is an equivalence relation,
systems $\mathcal L_{\theta_1}$ and~$\mathcal L_{\theta_2}$ are $\mathcal G^\sim_{\mathcal L}$-equivalent
if and only if the sets $\tilde{\mathfrak L}_{\theta_1}$ and~$\tilde{\mathfrak L}_{\theta_2}$ coincide.
Moreover, $\tilde{\mathfrak L}_{\theta_1}\cap\tilde{\mathfrak L}_{\theta_2}=\varnothing$
if and only if $\mathcal L_{\theta_1}$ and~$\mathcal L_{\theta_2}$ are $\mathcal G^\sim_{\mathcal L}$-inequivalent.
Therefore, the class~$\tilde{\mathcal L}$ can be represented as the disjoint union of
the subset~$\tilde{\mathfrak N}$ of systems in~$\tilde{\mathcal L}$ with no Lie-symmetry extensions and
the sets~$\tilde{\mathfrak L}_\theta$ each of which is associated with a Lie-symmetry extension case,
$\tilde{\mathcal L}=\tilde{\mathfrak N}\sqcup\big(\sqcup_{\theta\in\rm{CL}}\tilde{\mathfrak L}_\theta\big)$.
For any~$\theta\in{\rm CL}$, all elements of~\smash{$\tilde{\mathfrak L}_\theta$} are \smash{$\mathcal G^\sim_{\tilde{\mathcal L}}$}-equivalent
but not necessarily \smash{$G^\sim_{\tilde{\mathcal L}}$}-equivalent.
Thus, the mapping method reduces the group classification of the class~$\tilde{\mathcal L}$ up to the $G^\sim_{\tilde{\mathcal L}}$-equivalence
to the problem of finding, for each $\theta\in{\rm CL}$, a complete list of $G^\sim_{\tilde{\mathcal L}}$-inequivalent systems
in the set~\smash{$\tilde{\mathfrak L}_\theta$}.
In its turn, the latter problem may be alternatively seen as constructing, for each $\theta\in{\rm CL}$,
a set of point transformations from the system~$\mathcal L_\theta$ to systems from~$\tilde{\mathfrak L}_\theta$
such that the transformed systems constitute
a complete set of \smash{$G^\sim_{\tilde{\mathcal L}}$}-equivalent systems in~$\tilde{\mathfrak L}_\theta$.

For weakly similar classes with certain additional properties, other mapping procedures can be applied.
Suppose that for each~$\theta\in\mathcal S$ there exists a subset~$\mathfrak M_\theta$ of $\tilde{\mathfrak L}_\theta$
having the following properties:%
\footnote{%
The second property can be weakened, e.g., as follows.
If $\theta_1,\theta_2\in\mathcal S$ are $G^\sim_{\mathcal L}$-equivalent,
then for any $\tilde\theta_2\in\mathfrak M_{\theta_2}$
there exists $\tilde\theta_1\in\mathfrak M_{\theta_1}$ that is $G^\sim_{\tilde{\mathcal L}}$-equivalent to~$\tilde\theta_2$.
}
\begin{enumerate}\itemsep=0ex
\item
$\tilde{\mathcal L}=\cup_{\theta\in\mathcal S}\mathfrak M_\theta$.
\item
If $\theta_1,\theta_2\in\mathcal S$ are $G^\sim_{\mathcal L}$-equivalent,
then $\mathfrak M_{\theta_2}=\mathscr T_*\mathfrak M_{\theta_1}$ for some $\mathscr T\in G^\sim_{\tilde{\mathcal L}}$.
\item
If $\tilde{\mathcal L}_{\tilde\theta_1}\in\mathfrak M_{\theta_1}$
and $\tilde\theta_2$ is $G^\sim_{\tilde{\mathcal L}}$-equivalent to~$\tilde\theta_1$,
then there exists $\theta_2\in\mathcal S$ such that $\tilde{\mathcal L}_{\tilde\theta_2}\in\mathfrak M_{\theta_2}$
and $\theta_2$ is $G^\sim_{\mathcal L}$-equivalent to $\theta_1$.
\end{enumerate}
The third property necessarily holds if the class~$\mathcal L$ is semi-normalized.
Let~$\rm{CL}$ be a known classification list for the class~$\mathcal L$ modulo the $G^\sim_{\mathcal L}$-equivalence
with respect to a family $\{\mathfrak g^{\rm unf}_\theta\subseteq\mathfrak g_\theta\mid\theta\in\mathcal S\}$
of $\mathcal G^\sim_{\mathcal L}$-uniform subalgebras of the algebras~$\mathfrak g_\theta$.
We again construct the associated family \smash{$\{\mathfrak g^{\rm unf}_{\tilde\theta}\subseteq\mathfrak g_{\tilde\theta}\mid\tilde\theta\in\tilde{\mathcal S}\}$}
of \smash{$\mathcal G^\sim_{\tilde{\mathcal L}}$}-uniform subalgebras of the algebras~$\mathfrak g_{\tilde\theta}$,
which are thus \smash{$G^\sim_{\tilde{\mathcal L}}$}-uniform as well.
The classification list for the class~$\tilde{\mathcal L}$ is obtained by merging
complete lists of \smash{$G^\sim_{\tilde{\mathcal L}}$}-inequivalent values of~$\tilde\theta$
with $\tilde{\mathcal L}_{\tilde\theta}$ in the sets~$\mathfrak M_\theta$
for all values of~$\theta$ in~$\rm{CL}$.

\section{Equivalence groups and equivalence groupoids}\label{sec:FPEquivalenceGroup}

\subsection{General second-order linear evolution equations}

We use the paper~\cite{PopovychKunzingerIvanova2008} as a reference point for known results
on admissible transformations of the class~$\mathcal E$ of equations of the form~\eqref{eq:FPInhomoLinearEquation}.

\begin{proposition}\label{FP:pro:GeneralEquivGroup}
The class~$\mathcal E$ is normalized in the usual sense.
Its equivalence group~$G^\sim_{\mathcal E}$ consists of the transformations of the~form%
\footnote{Recall that in the course of working with the equivalence group of a class of differential equations,
the arbitrary elements of the class and their derivatives are considered as coordinates in a certain extended jet space and not as functions, 
cf.\ Section~\ref{FP:sec:AdmEqTransf}.%
}
\begin{gather}\label{eq:GenFormOfPointTransInEA}
\tilde t=T(t),\quad
\tilde x=X(t,x),\quad
\tilde u=U^1(t,x)u+U^0(t,x),\\ \label{eq:GenFormOfPointTransInEB}
\tilde A=\frac{X_x^2}{T_t}A,\quad
\tilde B=\frac{X_x}{T_t}\left(B-2\frac{U^1_x}{U^1}A\right)-\frac{X_t-X_{xx}A}{T_t},\quad
\tilde C=-\frac{U^1}{T_t}\mathrm E\frac1{U^1},\\ \nonumber
\tilde D=\frac{U^1}{T_t}\left(D+\mathrm E\frac{U^0}{U^1}\right),
\end{gather}
where~$T$, $X$, $U^0$ and~$U^1$ are arbitrary smooth functions of their arguments with~$T_tX_xU^1\ne0$, and
$\mathrm E:=\p_t-A\p_{xx}-B\p_x-C$.
\end{proposition}

The class~$\mathcal E$ is a natural choice for a superclass of the classes of Kolmogorov and heat equations considered below
due to its normalization in the usual sense.

\begin{remark}\label{rem:ClassEbreve}
For convenience of further singling out classes of Fokker--Planck equations,
we can reparameterize the equations from the class~$\mathcal E$ to the form
\[
u_t=\big(A(t,x)u\big)_{xx}+\big(B(t,x)u\big)_x+C(t,x)u+D(t,x),
\]
where the arbitrary elements~$A$ and~$D$ are the same as in~\eqref{eq:FPInhomoLinearEquation},
and $B$ and~$C$ in~\eqref{eq:FPInhomoLinearEquation} are expressed in terms of the new arbitrary elements as
$B+2A_x$ and $C+B_x+A_{xx}$, respectively.
In addition, we extend the arbitrary-element tuple $(A,B,C,D)$ with the derivatives~$A_x$, $A_{xx}$ and~$B_x$.
The reparameterized class~$\breve{\mathcal E}$ is normalized as well.
The $(t,x,u,A,D)$-components of equivalence transformations in the class~$\breve{\mathcal E}$
are of the same form as in Proposition~\ref{FP:pro:GeneralEquivGroup},
where instead of~$\mathrm E$
we use the proper expression $\breve{\mathrm E}=\p_t-\p_{xx}\circ A-\p_x\circ B-C$
for the differential operator associated with the reparameterized source equation,
and the $(B,C)$-components are
\begin{gather}\label{eq:EquivTransEbreveBC}
\tilde B=\frac{X_x}{T_t}\left(B-2\frac{U^1_x}{U^1}A\right)-\frac{X_t+3X_{xx}A}{T_t},\quad
\tilde C=-\frac{\breve{\mathrm E}^\dag(X_xU^1)}{T_tX_xU^1}.
\end{gather}
Here $\breve{\mathrm E}^\dag=-\p_t-A\p_{xx}+B\p_x-C$ is the formal adjoint of the operator~$\breve{\mathrm E}$.
The components for~$A_x$, $A_{xx}$ and~$B_x$ are obtained via the standard prolongation of the components for~$A$ and~$B$
to these derivatives.
\end{remark}

In the subsequent subsections we single out various subclasses of the class~$\mathcal E$
or of its reparameterization~$\breve{\mathcal E}$
via gauging arbitrary elements of these classes by wide families of admissible transformations
from the action groupoids of the corresponding equivalence groups
or imposing the condition that the arbitrary elements of a subclass do not depend on~$t$.
(This condition cannot be realized as a gauge.)
We study admissible transformations of the subclasses singled out.
The order of gauging is as follows.

Firstly, we single out the subclass~$\mathcal E_0$ of homogeneous equations,
gauging $D=0$ by a wide subset~$\mathcal M_{\mathcal E}$ of~$\mathcal G^{G^\sim_{\mathcal E}}$ 
and excluding~$D$ from the arbitrary-element tuple.
Here for each element of~$\mathcal M_{\mathcal E}$, we set $T(t)=t$, $X(t,x)=x$, $U^1(t,x)=1$,
and the parameter function $U^0$ is a solution of the equation from~$\mathcal E$
with the source value of the arbitrary-element tuple~$(A,B,C,D)$.

\begin{proposition}\label{FP:pro:EquivGroupE0}
The class~$\mathcal E_0$ is disjointedly semi-normalized with respect to the linear superposition of solutions.
Its equivalence group~$G^\sim_{\mathcal E_0}$ consists of the transformations of the~form
\begin{gather}\label{eq:EquivTransInE0A}
\tilde t=T(t),\quad
\tilde x=X(t,x),\quad
\tilde u=U^1(t,x)u,\\\label{eq:EquivTransInE0B}
\tilde A=\frac{X_x^2}{T_t}A,\quad
\tilde B=\frac{X_x}{T_t}\left(B-2\frac{U^1_x}{U^1}A\right)-\frac{X_t-X_{xx}A}{T_t},\quad
\tilde C=-\frac{U^1}{T_t}\mathrm E\frac1{U^1},
\end{gather}
where~$T$, $X$ and~$U^1$ are smooth functions of their arguments with~$T_tX_xU^1\ne0$ and
$\mathrm E=\p_t-A\p_{xx}-B\p_x-C$ is the differential operator associated with the source equation.
\end{proposition}

\begin{remark}\label{rem:ClassEbreve0}
The subclass~$\breve{\mathcal E}_0$ of homogenous equations
in the reparameterization~$\breve{\mathcal E}$ of the class~$\mathcal E$
is a similar reparameterization of the class~$\mathcal E_0$.
At the same time, in contrast to~$\breve{\mathcal E}$, the class~$\breve{\mathcal E}_0$
does not need the extension of the arbitrary-element tuple $(A,B,C)$ to derivatives
for preserving normalization properties of this class.
More specifically, as its non-reparameterized counterpart~$\mathcal E_0$,
the class~$\breve{\mathcal E}_0$ with the arbitrary-element tuple $(A,B,C)$
is disjointedly semi-normalized with respect to the linear superposition of solutions.
The $(t,x,u)$-, $A$- and $(B,C)$-components of the transformations
constituting the equivalence group~\smash{$G^\sim_{\breve{\mathcal E}_0}$} of the class~$\breve{\mathcal E}_0$
are defined by the equations~\eqref{eq:EquivTransInE0A}, the first equation in~\eqref{eq:EquivTransInE0B}
and the equations~\eqref{eq:EquivTransEbreveBC}, respectively,
where~$T$, $X$ and~$U^1$ are smooth functions of their arguments with~$T_tX_xU^1\ne0$.
\end{remark}

The obvious second choice would be the gauge~$A=1$, because this gauge does not affect the normalization property, cf.~\cite{PopovychKunzingerIvanova2008}.
Nevertheless, at this stage we prefer to single out the desired classes~$\bar{\mathcal P}$, $\bar{\mathcal K}$ and~$\bar{\mathcal F}$
from the class~$\mathcal E_0$ with the gauge~$B=0$, from the class~$\mathcal E_0$ with the gauge $C=0$
and from the class~$\breve{\mathcal E_0}$ with the gauge $C=0$, respectively.
Note that all these gauges are carried out by wide families of elements of the action groupoids of the corresponding equivalence groups.
The subsequent gauge~$A=\epsilon$ with fixed $\epsilon\in\{-1,1\}$ in the classes~$\bar{\mathcal K}$ and~$\bar{\mathcal F}$
is of the same nature, whereas the same gauge in the class~$\bar{\mathcal P}$
can only be realized by admissible transformations that do not necessarily belong to $\mathcal G^{G^\sim_{\bar{\mathcal P}}}$.
In this way, we derive the respective subclasses~$\mathcal P^\epsilon$, $\mathcal K^\epsilon$ and~$\mathcal F^\epsilon$ of reduced equations.
A similar procedure is then carried out for the subclasses of equations with time-independent coefficients
from the classes~$\mathcal E$, $\mathcal E_0$, $\bar{\mathcal P}$, $\bar{\mathcal K}$ and~$\bar{\mathcal F}$.

Consider the classes~$\mathcal E'$ and~$\mathcal E'_0$
that are singled out from the classes~$\mathcal E$ and~$\mathcal E_0$, respectively,
by the constraints $A_t=B_t=C_t=0$.
(The classes~$\breve{\mathcal E}'$ and~$\breve{\mathcal E}'_0$ can be considered in a similar way.)

\begin{proposition}\label{FP:pro:EquivGroupOfE'AndE'_0}
(i) The equivalence group~$G^\sim_{\mathcal E'}$ of the class~$\mathcal E'$ consists of the transformations of the~form
\begin{gather*}
\tilde t=c_1t+c_2,\quad
\tilde x=X(x),\quad
\tilde u={\rm e}^{c_3t}W(x)u+U^0(t,x),\\
\tilde A=\frac{X_x^2}{c_1}A,\quad
\tilde B=\frac{X_x}{c_1}\left(B-2\frac{W_x}WA\right)+\frac{X_{xx}}{c_1}A,\quad
\tilde C=\frac W{c_1}\mathrm E'\frac1W+\frac{c_3}{c_1},\\
\tilde D=\frac{{\rm e}^{c_3t}W}{c_1}\left(D+\mathrm E\frac{U^0}{{\rm e}^{c_3t}W}\right).
\end{gather*}

(ii) The equivalence group~$G^\sim_{\mathcal E'_0}$ of the class~$\mathcal E'_0$ consists of the transformations of the~form
\begin{gather*}
\tilde t=c_1t+c_2,\quad
\tilde x=X(x),\quad
\tilde u={\rm e}^{c_3t}W(x)u,\\
\tilde A=\frac{X_x^2}{c_1}A,\quad
\tilde B=\frac{X_x}{c_1}\left(B-2\frac{W_x}WA\right)+\frac{X_{xx}}{c_1}A,\quad
\tilde C=\frac W{c_1}\mathrm E'\frac1W+\frac{c_3}{c_1}.
\end{gather*}

Here~$c_1$, $c_2$ and~$c_3$ are arbitrary constants with~$c_1\ne0$,
$X$, $W$ and~$U^0$ are arbitrary smooth functions of their arguments with~$X_xW\ne0$,
$\mathrm E:=\p_t-A\p_{xx}-B\p_x-C$ and $\mathrm E':=A\p_{xx}+B\p_x+C$.
\end{proposition}

\begin{proof}
The class~$\mathcal E$ is normalized in the usual sense,
the class~$\mathcal E_0$ is disjointedly semi-normalized with respect to the linear superposition of solutions,
and the equations from the class~$\mathcal E'_0$ have no common nonzero solutions.
Therefore, the groups~$G^\sim_{\mathcal E'}$ and~\smash{$G^\sim_{\mathcal E'_0}$} are the subgroups
of the groups~$G^\sim_{\mathcal E}$ and~$G^\sim_{\mathcal E_0}$, respectively,
that consist of the elements preserving the constraints $A_t=B_t=C_t=0$.

We act on the $A$-, $B$- and $C$-components of~$\mathscr T\in G^\sim_{\mathcal E}$,
which are of the form~\eqref{eq:GenFormOfPointTransInEB},
by the operator $\mathrm D_{\tilde t}=(T_t)^{-1}\mathrm D_t-X_t(T_tX_x)^{-1}\mathrm D_x$.
Here $\mathrm D_t$ and~$\mathrm D_x$ are the operators of total derivatives with respect to the variables $t$ and~$x$
in the jet space with independent variables $(t,x,u)$ and the arbitrary elements of the class under consideration
as the dependent variables, and
$\mathrm D_{\tilde t}$ is the operator~$\mathrm D_t$ in the transformed variables,
$\mathrm D_{\tilde t}=\mathscr T_*\mathrm D_t$.
We take into account the constraints $A_t=B_t=C_t=0$ and $\tilde A_{\tilde t}=\tilde B_{\tilde t}=\tilde C_{\tilde t}=0$
and split the obtained equations with respect to $A$, $B$, $C$, $A_x$, $B_x$ and~$C_x$,
which leads to the equations $T_{tt}=0$, $X_t=0$, $(U^1_x/U^1)_t=(U^1_t/U^1)_t=0$, i.e.,
$T=c_1t+c_2$, $X=X(x)$ and $U^1={\rm e}^{c_3t}W(x)$.

Since the $A$-, $B$- and $C$-components~\eqref{eq:EquivTransInE0B} of transformations in~$G^\sim_{\mathcal E_0}$
coincide with those in~$G^\sim_{\mathcal E}$, which are given by~\eqref{eq:GenFormOfPointTransInEB},
the procedure of singling out~$G^\sim_{\mathcal E'_0}$ from~$G^\sim_{\mathcal E_0}$ is the same.
\end{proof}

\subsection{Heat equations with potentials}\label{sec:FP:HeatEqsWithPotsEquivAndAdmTrans}

The classification technique to be applied for classes of Kolmogorov and Fokker--Planck equations
requires the knowledge of the group classification of the class~$\mathcal P^\epsilon$,
which is constituted by the reduced heat equations with potentials,
$\mathcal P^\epsilon_C$: $u_t=\epsilon u_{xx}+C(t,x)u$ with fixed $\epsilon\in\{-1,1\}$.

Starting with the class~$\mathcal E_0$, one can reach its subclass~$\mathcal P^\epsilon$
via successively gauging $A=\epsilon$ and~$B=0$ by wide subsets of the action groupoids
of the equivalence groups of the class~$\mathcal E_0$ and its subclass singled out by the gauge $A=\epsilon$,
respectively.
In general, as was emphasized in~\cite{OpanasenkoBihloPopovych2017},
the order of gauging arbitrary elements of a class of differential equations may be important
if one is to ensure normalization properties of gauged subclasses.
To demonstrate the importance of the choice of the order,
let us gauge the arbitrary element~$B$ of the class~$\mathcal E_0$ to zero first, $B=0$.
It is achieved by the wide subset of~$\mathcal G(G^\sim_{\mathcal E_0})$,
where $T(t)=t$, $U^1(t,x)=1$ and
the parameter function~$X$ is an arbitrary solution of the Kolmogorov equation $X_t=AX_{xx}+BX_x$ with the source value of~$(A,B)$.
The gauged subclass is the class~$\bar{\mathcal P}$ of equations of the form
\begin{gather*}
\bar{\mathcal P}_\vartheta\colon\ \ u_t=A(t,x)u_{xx}+C(t,x)u \quad\mbox{with}\quad \vartheta:=(A,C).
\end{gather*}

\begin{proposition}\label{FP:prop:EquivGroupoidPbar}
The equivalence groupoid~$\mathcal G^\sim_{\bar{\mathcal P}}$ of the class~$\bar{\mathcal P}$ consists of the triples~$(\vartheta,\Phi,\tilde\vartheta)$,
where $\vartheta=(A,C)$ runs through the set of values of the arbitrary-element tuple of~$\bar{\mathcal P}$,
and $\Phi$ is an arbitrary point transformation of the form~\eqref{eq:GenFormOfPointTransInEA},
\begin{gather*}
\Phi\colon\quad\tilde t=T(t),\quad\tilde x=X(t,x),\quad \tilde u=U^1(t,x)u+U^0(t,x),
\end{gather*}
where
$T$ and $X$ are smooth functions of their arguments with~$T_tX_x\ne0$,
$U^1$ is found from the relation 
\[
\frac{U^1_x}{U^1}=\frac{X_{xx}}{2X_x}-\frac1{2A}\frac{X_t}{X_x}
\]
with $U^1\ne0$,
and $U^0/U^1$ is a solution of the equation~$\bar{\mathcal P}_\vartheta$.
The value $\tilde\vartheta=(\tilde A,\tilde C)$ of the arbitrary-element tuple of~$\bar{\mathcal P}$
is related to~$\vartheta$ by
\begin{gather*}
\Phi^*\tilde A=\frac{X_x^2}{T_t}A,\quad \Phi^*\tilde C=-\frac{U^1}{T_t}\mathrm P\frac1{U^1}
\end{gather*}
with the differential operator $\mathrm P:=\p_t-A\p_{xx}-C$, which defines the source equation.
\end{proposition}

The class~$\bar{\mathcal P}$ is not normalized since each admissible transformations of~$\bar{\mathcal P}$ 
satisfies the relation from Proposition~\ref{FP:prop:EquivGroupoidPbar}.
Even though this condition can be solved for~$U^1$, the result cannot be interpreted in terms of extended generalized equivalence groups
of classes of differential equations.
Recall that parameters of such groups may nonlocally depend on the arbitrary elements of the class in a specific way only.
In particular, an introduction of virtual arbitrary elements should make
the dependence of group elements of the extended arbitrary-element tuple completely local,
which is not the case for the class~$\bar{\mathcal P}$.
A similar example can be found in~\cite[Proposition~21]{PocheketaPopovych2017}.
See a correct construction of an extended generalized equivalence group, e.g., in~\cite[Section~9]{OpanasenkoBihloPopovych2017}.
At the same time, it is easy to compute the usual equivalence group of the class~$\bar{\mathcal P}$,
splitting the above classifying condition and the equation for $U^0/U^1$ with respect to~$\vartheta$.

\begin{proposition}
The usual equivalence group of the class~$\bar{\mathcal P}$ is constituted by the point transformations of the form
\begin{gather*}
\tilde t=T(t),\quad\tilde x=X(x),\quad \tilde u=c\sqrt{|X_x|}\ u,\\
\tilde A=\frac{X_x^2}{T_t}A,\quad \tilde C=\frac1{T_t}\left( C-\left(\frac{X_{xxx}}{2X_x}-\frac34\left(\frac{X_{xx}}{X_x}\right)^2\right)A\right),
\end{gather*}
where
$T$ and $X$ are arbitrary smooth functions of their arguments with~$T_tX_x\ne0$, and $c$ is an arbitrary nonzero constant.
\end{proposition}

Carrying out the further gauge~$A=\epsilon$ with fixed $\epsilon\in\{-1,1\}$
by the wide subset of~$\mathcal G^\sim_{\bar{\mathcal P}}$ singled out by the constraints
$T_t=\epsilon\sgn A$, $X_x=|A|^{-1/2}$, $U^0=0$ and $U^1_x/U^1=(AX_{xx}-X_t)/(2AX_x)$
with the source value of~$A$,
we obtain the subclass~$\mathcal P^\epsilon$.
One can easily switch between the classes~$\mathcal P^{-1}$ and~$\mathcal P^{+1}$ using the transformation $(t,C)\mapsto(-t,-C)$.
For convenience, we also use the alternative notation $\mathcal P:=\mathcal P^{+1}$.

\begin{proposition}\label{FP:pro:EquivalenceGroupHeatPotential}
The equivalence group~$G^\sim_{\mathcal P^\epsilon}$ of the class~$\mathcal P^\epsilon$
consists of the point transformations with components of the form
\begin{subequations}\label{FP:eq:EquivalenceGroupReducedP}
\begin{gather}
\tilde t=T(t),\quad\tilde x=X^1(t)x+X^0(t),\quad\tilde u=V(t)\exp\left(-\frac{\epsilon X^1_t}{4X^1}x^2-\frac{\epsilon X^0_t}{2X^1}x\right)u,
\label{FP:eq:EquivalenceGroupReducedPa}\\
\tilde C=\frac1{(X^1)^2}\left(C{+}\frac\epsilon4{X^1}\left(\frac1{X^1}\right)_{tt}x^2-\frac\epsilon2{X^1}\left(\frac{X^0_t}{(X^1)^2}\right)_tx
+\frac{V_t}{V}+\frac{X^1_t}{2X^1}+\frac\epsilon4\left(\frac{X^0_t}{X^1}\right)^2\right),
\label{FP:eq:EquivalenceGroupReducedPb}
\end{gather}
\end{subequations}
where~$X^0$, $X^1$ and~$V$ run through the set of smooth functions of~$t$ with~$X^1V\ne0$,
and $T_t=(X^1)^2$.
Moreover, the class~$\mathcal P^\epsilon$ is disjointedly semi-normalized with respect to the linear superposition of solutions.
\end{proposition}

Thus, normalization that is lost at the first stage of gauging is recovered at the second one.
If we reverse the order of the gauges, $A=\epsilon$ first and $B=0$ second
(both realized by wide subsets of the action groupoids of the corresponding equivalence groups),
normalization is preserved at each step.

The following lemma may be seen as a corollary of Proposition~\ref{FP:pro:EquivalenceGroupHeatPotential}
but it is also of interest per se.
Denote by \smash{$G^{\rm ess}_{\mathcal P^\epsilon_C}$} the essential subgroup
of the point symmetry group~\smash{$G_{\mathcal P^\epsilon_C}$} of the equation~$\mathcal P^\epsilon_C$,
which consists of the essential point symmetry transformations of~$\mathcal P^\epsilon_C$,%
\footnote{%
More specifically, for any equation~$\mathcal P^\epsilon_C\in\mathcal P^\epsilon$, its point symmetry transformations are of the form
\[
\tilde t=T(t),\quad\tilde x=X^1(t)x+X^0(t),\quad\tilde u=V(t)\exp\left(-\frac{\epsilon X^1_t}{4X^1}x^2-\frac{\epsilon X^0_t}{2X^1}x\right)u+U^0(t,x),
\]
where the parameter functions~$T$, $X$, $V$ and~$U^0$ satisfy a system of equations involving~$C$.
The group~$G^{\rm ess}_{\mathcal P^\epsilon_C}$ is the subgroup of~$G_{\mathcal P^\epsilon_C}$ associated with the constraint $U^0=0$.
}
\smash{$G^{\rm ess}_{\mathcal P^\epsilon_C}=G_{\mathcal P^\epsilon_C}\cap\pi_*G^\sim_{\mathcal P^\epsilon}$},
cf.\ \cite[Section~2]{PopovychKunzingerIvanova2008} and \cite[Section~3]{KurujyibwamiBasarabHorwathPopovych2018}.
Thus, $G_{\mathcal P^\epsilon_C}=G^{\rm ess}_{\mathcal P^\epsilon_C}\ltimes G^{\rm lin}_{\mathcal P^\epsilon_C}$,
where $G^{\rm lin}_{\mathcal P^\epsilon_C}$ is the (normal)~sub\-group of~$G_{\mathcal P^\epsilon_C}$
related to the linear superposition of solutions.

\begin{lemma}\label{FP:AuxProp}
If the projection~$\pi_*\mathscr T$ of an equivalence transformation~$\mathscr T$ of the class~$\mathcal P^\epsilon$
maps a nonzero solution~$f$ of an equation~$\mathcal P^\epsilon_C\in\mathcal P^\epsilon$ to a solution~$\tilde f$ thereof,
then this projection belongs to the group~$G^{\rm ess}_{\mathcal P^\epsilon_C}$.
\end{lemma}

\begin{proof}
As an element of the group~$G^\sim_{\mathcal P^\epsilon}$,
the transformation $\mathscr T$ is of the form~\eqref{FP:eq:EquivalenceGroupReducedP}
and, acting on equations from~$\mathcal P^\epsilon$, it maps the equation~$\mathcal P^\epsilon_C$
to the equation~$\mathcal P^\epsilon_{\tilde C}$ from the same class~$\mathcal P^\epsilon$.
The latter means that the point transformation $\pi_*\mathscr T$ of $(t,x,u)$
pushes forward each solution of~$\mathcal P^\epsilon_C$,
including~$f$, to a solution of~$\mathcal P^\epsilon_{\tilde C}$.
Therefore, the function~$\tilde f$ is a (nonzero) solution
of both the equations~$\mathcal P^\epsilon_C$ and~$\mathcal P^\epsilon_{\tilde C}$.
Any nonzero solution of an equation in the class~$\mathcal P^\epsilon$
completely defines the value of the arbitrary element~$C$ associated with this equation
and thus the equation itself.
Therefore, $C=\tilde C$ as functions, i.e.,
$C(t,x)=\tilde C(t,x)$ for any $(t,x)$ in the domain running by the independent variables.
In other words, the transformation~$\pi_*\mathscr T$ maps each solution of the equation~$\mathcal P^\epsilon_C$
to a solution of the same equation,
and this in turn means that $\pi_*\mathscr T\in G^{\rm ess}_{\mathcal P^\epsilon_C}$.
\end{proof}

The class~$\mathcal P'^\epsilon$ of reduced heat equations with time-independent potentials and fixed $\epsilon\in\{-1,1\}$
consists of the equations of the form
\[
u_t=\epsilon u_{xx}+C(x)u.
\]
This class can be obtained via gauging the arbitrary elements~$A$ and~$B$ in the class~$\mathcal E'_0$
to $A=\epsilon$ and $B=0$
by the wide subset~$\mathcal M_{\mathcal E'_0}$ of~\smash{$\mathcal G(G^\sim_{\mathcal E'_0})$}
singled out by the constraints $T_t=\epsilon\sgn A$, $X_x=|A|^{-1/2}$ and $2W_x=(B/A+X_{xx}/X_x)W$
with the source value of $(A,B)$.
Nevertheless, since the class~$\mathcal E'_0$ is not normalized in any sense,
the class~$\mathcal P^\epsilon$ is the most appropriate superclass for~$\mathcal P'^\epsilon$
in the course of studying transformational properties of~$\mathcal P'^\epsilon$.

\begin{corollary}\label{cor:EquivGroupOfP'eps}
The equivalence group~$G^\sim_{\mathcal P'^\epsilon}$ of the class~$\mathcal {\mathcal P'^\epsilon}$ of
reduced heat equations with time-indepen\-dent potentials consists of the transformations of the form
\begin{gather}\label{eq:EquivtransOfP'eps}
\tilde t=c_1^2t+c_2,\quad\tilde x=c_1x+c_3,\quad\tilde u=c_4{\rm e}^{c_5t}u,\quad\tilde C=\frac{1}{c_1^2}\left(C+c_5\right),
\end{gather}
where~$c_1$, \dots, $c_5$ are arbitrary constants with~$c_1c_4\ne0$.
\end{corollary}

\begin{proof}
We follow the proof of Proposition~\ref{FP:pro:EquivGroupOfE'AndE'_0}
and use the notation $\mathrm D_t$, $\mathrm D_x$ and~$\mathrm D_{\tilde t}$ of this proof.
As the class~$\mathcal P^\epsilon$ is disjointedly semi-normalized with respect to the linear superposition of solutions,
the group~$G^\sim_{\mathcal P'^\epsilon}$ consists of the elements of~$G^\sim_{\mathcal P^\epsilon}$
that preserve the constraint $C_t=0$.
We act on the relation~\eqref{FP:eq:EquivalenceGroupReducedPb}
by the operator $\mathrm D_{\tilde t}=(X^1)^{-2}\mathrm D_t-(X^1)^{-3}(X^1_tx+X^0_t)\mathrm D_x$
and take into account that $C_t=0$ and $\tilde C_{\tilde t}=0$.
Splitting the derived equation with respect to~$C$ and~$C_x$,
we yield the system $X^0_t=X^1_t=(V_t/V)_t=0$.
\end{proof}

The rest of the section is devoted to the description of the equivalence groupoid~$\mathcal G^\sim_{\mathcal P'}$
of the class $\mathcal P':=\mathcal P'^{+1}$.
The analogous result for $\epsilon=-1$ is easily derived from this description by the transformation $(t,C)\mapsto(-t,-C)$.

Since the class~$\mathcal P'$ consists of linear homogeneous differential equations,
it is natural to classify the admissible transformations of this class
up to the $\mathcal G^{G^\sim_{\mathcal P'}}\ltimes\mathcal G^{\rm lin}_{\mathcal P'}$-equivalence
instead of the standard $G^\sim_{\mathcal P'}$-equivalence,
i.e., the linear superposition of solutions should additionally be taken into account.
Here $\mathcal G^{\rm lin}_{\mathcal P'}$ is the normal subgroupoid
of both~$\mathcal G^\sim_{\mathcal P'}$ and~$\mathcal G^{\rm f}_{\mathcal P'}$
that is associated with this superposition.
In other words, it consists of all the triples of the form $(C,\Phi,C)$,
where $C$ is an arbitrary smooth function of~$x$ and the point transformation~$\Phi$ is of the form
$\tilde t=t$, $\tilde x=x$, $\tilde u=u+f(t,x)$ with an arbitrary solution~$f(t,x)$ of the equation~$\mathcal P'_C$.
The symbol `$\ltimes$' denotes the semidirect $\star$-product.
Let $\mathcal G^{\rm ess}_{\mathcal P'}$ be the subgroupoid of~$\mathcal G^\sim_{\mathcal P'}$ that consists of triples
with $u$-components of the transformational parts being of the form $\tilde u=U(t,x)u$ for some nonzero functions~$U$ of~$(t,x)$.
The groupoid~$\mathcal G^\sim_{\mathcal P'}$ can be represented
as the semidirect $\star$-product of~$\mathcal G^{\rm ess}_{\mathcal P'}$ and~$\mathcal G^{\rm lin}_{\mathcal P'}$,
$\mathcal G^\sim_{\mathcal P'}=\mathcal G^{\rm ess}_{\mathcal P'}\ltimes\mathcal G^{\rm lin}_{\mathcal P'}$.
Moreover, the intersection~$\mathcal G^{\rm ess}_{\mathcal P'}$ and~$\mathcal G^{\rm lin}_{\mathcal P'}$ is trivial
since it coincides with the base groupoid of the class~$\mathcal P'$.
Hence for any $\mathcal T\in\mathcal G^\sim_{\mathcal P'}$, there exist unique
$\mathcal T_0\in\mathcal G^{\rm ess}_{\mathcal P'}$ and $\mathcal T_1,\mathcal T_2\in\mathcal G^{\rm lin}_{\mathcal P'}$
such that $\mathcal T=\mathcal T_0\star\mathcal T_1=\mathcal T_2\star\mathcal T_0$.
Thus, the classification of the admissible transformations of the class~$\mathcal P'$
up to the \smash{$\mathcal G^{G^\sim_{\mathcal P'}}\ltimes\mathcal G^{\rm lin}_{\mathcal P'}$}-equivalence
reduces to the classification of the elements of~$\mathcal G^{\rm ess}_{\mathcal P'}$
up to the $G^\sim_{\mathcal P'}$-equivalence.

\begin{theorem}\label{FP:thm:EquivGroupoidHeatStationary}
A minimal self-consistent generating (up to the $G^\sim_{\mathcal P'}$-equivalence and the linear superposition of solutions)
set~$\mathcal B$ of admissible transformations for the class~$\mathcal P'$
is the union of the following families of admissible transformations $(C,\Phi,\tilde C)$ (below $\mu\in\mathbb R$):
\begin{gather*}
\mathcal T_{1\mu}:=\left(\dfrac\mu{x^2}+x^2,\Phi_1,\dfrac\mu{\tilde x^2}\right), \quad
\Phi_1\colon\ \tilde t=\dfrac12\tan 2t,\ \tilde x=\dfrac x{\cos 2t},\ \tilde u=|\cos2t|^{1/2}{\rm e}^{-\tan(2t)x^2/2}u;\\[1.5ex]
\mathcal T_{2\mu}:=\left(\dfrac\mu{x^2}-x^2,\Phi_2,\dfrac\mu{\tilde x^2}\right), \quad
\Phi_2\colon\ \tilde t=\dfrac14{\rm e}^{4t},\ \tilde x={\rm e}^{2t}x,\ \tilde u={\rm e}^{-x^2/2-t}u;\\[1.5ex]
\mathcal T_3:=\left(x,\Phi_3,0\right), \quad \Phi_3\colon\ \tilde t=t,\ \tilde x=x-t^2,\ \tilde u={\rm e}^{tx-t^3/3}u.
\end{gather*}
\end{theorem}

\begin{proof}
Since the class~$\mathcal P'$ is a subclass of the class~$\mathcal P$,
the equivalence groupoid of~$\mathcal P'$ is a subgroupoid of the equivalence groupoid of~$\mathcal P$.
The class~$\mathcal P$ is disjointedly semi-normalized with respect to the linear superposition of solutions.
Hence the classification of the admissible transformations of~$\mathcal P'$
up to the $G^\sim_{\mathcal P'}$-equivalence and the linear superposition of solutions
reduces to finding, modulo the $G^\sim_{\mathcal P'}$-equivalence,
the point transformations of the form~\eqref{FP:eq:EquivalenceGroupReducedPa} between equations in~$\mathcal P'$,
i.e., transformations preserving the independence of the arbitrary element~$C$ on~$t$.
In other words, the relation between (time-independent) source and target arbitrary elements~$C$ and~$\tilde C$
is given by the equality~\eqref{FP:eq:EquivalenceGroupReducedPb} that is interpreted as a relation between functions,
i.e., $\tilde C$ should be substituted by~$\Phi^*\tilde C$.
We act on this relation by the operator $-(X^1)^2\p_{\tilde t}=-\p_t+(X^1)^{-1}(X^1_tx+X^0_t)\p_x$ taking into
account the aforementioned time-independence of the arbitrary elements.
This amounts to the relation
\begin{gather}\label{FP:eq:EquivGroupoidPoten}
\begin{split}
&\frac{X^1_tx+X^0_t}{(X^1)^3}C_x+\frac{2X^1_t}{(X^1)^3}C
-\frac{(X^1)^2}4\left(\frac1{(X^1)^3}\left(\frac1{X^1}\right)_{\!tt\,}\right)_{\!t}x^2
+\left(\frac{X^0_t}{2(X^1)^3}\right)_{\!tt}x
\\[1ex]&\qquad
-\left(\frac1{(X^1)^2}\frac{V_t}{V}\right)_{\!t}+\left(\frac1{4(X^1)^2}\right)_{\!tt}-\left(\frac{(X^0_t)^2}{2(X^1)^4}\right)_{\!t}=0.
\end{split}
\end{gather}
To analyze solutions of the equation~\eqref{FP:eq:EquivGroupoidPoten},
we employ the extension of the method of furcate splitting~\cite{BihloPopovych2020,NikitinPopovych2001,OpanasenkoBoykoPopovych2020}
to classifying admissible transformations~\cite{OpanasenkoBihloPopovych2017,VaneevaPopovychSophocleous2009}.
It allows us to find precisely the $G^\sim_{\mathcal P'}$-equivalent values of the arbitrary element~$C$
that are sources of admissible transformations in~$\mathcal P'$ not generated by the equivalence group~$G^\sim_{\mathcal P'}$.
(The exact form of the corresponding point transformations is to be computed directly thereafter.)
The \emph{template form} of equations for such~$C$ is
\[
(a_1x+a_2)C_x+2a_1C+a_3x^2+a_4x+a_5=0
\]
for some constants~$a$'s.
The complete system of template-form equations for a specific value of the arbitrary element~$C$
is obtained by varying admissible transformations with the source~$C$ and fixing values of~$t$ in~\eqref{FP:eq:EquivGroupoidPoten}.
Denote by~$k$ the number of linearly independent tuples~$(a_1,a_2,a_3,a_4,a_5)$ associated with these equations.
It is obvious that any system of template-form equations with $k>2$ has no solutions.
Consider the possible values $k\in\{0,1,2\}$ separately.

\medskip\par\noindent
$\boldsymbol{k=0.}$
Each value of the tuple of parameter functions~$X^0$, $X^1$ and~$V$ that leads only to the identity template-form equation
corresponds to a wide family of admissible transformations with the same transformational parts.
Any such family is contained in the action groupoid of the equivalence group~$G^\sim_{\mathcal P'}$.

\medskip\par\noindent
$\boldsymbol{k=2.}$ Up to linearly recombining two fixed independent template-form equations,
we can assume without loss of generality that they have the form
\begin{gather*}
xC_x+2C+a_3x^2+a_4x+a_5=0,\quad
C_x+a'_3x^2+a'_4x+a'_5=0.
\end{gather*}
Integrating the second equation we see that $C(x)$ is at most a cubic polynomial,
but the first equation implies that the coefficient of~$x^3$ in~$C(x)$ vanishes.
Thus, $C(x)$ is at most a quadratic polynomial in~$x$,
and thus up to the $G^\sim_{\mathcal P'}$-equivalence it belongs to $\{-x^2,x^2,x,0\}$.
Moreover, $k=2$ for each such~$C$.

\medskip\par\noindent
$\boldsymbol{k=1.}$ Let~$(a_1,a_2,a_3,a_4,a_5)$ be the corresponding (nonzero) coefficient tuple.
For admissible transformations in $\mathcal G^\sim_{\mathcal P'}\setminus(\mathcal G^{G^\sim_{\mathcal P'}}\ltimes \mathcal G^{\rm lin})$,
the coefficient tuple of $(xC_x,C_x,x^2,x^1,x^0)$ in the equation~\eqref{FP:eq:EquivGroupoidPoten}
is proportional to the tuple $(a_1,a_2,a_3,a_4,a_5)$ with a nonzero multiplier~$\lambda(t)$.

If $a_1=0$, then $a_2\ne0$ and $X^1_t=0$. Hence $a_3=0$ as well,
and without loss of generality we can set $a_2=1$. The template-form equation assumes the form
\[
C_x+a_4x+a_5=0,
\]
meaning that $C(x)$ is at most a quadratic polynomial and should be excluded from consideration here
since $k=2$ for any such~$C$.

If $a_1\neq0$, then we can set $a_1=1$ without loss of generality and $a_2=0$ ($\!{}\bmod G^\sim_{\mathcal P'}$).
Therefore $X^0_t=0$, which implies $a_4=0$. The template-form equation takes the form
\[
xC_x+2C+a_3x^2+a_5=0,
\]
which has the general solution $C(x)=\mu/x^2-a_3x^2/4-a_5/2$, where $\mu\ne0$ in view of $k=1$,
and therefore $C(x)=\mu/{x^2}\pm x^2$ up to the~$G^\sim_{\mathcal P'}$-equivalence.

\medskip\par

The heat equations with the potentials~$\mu/x^2\pm x^2$, $\mu\in\mathbb R$, and $x$ are related via point transformations
to the heat equations with the canonical potentials $\mu/x^2$ and~0, respectively.
The heat equations with potentials~$\mu/x^2$ and~$\tilde\mu/x^2$, where $\mu\ne\tilde\mu$,
are not similar to each other with respect to point transformations.%
\footnote{
Indeed, it follows immediately from Proposition~\ref{FP:pro:EquivalenceGroupHeatPotential} that
if $\mathcal P_{\mu/x^2}\sim\mathcal P_{\mu'/x^2}$ $(\!{}\bmod G^\sim_{\mathcal P})$, then $\mu=\mu'$.
}
Therefore, a generating (up to the $G^\sim_{\mathcal P'}$-equivalence and the linear superposition of solutions)
set of admissible transformations for the class~$\mathcal P'$ consists of
$\mathcal T_{1\mu}$, $\mathcal T_{2\mu}$, $\mu\in\mathbb R$, $\mathcal T_3$
and $G^\sim_{\mathcal P'}$-inequivalent essential point symmetries of the heat equations
with canonical potentials~$\mu/x^2$, $\mu\in\mathbb R$.
The essential point symmetry groups of the equations~$\mathcal P_{\mu/x^2}$, $\mu\neq0$, and~$\mathcal P_0$ are
constituted by the point transformations of the form
\begin{gather}\label{eq:SymGroupPmux2}
\tilde t=\frac{\alpha t+\beta}{\gamma t+\delta},\quad \tilde x=\frac{\Lambda x}{\gamma t+\delta},\quad
\tilde u=\sigma \sqrt{|\gamma t+\delta|}\,u\exp{\frac{\gamma x^2}{4(\gamma t+\delta)}}
\end{gather}
and
\begin{gather*}
\tilde t=\frac{\alpha t+\beta}{\gamma t+\delta},\quad \tilde x=\frac{\Lambda x+\kappa t+\nu}{\gamma t+\delta},\quad
\tilde u=\sigma\sqrt{|\gamma t+\delta|}\,u\exp{\frac{\gamma\Lambda^2x^2-2\Lambda\Gamma x+\Gamma(\kappa t+\nu)}{4\Lambda^2(\gamma t+\delta)}},
\end{gather*}
respectively, where $\alpha$, $\beta$, $\gamma$, $\delta$, $\kappa$, $\nu$ and $\sigma$ are arbitrary constants
with $\alpha\delta-\beta\gamma>0$ and $\sigma\neq0$,
and the respective tuples $(\alpha,\beta,\gamma,\delta)$ and $(\alpha,\beta,\gamma,\delta,\kappa,\nu)$ are defined up to nonzero multipliers,
$\Lambda :=\pm\sqrt{\alpha\delta-\beta\gamma}$, $\Gamma:=\kappa\delta-\nu\gamma$.
Each element of the vertex group at~$C(t,x)=\mu/x^2$ with $\mu\in\mathbb R$ is $G^\sim_{\mathcal P'}$-equivalent to
either the unit at~$C$ or the triple
\[
\mathcal T_{4\mu}:=\left(\dfrac\mu{x^2},\Phi_4,\dfrac{\mu}{\tilde x^2}\right), \quad
\Phi_4\colon\ \tilde t=-\dfrac1t,\ \tilde x=\dfrac xt,\ \tilde u=\sqrt{|t|}\,{\rm e}^{-x^2/(4t)}u.
\]
In particular, any transformation of the form~\eqref{eq:SymGroupPmux2} with $\gamma\ne0$
is the successive composition of a scale transformation, a shift with respect to~$t$, the transformation~$\Phi_4$ and one more shift with respect to~$t$,
whereas any Galilean boost is the successive composition of the transformation~$\Phi_4$ 
with a shift with respect to~$x$ and with one more copy of~$\Phi_4$.
At the same time, the admissible transformation~$\mathcal T_{4\mu}$ itself can be decomposed as
$\mathcal T_{4\mu}=\check{\mathcal T}_{0\mu}\star\mathcal T_{1\mu}\star\mathcal T_{0\mu}\star\mathcal T_{1\mu}{}^{-1}$,
where the admissible transformations
\begin{gather*}
\mathcal T_{0\mu}:=\left(\dfrac\mu{x^2}+x^2,\Phi_0,\dfrac{\mu}{\tilde x^2}+\tilde x^2\right), \quad
\Phi_0\colon\ \tilde t=t-\frac\pi4,\ \tilde x=x,\ \tilde u=u,
\\
\check{\mathcal T}_{0\mu}:=\left(\dfrac\mu{x^2},\check\Phi_0,\dfrac{\mu}{\tilde x^2}\right), \quad
\check\Phi_0\colon\ \tilde t=4t,\ \tilde x=2x,\ \tilde u=\frac{u}{\sqrt2}
\end{gather*}
from the vertex groups at $C(t,x)=\mu/x^2+x^2$ and at $C(t,x)=\mu/x^2$, respectively, belong to~\smash{$\mathcal G^{G^\sim_{\mathcal P'}}$} as well.

The minimality and the self-consistency of the set $\mathcal B=\{\mathcal T_{1\mu},\mathcal T_{2\mu},\mathcal T_3\mid\mu\in\mathbb R\}$
is obvious.
\end{proof}

\begin{remark}
Recall that point transformations are interpreted as local diffeomorphisms.
Thus, the transformation~$\Phi_1$ becomes invertible after restricting it on an interval $\big(\pi k/2,\pi(k+1)/2\big)$ with $k\in\mathbb Z$.
The range of the $t$-component of~$\Phi_2$ is $(0,+\infty)$.
To get, e.g., the range $(-\infty,0)$, we should take the composition $\mathcal T_{4\mu}\star\mathcal T_{2\mu}$,
and to get other ranges, we should compose $\mathcal T_{2\mu}$ with other elements of the vertex group at $C(t,x)=\mu/x^2$.
\end{remark}

In view of the self-consistency of the generating set~$\mathcal B$
up to the $\mathcal G^{G^\sim_{\mathcal P'}}\ltimes \mathcal G^{\rm lin}$-equivalence,
any admissible transformation~$\mathcal T$ between equations from the $\mathcal G^\sim_{\mathcal P'}$-orbit of~\smash{$\mathcal P'_{\mu/x^2}$}
can be presented as the composition
\[
\mathcal T=\breve{\mathcal T}\star\mathcal T^{\varepsilon_1}\star\mathcal T_3^{\varepsilon_2}
\star\mathcal T'\star\mathcal T_3^{-\varepsilon_3}\star\bar{\mathcal T}^{-\varepsilon_4}\star\check{\mathcal T}\star\mathcal T_0,
\]
where $\varepsilon_1,\dots,\varepsilon_4\in\{0,1\}$,
$\breve{\mathcal T},\check{\mathcal T}\in\mathcal G^{G^\sim_{\mathcal P'}}$,
$\mathcal T,\bar{\mathcal T}\in\{\mathcal T_{1\mu},\mathcal T_{2\mu}\}$,
$\mathcal T'\in\mathcal G_{\mu/x^2}$,
$\mathcal T_0\in\mathcal G^{\rm lin}_{\mathcal P'}$,
and each two successive admissible transformations in the decomposition are composable,
in particular, $\varepsilon_1\varepsilon_2=\varepsilon_3\varepsilon_4=0$ and, if $\mu\ne0$, $\varepsilon_2=\varepsilon_3=0$.
Furthermore, either $\mathcal T'\in\mathcal G_{\mu/x^2}\cap\mathcal G^{G^\sim_{\mathcal P'}}$
or $\mathcal T'=\mathcal T'_1\star\mathcal T_{4\mu}\star\mathcal T'_2$
with \smash{$\mathcal T'_1,\mathcal T'_2\in\mathcal G_{\mu/x^2}\cap\mathcal G^{G^\sim_{\mathcal P'}}$}, and
$\mathcal T_{4\mu}=\check{\mathcal T}_{0\mu}\star\mathcal T_{1\mu}\star\mathcal T_{0\mu}\star\mathcal T_{1\mu}{}^{-1}$.

If the source of $\mathcal T\in\mathcal G^\sim_{\mathcal P'}$ belongs to the complement
of the union of the $\mathcal G^\sim_{\mathcal P'}$-orbits of~\smash{$\mathcal P'_{\mu/x^2}$}, $\mu\in\mathbb R$, in~$\mathcal P'$,
then $\mathcal T=\check{\mathcal T}\star\mathcal T_0$, where
\smash{$\check{\mathcal T}\in\mathcal G^{G^\sim_{\mathcal P'}}$} and
$\mathcal T_0\in\mathcal G^{\rm lin}_{\mathcal P'}$.

Theorem~\ref{FP:thm:EquivGroupoidHeatStationary} combined with Theorem~\ref{FP:thm:GroupClassificationOfTimeIndepPotentials} below
implies the following assertion.

\begin{corollary}\label{FP:prop:EquivalenceGroupoidTIPoten}
A point transformation that is, up to the linear superposition of solutions,
not the projection of a transformation from the equivalence group~$G^\sim_{\mathcal P'}$ to the space of independent and dependent variables
connects two similar equations from the class~$\mathcal P'$
if and only if their essential Lie invariance algebras
strictly contain the (essential) kernel algebra $\mathfrak g^{\cap\rm ess}_{\mathcal P'}=\mathfrak g^\cap_{\mathcal P'}$ of the class~$\mathcal P'$.
\end{corollary}

In other words, extension cases for admissible transformations of the class~$\mathcal P'$ are the same as extension cases for essential Lie symmetries within this class,
\[
\mathcal T\in\mathcal G^\sim_{\mathcal P'}\setminus(\mathcal G^{G^\sim_{\mathcal P'}}\ltimes \mathcal G^{\rm lin})\ \Leftrightarrow\
\mathfrak g^{\rm ess}_{{\rm s}(\mathcal T)},\mathfrak g^{\rm ess}_{{\rm t}(\mathcal T)}\varsupsetneq\mathfrak g^{\cap\rm ess}_{\mathcal P'}.
\]

Theorem~\ref{FP:thm:EquivGroupoidHeatStationary} and Corollary~\ref{FP:prop:EquivalenceGroupoidTIPoten}
can be extended to the subclasses~$\mathcal E'$ and~$\mathcal E'_0$ of the classes~$\mathcal E$ and~$\mathcal E_0$, respectively,
that consist of the equations with time-independent coefficients~$A$, $B$ and~$C$.

\begin{corollary}\label{FP:cor:EquivGroupoidOfE'AndE'_0}
(i) A minimal self-consistent generating (up to the \smash{$G^\sim_{\mathcal E'}$}-equivalence)
set of admissible transformations for the class~$\mathcal E'$
is $\mathcal B=\{\mathcal T_{1\mu},\mathcal T_{2\mu},\mathcal T_3\mid\mu\in\mathbb R\}$.

(ii) A minimal self-consistent generating (up to the \smash{$G^\sim_{\mathcal E'_0}$}-equivalence and the linear superposition of solutions)
set of admissible transformations for the class~$\mathcal E'_0$
is $\mathcal B=\{\mathcal T_{1\mu},\mathcal T_{2\mu},\mathcal T_3\mid\mu\in\mathbb R\}$.
\end{corollary}

\begin{proof}
The class~$\mathcal E'_0$ is related to the class~$\mathcal P'$
via the wide subset~$\mathcal M_{\mathcal E'_0}$ of~\smash{$\mathcal G(G^\sim_{\mathcal E'_0})$}.
In view of Proposition~\ref{FP:pro:EquivGroupOfE'AndE'_0} and Corollary~\ref{cor:EquivGroupOfP'eps},
two equations from~$\mathcal P'$ are $G^\sim_{\mathcal P'}$-equivalent if and only if
they are  $G^\sim_{\mathcal E'_0}$-equivalent.
The subgroupoid $\mathcal G^{\rm lin}_{\mathcal E'_0}$ of~$\mathcal G^{\rm f}_{\mathcal E'_0}$
that is related to the linear superposition of solutions for equations from~$\mathcal E'_0$
is the normal subgroupoid of~$\mathcal G^\sim_{\mathcal E'_0}$.
Hence the set~$\mathcal B$ from Theorem~\ref{FP:thm:EquivGroupoidHeatStationary}
is a minimal self-consistent generating (up to the \smash{$G^\sim_{\mathcal E'_0}$}-equivalence and the linear superposition of solutions)
set of admissible transformations for the class~$\mathcal E'_0$.

The argumentation for the class~$\mathcal E'$ is similar.
\end{proof}

\begin{corollary}\label{FP:cor:EquivGroupoidOfE'AndE'_0_2}
(i) A point transformation that is not the projection of a transformation
from the equivalence group~$G^\sim_{\mathcal E'}$ to the space of independent and dependent variables
connects two similar equations from the class~$\mathcal E '$
if and only if the essential Lie invariance algebras of their homogenous counterparts
strictly contain the (essential) kernel algebra
$\mathfrak g^{\cap\rm ess}_{\mathcal E'_0}=\mathfrak g^\cap_{\mathcal E'_0}$ of the class~$\mathcal E'_0$.

(ii) A point transformation that is, up to the linear superposition of solutions,
not the projection of an equivalence transformation of the class~$\mathcal E'_0$ to the space of independent and dependent variables
connects two similar equations from this class
if and only if their essential Lie invariance algebras
strictly contain the (essential) kernel algebra $\mathfrak g^{\cap\rm ess}_{\mathcal E'_0}=\mathfrak g^\cap_{\mathcal E'_0}$
of the class~$\mathcal E'_0$.
\end{corollary}

Thus, Corollaries~\ref{FP:prop:EquivalenceGroupoidTIPoten} and~\ref{FP:cor:EquivGroupoidOfE'AndE'_0_2}
give the second series of examples of classes of differential equations
with such a relation between admissible transformations and Lie symmetries in the literature
after the class of the general Burgers--Korteweg--de Vries equations of an arbitrary order~$n\geqslant2$ with time-independent arbitrary elements,
see~\cite[Section~10]{OpanasenkoBihloPopovych2017} and \cite{Opanasenko2019}.

\subsection{Kolmogorov equations}

The class~$\bar{\mathcal K}$ of Kolmogorov equations, $\bar{\mathcal K}_{A,B}\colon u_t=A(t,x)u_{xx}+B(t,x)u_x$,
is singled out from the class~$\mathcal E_0$ via gauging its arbitrary elements
by the wide family of elements of~\smash{$\mathcal G(G^\sim_{\mathcal E'_0})$}
with $T(t)=t$, $X(t,x)=x$, $U^1(t,x)=1/U(t,x)$,
where~$U$ is an arbitrary nonzero solution of the source equation, $U_t=AU_{xx}+BU_x+CU$.
The equivalence groupoid~$\mathcal G^\sim_{\bar{\mathcal K}}$ of the class~$\bar{\mathcal K}$
is easily derived from that of~$\mathcal E_0$ or~$\mathcal E$.

\begin{proposition}\label{FP:pro:FPEquivClassKolmogorov}
A point transformation~$\Phi$ connects two equations $\bar{\mathcal K}_{A,B}$ and~$\bar{\mathcal K}_{\tilde A,\tilde B}$
from the class~$\bar{\mathcal K}$ if and only if its components are of the form~\eqref{eq:GenFormOfPointTransInEA},
\[
\tilde t=T(t),\quad \tilde x=X(t,x),\quad \tilde u=U^1(t,x)u+U^0(t,x),
\]
where~$T$ and~$X$ are arbitrary smooth functions of their arguments with~$T_tX_x\ne0$, $U^1\ne0$,
and~$1/U^1$ and $U^0/U^1$ run through the solution set of the equation~$\bar{\mathcal K}_{A,B}$.
The arbitrary elements of the source and the target equations are related by
\[
\Phi^*\tilde A = \frac{X_x^2}{T_t}A, \quad \Phi^*\tilde B=\frac{X_x}{T_t}\left(B-2\frac{U^1_x}{U^1}A\right)-\frac{X_t-X_{xx}A}{T_t}.
\]
\end{proposition}

\begin{corollary}\label{FP:cor:EquivGroupBarK}
The equivalence group~$G^\sim_{\bar{\mathcal K}}$ of the class~$\bar{\mathcal K}$ of Kolmogorov equations
is constituted by the point transformations of the form
\[
\tilde t=T(t),\quad \tilde x=X(t,x),\quad \tilde u=U^1u+U^0,\quad
\tilde A = \frac{X_x^2}{T_t}A, \quad \tilde B=\frac{X_x}{T_t}B-\frac{X_t-X_{xx}A}{T_t},
\]
where $T$ and $X$ run through the set of smooth functions of their arguments with~$T_tX_x\ne0$,
and~$U^0$ and $U^1$ are arbitrary constants with~$U^1\ne0$.
\end{corollary}

It follows from Proposition~\ref{FP:pro:FPEquivClassKolmogorov} and Corollary~\ref{FP:cor:EquivGroupBarK}
that the class~$\bar{\mathcal K}$ is not normalized in any sense.

The class~$\bar{\mathcal K}$ is related to its subclass~$\mathcal K^\epsilon$
with fixed $\epsilon\in\{-1,1\}$
of reduced (backward for $\epsilon=-1$ and forward  for $\epsilon=1$) Kolmogorov equations
\begin{gather}\label{FP:eq:GaugedKolmogorov}
\mathcal K^\epsilon_B\colon\ u_t=\epsilon u_{xx}+B(t,x)u_x
\end{gather}
via the wide subset~$\mathcal M_{\bar{\mathcal K}}$ of elements of~\smash{$\mathcal G^{G^\sim_{\bar{\mathcal K}}}$}
with $T_t=\epsilon\sgn A$, $X_x=|A|^{-1/2}$, $U^0=0$ and $U^1=1$,
which imposes the gauge~$A=\epsilon$.
Within the framework of group analysis of differential equations,
the difference between forward and backward equations is not essential.
The classes~$\mathcal K^{+1}$ and~$\mathcal K^{-1}$ are obviously similar
with respect to alternating the signs of~$t$ and~$B$, and thus their equivalence groups coincide.

\begin{proposition}\label{FP:pro:EquinGroupoidOfK}
A point transformation~$\Phi$ connects two equations~$\mathcal K^\epsilon_B$ and~$\mathcal K^\epsilon_{\tilde B}$
from the class~$\mathcal K$ if and only if its components are of the form
\begin{gather*}
\tilde t=T(t),\quad \tilde x=X^1(t)x+X^0(t),\quad \tilde u=U^1(t,x)u+U^0(t,x),
\end{gather*}
where~$T$, $X^0$ and~$X^1$ are arbitrary smooth functions of~$t$ with $X^1\ne0$ and $T_t=(X^1)^2$,
$U^1\ne0$, and~$1/U^1$ and~$U^0/U^1$ are arbitrary solutions of the equation~$\mathcal K^\epsilon_B$.
The relation between the arbitrary elements of the source and the target equations is given by
\begin{gather*}
\Phi^*\tilde B=\frac1{X^1}\left(B-2\epsilon\frac{U^1_x}{U^1}\right)-\frac{X^1_tx+X^0_t}{(X^1)^2}.
\end{gather*}
\end{proposition}

The admissible transformations $(B,\Phi,\tilde B)$ described in Proposition~\ref{FP:pro:EquinGroupoidOfK}
constitute the equivalence groupoid~$\mathcal G^\sim_{\mathcal K^\epsilon}$ of the class~$\mathcal K^\epsilon$.

\begin{corollary}\label{FP:cor:EquivGroupK}
The equivalence group~$G^\sim_{\mathcal K^\epsilon}$ of the class~$\mathcal K^\epsilon$
consists of the point transformations of the form
\begin{gather}\label{FP:eq:EquivalenceGroupReducedKolmogorov}
\tilde t=T(t),\quad \tilde x=X^1(t)x+X^0(t),\quad \tilde u=c_1u+c_2,
\quad\tilde B=\frac1{X^1}B-\frac{X^1_tx+X^0_t}{(X^1)^2},
\end{gather}
where~$T$, $X^0$ and~$X^1$ are arbitrary smooth functions of~$t$ with $X^1\ne0$ and $T_t=(X^1)^2$,
and~$c_1$ and~$c_2$ are arbitrary constants with~$c_1\ne0$.
\end{corollary}

\begin{corollary}\label{eq:FPEquivClassK}
The group $G^\sim_{\mathcal K^\epsilon}$ can be embedded via the trivial prolongation of point transformations therein
to the arbitrary element~$A$, $\tilde A=A$, to the group $G^\sim_{\bar{\mathcal K}}$ as a subgroup thereof.
\end{corollary}

Corollary~\ref{eq:FPEquivClassK} is not trivial since the class~$\bar{\mathcal K}$ is not normalized.

Let us discuss the computation of the equivalence group~$G^\sim_{\bar{\mathcal K}'}$
of the class~$\bar{\mathcal K}'$ of Kolmogorov equations with time-independent diffusion and drift coefficients,
\[
u_t=A(x)u_{xx}+B(x)u_x.
\]
Since the class~$\bar{\mathcal K}$ is not normalized
or disjointedly semi-normalized with respect to the linear superposition of solutions,
we cannot directly prove that the elements of~$G^\sim_{\bar{\mathcal K}}$
preserving the time-independence of~$A$ and~$B$ constitute the entire group~$G^\sim_{\bar{\mathcal K}'}$.
In other words, singling out~$G^\sim_{\bar{\mathcal K}'}$ as a subgroup of~$G^\sim_{\bar{\mathcal K}}$
is not a proper way for computing~$G^\sim_{\bar{\mathcal K}'}$.
For this purpose, we should take another approach,
interpreting the class~$\bar{\mathcal K}'$ as the subclass of~$\mathcal E_0$
that is singled out by the constraint~$A_t=B_t=C=0$ and reparameterized by excluding~$C$
from the tuple of arbitrary elements.
The constraint~$A_t=B_t=0$ cannot be realized
via relating the class~$\bar{\mathcal K}$ to its subclass via a wide subset of its admissible transformations,
and hence this is not a gauge.
Since the explicit description of the equivalence groupoid of~$\bar{\mathcal K}'$ is cumbersome
and in fact is not needed in what follows, we present the group~$G^\sim_{\bar{\mathcal K}'}$ only.

\begin{proposition}\label{FP:cor:EquivGroupK'}
The equivalence group~$G^\sim_{\bar{\mathcal K}'}$ of the class~$\bar{\mathcal K}'$ is constituted by
the point transformations of the form
\[
\tilde t=c_1t+c_2,\quad\tilde x=X(x),\quad\tilde u= c_3u+c_4, \quad\tilde A=\frac{X_x^2}{c_1}A,\quad
\tilde B=\frac{X_x}{c_1}B+\frac{X_{xx}}{c_1}A,
\]
where~$X$ is an arbitrary function of~$x$ with $X_x\ne0$ and $c_1$, \dots, $c_4$ are arbitrary constants with $c_1c_3\ne0$.
\end{proposition}

Once again, it is convenient for further purposes to consider the subclass~$\mathcal K'^\epsilon$ of~$\bar{\mathcal K}'$
with fixed $\epsilon\in\{-1,1\}$
that consists of reduced (backward for $\epsilon=-1$ and forward  for $\epsilon=1$) Kolmogorov equations
with time-independent drifts,
\[u_t=\epsilon u_{xx}+B(x)u_x.\]
The corresponding gauging $A=\epsilon$ is imposed
by a wide subset of the action groupoid of~$G^\sim_{\bar{\mathcal K}'}$.

\begin{proposition}\label{pro:EquivGroupOfK'eps}
The equivalence group~$G^\sim_{\mathcal K'^\epsilon}$ of the class~$\mathcal K'^\epsilon$ of reduced
Kolmogorov equations with time-in\-de\-pen\-dent drifts consists of the point transformations of the form
\begin{gather*}
\tilde t=c_1^{\,\,2}t+c_2,\quad \tilde x=c_1x+c_3,\quad \tilde u=c_4u+c_5,\quad \tilde B=\frac{1}{c_1}B,
\end{gather*}
where~$c_1$, \dots, $c_5$ are arbitrary constants with~$c_1c_4\ne0$.
\end{proposition}

\subsection[Fokker-Planck equations]{Fokker--Planck equations}

The class~$\breve{\mathcal E}_0$ is related to its subclass~$\bar{\mathcal F}$ of Fokker--Planck equations
\[
\bar{\mathcal F}_{A,B}\colon\quad u_t=\big(A(t,x)u\big)_{xx}+\big(B(t,x)u\big)_x,
\]
where~$A$ and~$B$ are arbitrary smooth functions of~$(t,x)$ with~$A\neq0$,
by the wide subset of~\smash{$\mathcal G(G^\sim_{\breve{\mathcal E_0}})$}
singled out by the constraint $T(t)=t$, $X(t,x)=x$ and $U^1_t=-AU^1_{xx}+BU^1_x-CU^1$
with the source values of~$(A,B,C)$.

\begin{proposition}
A point transformation~$\Phi$ connects Fokker--Planck equations~$\bar{\mathcal F}_{A,B}$ and $\bar{\mathcal F}_{\tilde A,\tilde B}$
if and only if its components are of the form~\eqref{eq:GenFormOfPointTransInEA},
\[
\tilde t=T(t),\quad\tilde x=X(t,x),\quad \tilde u=U^1(t,x)u+U^0(t,x),
\]
where~$T$ and~$X$ run through the set of arbitrary  functions of their arguments with $T_tX_x\ne0$,
$U^1$~is an arbitrary nonzero smooth function of~$(t,x)$
such that $X_xU^1$ satisfies the adjoint equation to~$\bar{\mathcal F}_{A,B}$, $(X_xU^1)_t=-A(X_xU^1)_{xx}+B(X_xU^1)_x$,
and $U^0$ is an arbitrary solution of~$\bar{\mathcal F}_{A,B}$.
The arbitrary elements of the target equation are defined by
\[
\Phi^*\tilde A=\frac{X_x^2}{T_t}A, \quad
\Phi^*\tilde B=\frac{X_x}{T_t}\left(B-2\frac{U^1_x}{U^1}A\right)-\frac{X_t+3X_{xx}A}{T_t}.
\]
\end{proposition}

\begin{proof}
The equivalence groupoid~\smash{$\mathcal G^\sim_{\bar{\mathcal F}}$} of the class~$\bar{\mathcal F}$
is found as the subgroupoid of the equivalence groupoid
\smash{$\mathcal G^\sim_{\breve{\mathcal E}_0}=\mathcal G(G^\sim_{\breve{\mathcal E}_0})\ltimes\mathcal G^{\rm lin}_{\breve{\mathcal E}_0}$}
of the class~$\breve{\mathcal E_0}$, where the source and the target arbitrary-element tuples
of the corresponding admissible transformations satisfy the constraints $C=0$ and $\tilde C=0$, respectively.
Here \smash{$\mathcal G^{\rm lin}_{\breve{\mathcal E}_0}$} denotes the subgroupoid of~\smash{$\mathcal G^{\rm f}_{\breve{\mathcal E}_0}$}
associated with the linear superposition of solutions.
\end{proof}

Either directly or using the duality between Fokker--Planck and Kolmogorov equations~\cite[Section~5]{PopovychKunzingerIvanova2008},
we can show the following result.

\begin{corollary}\label{FP:cor:EquivGroupBarF}
The equivalence group~$G^\sim_{\bar{\mathcal F}}$ of the class~$\bar{\mathcal F}$ consists of
the point transformations of the form
\[
\tilde t=T(t),\quad\tilde x=X(t,x),\quad \tilde u=\frac{c_1}{X_x}u,\quad
\tilde A=\frac{X_x^2}{T_t}A, \quad \tilde B=\frac{X_x}{T_t}B-\frac{X_t+X_{xx}A}{T_t},
\]
where $T$ and~$X$ are smooth functions of their arguments with~$T_tX_x\neq0$,
and~$c_1$ is an arbitrary nonzero constant.
\end{corollary}

Similarly to~$\bar{\mathcal K}$, we relate the class~$\bar{\mathcal F}$
via the wide subset of $\mathcal G(G^\sim_{\bar{\mathcal F}})$,
where $T_t=\epsilon\sgn A$ and $X_x=|A|^{-1/2}$,
to its subclass~$\mathcal F^\epsilon$ with fixed $\epsilon\in\{-1,1\}$
of reduced (backward for $\epsilon=-1$ and forward for $\epsilon=1$) Fokker--Planck equations of the form
\begin{gather*}
u_t=\epsilon u_{xx}+(B(t,x)u)_x,
\end{gather*}
cf.~\cite{Miyadzawa1989}.
Note that the equation~$\mathcal F^\epsilon_B$ is formally adjoint to~$\mathcal K^{-\epsilon}_B$.

\begin{proposition}\label{FP:prop:EquivalenceGroupGaugedFP}
The equivalence groupoid~$\mathcal G^\sim_{\mathcal F^\epsilon}$ of the class~$\mathcal F^\epsilon$
is constituted by the triples of the form $(B,\Phi,\tilde B)$,
where~$B$ and~$\tilde B$ are the drifts of the source and the target Fokker--Planck equations, respectively,
and~$\Phi$ is a point transformation with the components
\[
\tilde t=T(t),\quad \tilde x=X^1(t)x+X^0(t),\quad \tilde u=U^1(t,x)u+U^0(t,x),
\]
where~$T$, $X^0$ and~$X^1$ are arbitrary smooth functions of~$t$ with $X^1\ne0$ and $T_t=(X^1)^2$,
$X^1U^1$ is an arbitrary nonzero solution of the adjoint equation to~$\mathcal F^\epsilon_B$,
$
(X^1U^1)_t=-\epsilon X^1U^1_{xx}+BX^1U^1_x,
$
and $U^0$ is an arbitrary solution of the source equation~$\mathcal F^\epsilon_B$.
The relation between the source and the target drifts is given by
\[
\Phi^*\tilde B=\frac{1}{X^1}\left(B-2\epsilon\frac{U^1_x}{U^1}\right)-\frac{X^1_tx+X^0_t}{(X^1)^2}.
\]
\end{proposition}

\begin{corollary}\label{FP:cor:EquivGroupF}
The equivalence group~$G^\sim_{\mathcal F^\epsilon}$ of the class~$\mathcal F^\epsilon$ consists of point transformations of the form
\[
\tilde t=T(t),\quad \tilde x=X^1(t)x+X^0(t),\quad \tilde u=\frac{c_1}{X^1}u+c_2,\quad
\tilde B=\frac{1}{X^1}B-\frac{X^1_tx+X^0_t}{(X^1)^2},
\]
where~$T$, $X^0$ and~$X^1$ are arbitrary smooth functions of~$t$ with $X^1\ne0$ and $T_t=(X^1)^2$,
and $c_1$ and~$c_2$ are arbitrary constants with $c_1\neq0$.
\end{corollary}

The class~$\bar{\mathcal F}'$ of Fokker--Planck equations
with time-independent diffusion and drift coefficients,
\[
u_t=(A(x)u)_{xx}+(B(x)u)_x,
\]
is of high interest in view of physical importance of such equations.
Thus, all the Fokker--Planck equations mentioned in the introduction as those modeling various stochastic processes
have time-independent diffusion and drift coefficients.
The equivalence groupoid of the class~$\bar{\mathcal F}'$ is found analogously to that of the class~$\bar{\mathcal K}'$,
see a discussion before Proposition~\ref{FP:cor:EquivGroupK'}.
However, since it is irrelevant in what follows,
we opt to present the equivalence group of~$\bar{\mathcal F}'$ only.

\begin{proposition}\label{FP:pro:EquivGroupStationaryFP}
The equivalence group~$G^\sim_{\bar{\mathcal F}'}$ of the class~$\bar{\mathcal F}'$ consists of
the point transformations of the form
\[
\tilde t=c_1t+c_2,\quad\tilde x=X(x),\quad\tilde u= \frac{c_3}{X_x}u, \quad
\tilde A=\frac{X_x^2}{c_1}A,\quad
\tilde B=\frac{X_x}{c_1}B-\frac{X_{xx}}{c_1}A,
\]
where~$X$ is an arbitrary function of~$x$ with $X_x\ne0$,
and~$c_1$, $c_2$ and $c_3$ are arbitrary constants with $c_1c_3\ne0$.
\end{proposition}

Similarly as before, we gauge the arbitrary element~$A$ of the class~$\bar{\mathcal F}'$ to a fixed $\epsilon\in\{-1,1\}$
by a wide subset of $\mathcal G(G^\sim_{\bar{\mathcal F}'})$,
thus relating this class to its subclass of reduced (backward for $\epsilon=-1$ and forward for $\epsilon=1$) Fokker--Planck equations
with time-independent drifts
\[
\mathcal F'^\epsilon_B\colon\quad u_t=\epsilon u_{xx}+(B(x)u)_x.
\]
Again, the equation~$\mathcal F'^\epsilon_B$ is formally adjoint to~$\mathcal K'^{-\epsilon}_B$.

\begin{proposition}\label{pro:EquivGroupOfF'eps}
The equivalence group~$G^\sim_{\mathcal F'^\epsilon}$ of the class~$\mathcal F'^\epsilon$ of reduced
Fokker--Planck equations with time-independent drifts consists of the point transformations of the form
\begin{gather}\label{eq:EquivTransF'eps}
\tilde t=c_1^2t+c_2,\quad \tilde x=c_1x+c_3,\quad \tilde u=c_4u,\quad \tilde B=\frac1{c_1}B,
\end{gather}
where~$c_1$, \dots, $c_4$ are arbitrary constants with~$c_1c_4\ne0$.
\end{proposition}

\section[Group classifications of classes of Kolmogorov and Fokker-Planck equations]
{Group classifications of classes of Kolmogorov\\ and Fokker--Planck equations}\label{sec:FPGCArbitrary}

Any equation~$E$ from the class~$\mathcal E_0$ admits, in view of the linear superposition of solutions,
the point symmetry transformations of the form $\tilde t=t$, $\tilde x=x$, $\tilde u=u+f(t,x)$,
where~$f(t,x)$ is an arbitrary solution of~$E$.
Therefore, the maximal Lie invariance algebra~$\mathfrak g_E$ of~$E$ is of the form
$\mathfrak g_E=\mathfrak g^{\rm ess}_E\lsemioplus\mathfrak g^{\rm lin}_E$.
Here $\mathfrak g^{\rm lin}_E$ is the infinite-dimensional ideal of~$\mathfrak g_E$
that corresponds to the linear superposition of solutions,
$\mathfrak g^{\rm lin}_E:=\{f(t,x)\p_u\mid f \text{ is a solution of }E\}$,
and $\mathfrak g^{\rm ess}_E=\mathfrak g_E\cap\pi_*\mathfrak g^\sim_{\mathcal E_0}$ 
is a finite-dimensional subalgebra of~$\mathfrak g_E$,
which is called the essential Lie invariance algebra of the equation~$E$.
In what follows, we classify the essential Lie invariance algebras
of Kolmogorov and Fokker--Planck equations.
The following proposition is an obvious consequence of Proposition~\ref{FP:pro:EquivGroupE0}.

\begin{proposition}\label{pro:glinInE0}
The family $\{\mathfrak g^{\rm lin}_E\mid E\in\mathcal E_0\}$ is \smash{$\mathcal G^\sim_{\mathcal E_0}$}-uniform,
and thus for any subclass~$\mathcal L$ of~$\mathcal E_0$,
the corresponding family $\{\mathfrak g^{\rm lin}_E\mid E\in\mathcal L\}$ is \smash{$\mathcal G^\sim_{\mathcal L}$}-uniform.
\end{proposition}

The class~$\mathcal E_0$ is related via a wide subset of $\mathcal G(G^\sim_{\mathcal E_0})$ to its subclass~$\mathcal P^\epsilon$,
and both the classes~$\mathcal E_0$ and~$\mathcal P^\epsilon$ are disjointedly semi-normalized with respect to the linear superposition principle.
In view of Proposition~\ref{pro:GroupClassificationOfSeminormClassesByMappingMethod},
the group classification problem for the class~$\mathcal E_0$ reduces to that for the class~$\mathcal P^\epsilon$.
These problems were solved by S.~Lie himself~\cite{Lie1881}.
Below, we revisit the classification result for the class~$\mathcal P^\epsilon$ using Theorem~\ref{FP:thm:EquivGroupoidHeatStationary}.

\begin{theorem}\label{thm:FP:ClassificationsE0AndP}
A complete list of $G^\sim_{\mathcal P^\epsilon}$-inequivalent (resp.\ $\mathcal G^\sim_{\mathcal P^\epsilon}$-inequivalent)
essential Lie-sym\-me\-try extensions of the (essential) kernel algebra
$\mathfrak g^{\cap\,\rm ess}_{\mathcal P^\epsilon}=\mathfrak g^\cap_{\mathcal P^\epsilon}=\langle u\p_u\rangle$
in the class~$\mathcal P^\epsilon$ is exhausted by the following cases:
\begin{enumerate}\itemsep=-0.3ex
\item $C=C(x)\colon$\quad$\p_t$;
\item $C=\mu x^{-2}\colon$\quad $\p_t$, \ $D$, \ $\Pi$;
\item $C=0\colon$\quad $\p_t$, \ $D$, \ $\Pi$, \ $\p_x$, \ $G$.
\end{enumerate}
Here $\mu\ne0$, $D=2t\p_t+x\p_x$, $\Pi=4t^2\p_t+4tx\p_x-(\epsilon x^2+2t)u\p_u$, $G=2t\p_x-\epsilon xu\p_u$,
and only basis vector fields of essential Lie-symmetry extensions are presented.
The similar assertion on the class~$\mathcal E_0$ is obtained via setting $A=\epsilon$ and $B=0$ for each element of the list.
\end{theorem}

\begin{remark}\label{rem:Case1WithinPeps}
For the Lie-symmetry extension in Case~1 to be maximal,
the potential~$C$ should take a general value, which is not equivalent to the values from the other cases,
i.e., for any $\alpha,\beta,\gamma\in\mathbb R$ and $\mu\in\mathbb R\setminus\{0\}$ we have
$C(x)\ne\mu(x+\beta)^{-2}+\alpha(x+\beta)^2+\gamma$ and $C(x)\ne\alpha x^2+\beta x+\gamma$.
In view of Theorem~\ref{FP:thm:EquivGroupoidHeatStationary},
the $G^\sim_{\mathcal P^\epsilon}$-equivalence, which coincides with the $\mathcal G^\sim_{\mathcal P^\epsilon}$-equivalence,
reduces within Case~1 to the $G^\sim_{\mathcal P'^\epsilon}$-equivalence,
cf.\ Corollary~\ref{cor:EquivGroupOfP'eps}.
\end{remark}

We use the classification list for the class~$\mathcal P^\epsilon$ presented in Theorem~\ref{thm:FP:ClassificationsE0AndP}
as a base for applying the mapping method in order to solve group classification problems
for classes~$\bar{\mathcal K}$, $\mathcal K^\epsilon$, $\bar{\mathcal F}$ and~$\mathcal F^\epsilon$,
which are weakly similar to the class~$\mathcal P^\epsilon$.

\subsection{Group classifications with respect to equivalence groupoids}

The group classification problems for the entire classes of Kolmogorov and Fokker--Planck equations
with respect to the associated equivalence groupoids were solved in~\cite{PopovychKunzingerIvanova2008}.
We revisit those results via taking into account Remark~\ref{rem:Case1WithinPeps}
and involving more assertions from Sections~\ref{sec:FPEquivalenceGroup} and~\ref{sec:FPGCTimeIndependent}.

\begin{theorem}\label{thm:FPClassificationFP}
The kernel algebras~$\mathfrak g^\cap_{\bar{\mathcal K}}$ and~$\mathfrak g^\cap_{\bar{\mathcal F}}$
of the classes~$\bar{\mathcal K}$ and~$\bar{\mathcal F}$ are~$\langle \p_u,u\p_u\rangle$ and~$\langle u\p_u\rangle$, respectively.
Complete lists of $\mathcal G^\sim_{\bar{\mathcal K}}$-inequivalent essential Lie-symmetry extensions
of the essential kernel algebra $\mathfrak g^{\cap\,\rm ess}_{\bar{\mathcal K}}=\langle u\p_u\rangle$ and of $\mathcal G^\sim_{\bar{\mathcal F}}$-inequivalent
Lie-symmetry extensions of the essential kernel algebra~$\mathfrak g^{\cap\,\rm ess}_{\bar{\mathcal F}}=\mathfrak g^\cap_{\bar{\mathcal F}}$
are exhausted by the following cases:
\begin{enumerate}\itemsep=-0.3ex
\item $B=B(x)\colon$\quad$\p_t$;
\item $B=\epsilon\breve\epsilon x^{-1}\big(1-2a\tan(a\ln|x|)\big)\colon$\ \ $\p_t$, \ $D-\tfrac12\epsilon xB(x)u\p_u$, \ $\Pi-2\epsilon txB(x)u\p_u$;
\item $B=\epsilon\breve\epsilon bx^{-1}\colon$\quad $\p_t$, \ $D$, \ $\Pi-2\breve\epsilon btu\p_u$;
\item $B=0\colon$\quad $\p_t$, \ $\p_x$, \ $D$, \ $\Pi$, \ $G$.
\end{enumerate}
In all the cases, $A=\epsilon$ with $\epsilon$ taking a single value in $\{-1,1\}$,
and only basis vector fields of essential Lie-symmetry algebras besides $u\p_u$ are presented.
In addition, $a>0$, $b\geqslant1$, $b\ne2$,
\mbox{$D=2t\p_t+x\p_x$}, $\Pi=4t^2\p_t+4tx\p_x-(\epsilon x^2+2t)u\p_u$, $G=2 t\p_x-\epsilon xu\p_u$,
$\breve\epsilon=1$ for the class~$\bar{\mathcal K}$ and $\breve\epsilon=-1$ for the class~$\bar{\mathcal F}$.
\end{theorem}

\begin{proof}
The class~$\bar{\mathcal K}$ is related to its subclass~$\mathcal K^\epsilon$
via a wide subset of the action groupoid of its equivalence group,
and thus these classes are weakly similar.
The kernel algebras of these classes coincide,
and the family of uniform subalgebras corresponding to the linear superposition of solutions in~$\bar{\mathcal K}$
is clearly mapped to its counterpart in~$\mathcal K^\epsilon$.
Then the consideration after Definition~\ref{FP:def:WeaklySimilarClasses} implies
that the group classification of~$\bar{\mathcal K}$ up to the $\mathcal G^\sim_{\bar{\mathcal K}}$-equivalence
reduces to the group classification of~$\mathcal K^\epsilon $ up to the $\mathcal G^\sim_{\mathcal K^\epsilon}$-equivalence.
The same claim is also true for the classes~$\bar{\mathcal F}$ and~$\mathcal F^\epsilon$.
(See also the beginning of the next Section~\ref{sec:GroupClassificationsWrtEquivGroupsKF} for a stronger relation between the above classes.)

The classes~$\mathcal K^\epsilon$ and~$\mathcal F^\epsilon$ are weakly similar to the class~$\mathcal P^\epsilon$.
Indeed, given an equation~$\mathcal P^\epsilon_C$: $u_t=\epsilon u_{xx}+C(t,x)u$,
point transformations of the form
\begin{gather}\label{eq:TransPCToKBORFB}
(\tilde t,\tilde x)=(t,x),\ \tilde u=\frac u{\hat U(t,x)}\quad\mbox{and}\quad
(\tilde t,\tilde x)=(t,x),\ \tilde u=\check U(t,x)u
\end{gather}
with arbitrary nonzero solutions~$\hat U$ and~$\check U$
of the equations~$\mathcal P^\epsilon_C$ and~$(\mathcal P^\epsilon_C)^\dag=\mathcal P^{-\epsilon}_{-C}$
map it to the reduced Kolmogorov and Fokker--Planck equations
\begin{gather*}
\mathcal K^\epsilon_B\colon\quad \tilde u_t=\epsilon\tilde u_{xx}+B(t,x)\tilde u\quad\mbox{with}\quad B:= 2\epsilon\hat U_x/\hat U,\\
\mathcal F^\epsilon_B\colon\quad \tilde u_t=\epsilon\tilde u_{xx}+\big(B(t,x)\tilde u\big)_x\quad\mbox{with}\quad B:=-2\epsilon\check U_x/\check U,
\end{gather*}
respectively.
Furthermore, all reduced Kolmogorov and Fokker--Planck equations can be obtained in this way.

There is a clear correspondence between the families of uniform subalgebras
corresponding to the linear superposition of solutions in~$\mathcal K^\epsilon$
(resp.\ $\mathcal F^\epsilon$) and~$\mathcal P^\epsilon$.
Moreover, the essential kernel algebras
of the classes~$\mathcal K^\epsilon$, $\mathcal F^\epsilon$ and~$\mathcal P^\epsilon$ coincide,
$\mathfrak g^{\cap\rm ess}_{\mathcal K^\epsilon}=\mathfrak g^{\cap\rm ess}_{\mathcal F^\epsilon}
=\mathfrak g^{\cap\rm ess}_{\mathcal P^\epsilon}=\langle u\p_u\rangle$.

Thus, the group classification problems up to the general point equivalence
for the classes~$\mathcal P^\epsilon$ and $\mathcal K^\epsilon$ (resp.\ $\mathcal F^\epsilon$) are equivalent.
In accordance with Section~\ref{FP:sec:MapMethod}, to obtain a classification list for the class~$\mathcal K^\epsilon$ (resp.\ $\mathcal F^\epsilon$)
up to this equivalence, one needs to find a single element from the image of each representative of the classification list for the class~$\mathcal P^\epsilon$
that is presented in Theorem~\ref{thm:FP:ClassificationsE0AndP}.
In other words, one should find a particular (e.g., time-independent) solution of the equation~$\mathcal P^\epsilon_C$
(resp.\ $\mathcal P^{-\epsilon}_{-C}$) for each of the potentials~$C$ from Theorem~\ref{thm:FP:ClassificationsE0AndP}
and recover the corresponding drift~$B$.
Then an arbitrary time-independent potential~$C$ leads to an arbitrary time-independent drift,
which gives Case~1 of Theorem~\ref{thm:FPClassificationFP}.
For the potential~$C=\mu/x^2$ with $\mu\ne0$,
each time-independent solution of~$\mathcal P^\epsilon_C$ (resp.\ $\mathcal P^{-\epsilon}_{-C}$)
satisfies the Euler equation $\epsilon x^2u_{xx}+\mu u=0$,
whose general solution depends on the value of~$\mu$.
This results in splitting a single classification case in the class~$\mathcal P^\epsilon$
to Cases~2 and~3 for the group classification of Fokker--Planck and Kolmogorov equations.
The potential $C=0$ trivially gives Case~4.
The corresponding essential Lie invariance algebras are obtained via pushing forward,
by the relevant point transformation of the form~\eqref{eq:TransPCToKBORFB},
the basis vector fields of the essential Lie invariance algebras
of the corresponding heat equations with potentials.
\end{proof}

\begin{remark}
It follows from Remark~\ref{rem:Case1WithinPeps}
and results of Sections~\ref{sec:GroupClassificationF'} and~\ref{sec:GroupClassificationK'}
that the Lie-symmetry extension in Case~1 of Theorem~\ref{thm:FPClassificationFP} is maximal
if and only if the corresponding value~$B=B(x)$ is general, i.e.,
it is $G^\sim_{\mathcal P'^\epsilon}$-inequivalent to values
collected in Table~\ref{tab:FPStationaryDriftsGC} except the general Case~0 of this table.
Moreover, Theorem~\ref{FP:thm:EquivGroupoidHeatStationary} with its consequences,
Propositions~\ref{FP:pro:EquinGroupoidOfK}, \ref{pro:EquivGroupOfK'eps}, \ref{FP:prop:EquivalenceGroupGaugedFP} and~\ref{pro:EquivGroupOfF'eps}
and lemmas from Sections~\ref{sec:GroupClassificationF'} and~\ref{sec:GroupClassificationK'}
imply that equations~$\bar{\mathcal K}_B$ and~\smash{$\bar{\mathcal K}_{\tilde B}$}
(resp.~$\bar{\mathcal F}_B$ and~\smash{$\bar{\mathcal F}_{\tilde B}$})
with $A=\epsilon$ and with general values $B=B(x)$ and $\tilde B=\tilde B(x)$ are similar if and only if
$c_1\tilde B(c_1x+c_2)=B(x)+2\epsilon\breve\epsilon W_x(x)/W(x)$,
where $c_1$ and~$c_2$ are arbitrary constants with $c_1\ne0$
and $W$ is an arbitrary nonvanishing solution of the equation $\epsilon\breve\epsilon W_{xx}+BW_x=\lambda W$
for a constant~$\lambda$.
\end{remark}

\begin{remark}
It is clear from the proof that,
after replacing~$\bar{\mathcal K}$ and~$\bar{\mathcal F}$ by~$\mathcal K^\epsilon$ and~$\mathcal F^\epsilon$,
respectively, Theorem~\ref{thm:FPClassificationFP} also gives the solutions of the group classification problems
for the classes~$\mathcal K^\epsilon$ and~$\mathcal F^\epsilon$ up to the general point equivalence.
\end{remark}

\subsection{Group classifications with respect to equivalence groups}\label{sec:GroupClassificationsWrtEquivGroupsKF}

Recall that the class~$\mathcal K^\epsilon$ is related to its superclass~$\bar{\mathcal K}$
via the subset~$\mathcal M_{\bar{\mathcal K}}$ of~\smash{$\mathcal G^{G^\sim_{\bar{\mathcal K}}}$},
see the paragraph with the equation~\eqref{FP:eq:GaugedKolmogorov}.
The essential kernel algebras of the classes~$\bar{\mathcal K}$ and~$\mathcal K^\epsilon$ are the same,
$\mathfrak g^{\cap\rm ess}_{\bar{\mathcal K}}=\mathfrak g^{\cap\rm ess}_{\mathcal K^\epsilon}=\langle u\p_u \rangle$.
If the projection of an equivalence transformation of the class~$\bar{\mathcal K}$
to the space with the coordinates $(t,x,u,B)$ relates two equations from the class~$\mathcal K^\epsilon$,
then, in view of Corollaries~\ref{FP:cor:EquivGroupBarK} and~\ref{FP:cor:EquivGroupK},
this projection belongs to the group~$G^\sim_{\mathcal K^\epsilon}$.
This implies that the $G^\sim_{\mathcal K^\epsilon}$- and $G^\sim_{\bar{\mathcal K}}$-equivalences
are consistent via~$\mathcal M_{\bar{\mathcal K}}$.
Additionally taking into account Proposition~\ref{pro:glinInE0}, we conclude that
the group classification of~$\bar{\mathcal K}$ up to the $G^\sim_{\bar{\mathcal K}}$-equivalence reduces
to the group classification of~$\mathcal K^\epsilon$ up to the $G^\sim_{\mathcal K^\epsilon}$-equivalence.
In view of Corollaries~\ref{FP:cor:EquivGroupBarF} and~\ref{FP:cor:EquivGroupF},
the analogous claim holds for the classes~$\bar{\mathcal F}$ and~$\mathcal F^\epsilon$.

To solve the group classification problems for the classes~$\mathcal K^\epsilon$ and~$\mathcal F^\epsilon$,
we apply the version of the mapping method described at the end of Section~\ref{FP:sec:MapMethod}
using the class~$\mathcal P^\epsilon$ of heat equations with potentials as that whose group classification is known.

We begin with the class~$\mathcal K^\epsilon$.
Consider the following set of point transformations mapping a
fixed heat equation~$\mathcal P^\epsilon_C$ with potential~$C=C(t,x)$ to equations in the class~$\mathcal K^\epsilon$:
\begin{gather*}
\mathfrak T_C:=\big\{\Phi_U\colon\, \tilde t=t,\ \tilde x=x,\ \tilde u=u/U(t,x)\mid U\ne0\colon\,U_t=\epsilon U_{xx}+CU\big\}.
\end{gather*}
Thus, a transformation~$\Phi_U\in\mathfrak T_C$ with a fixed nonzero solution~$U$ of~$\mathcal P^\epsilon_C$ maps the equation~$\mathcal P^\epsilon_C$
to the Kolmogorov equation~$\mathcal K^\epsilon_B$ with the drift~$B=2\epsilon U_x/U$, cf.~\cite{Bluman1980,BlumanShtelen2004}.
By~$\mathfrak K_C$ we denote the set of Kolmogorov equations that are the images of~$\mathcal P^\epsilon_C$ under such maps,
\[
\mathfrak K_C:=(\mathfrak T_C)_*\mathcal P^\epsilon_C=\big\{\mathcal K^\epsilon_B\mid B=2\epsilon U_x/U,\ U\ne0\colon\, U_t=\epsilon U_{xx}+CU \big\}.
\]

\begin{lemma}\label{FP:Lemma1a}
(i) For any equation $\mathcal K^\epsilon_B\in\mathcal K^\epsilon$, there exists a function~$C$ of~$(t,x)$
such that $\mathcal K^\epsilon_B\in\mathfrak K_C$.

(ii)
If reduced Kolmogorov equations $\mathcal K^\epsilon_B\in\mathfrak K_C$ and $\mathcal K^\epsilon_{\tilde B}\in\mathfrak K_{\tilde C}$
are $G^\sim_{\mathcal K^\epsilon}$-equivalent,
then the equations~$\mathcal P^\epsilon_C$ and~$\mathcal P^\epsilon_{\tilde C}$
are $G^\sim_{\mathcal P^\epsilon}$-equivalent.
\end{lemma}

\begin{proof}
(i) An equation~$\mathcal K^\epsilon_B$ belongs to~$\mathfrak K_C$
for any~$C$ of the form $C=(U_t-\epsilon U_{xx})/U$,
where~$U$ is a nonzero solution of the equation $2U_x=\epsilon BU$.

(ii) The equations~$\mathcal P^\epsilon_C$ and~$\mathcal P^\epsilon_{\tilde C}$ are similar,
and the class~$\mathcal P^\epsilon$ is semi-normalized.
\end{proof}

In what follows, for each of the classes under consideration,
$\varpi$ denotes the projection from the space with coordinates $(t,x,u,\theta)$ to the space with coordinates $(t,x)$,
where $\theta$ is the corresponding arbitrary-element tuple.

\begin{lemma}\label{FP:Lemma1b}
If $\tilde C=\mathscr T_*C$ for some $\mathscr T\in G^\sim_{\mathcal P^\epsilon}$,
then the transformation $\breve{\mathscr T}\in G^\sim_{\mathcal K^\epsilon}$ with $\varpi_*\breve{\mathscr T}=\varpi_*\mathscr T$
and the identity $u$-component maps $\mathfrak K_C$ onto~$\mathfrak K_{\tilde C}$, $\mathfrak K_{\tilde C}=\breve{\mathscr T}_*\mathfrak K_C$.
\end{lemma}

\begin{proof}
It is convenient to put tildes for the arguments $t$ and~$x$ of functions with tildes.
For an arbitrary Kolmogorov equation \smash{$\mathcal K^\epsilon _{\tilde B}\in\mathfrak K_{\tilde C}$},
we have $\tilde B=2\epsilon\tilde U_{\tilde x}/\tilde U$ for some nonzero solution~$\tilde U$ of~\smash{$\mathcal P^\epsilon_{\tilde C}$}.
Since the transformation~$\mathscr T$ belongs to~$G^\sim_{\mathcal P^\epsilon}$ and, therefore,
is of the form~\eqref{FP:eq:EquivalenceGroupReducedP},
the solution~$\tilde U$ of~$\mathcal P^\epsilon_{\tilde C}$ can be represented in the form
\begin{gather}\label{eq:TransOfU}
\tilde U(\tilde t,\tilde x)=V(t)U(t,x)\exp\left(-\epsilon\frac{X^1_t(t)}{4X^1(t)}x^2-\epsilon\frac{X^0_t(t)}{2X^1(t)}x\right)
\end{gather}
for some solution~$U$ of~$\mathcal P^\epsilon_C$,
where the expression of $(\tilde t,\tilde x)$ via $(t,x)$ is defined in~\eqref{FP:eq:EquivalenceGroupReducedKolmogorov}.
Setting $B:=2\epsilon U_x/U$, we obtain
\begin{gather}\label{eq:TransOfB}
\tilde B(\tilde t,\tilde x)=\frac{2\epsilon U_x(t,x)}{X^1(t)U(t,x)}-\frac{X^1_t(t)x+X^0_t(t)}{(X^1(t))^2}=\frac{B(t,x)}{X^1(t)}-\frac{X^1_t(t)x+X^0_t(t)}{(X^1(t))^2}.
\end{gather}
This means that the Kolmogorov equation~$\mathcal K^\epsilon_B$ with the drift~$B=2\epsilon U_x/U$,
which belongs to~$\mathfrak K_C$,
is mapped to~$\mathcal K^\epsilon_{\tilde B}$ by the point transformation
$\tilde t=T(t)$, $\tilde x=X^1(t)x+X^0(t)$, $\tilde u=u$ of~$(t,x,u)$ with $T_t=(X^1)^2$,
which is prolonged to~$B$ according to~\eqref{eq:TransOfB}.
This defines the required transformation $\breve{\mathscr T}\in G^\sim_{\mathcal K^\epsilon}$,
which does not depend on the choice of~$\tilde B$.
Hence $\mathfrak K_{\tilde C}\subseteq\breve{\mathscr T}_*\mathfrak K_C$.
Reverting the argumentation,
\smash{$(\mathcal K^\epsilon_B\in\mathfrak K_C)\rightsquigarrow
(U\in\mathcal P^\epsilon_C)\rightsquigarrow
(\tilde U\in\mathcal P^\epsilon_{\tilde C})\rightsquigarrow
(\mathcal K^\epsilon_{\tilde B}\in\mathfrak K_{\tilde C})$},
we derive the inverse inclusion.
As a result, $\mathfrak K_{\tilde C}=\breve{\mathscr T}_*\mathfrak K_C$.
\end{proof}

\begin{corollary}\label{cor:FPSimilarity1}
Kolmogorov equations
$\mathcal K^\epsilon_B=(\Phi_U)_*\mathcal P^\epsilon_C$ and
$\mathcal K^\epsilon_{\tilde B}=(\Phi_{\tilde U})_*\mathcal P^\epsilon_C$
from~$\mathfrak K_C$ are $G^\sim_{\mathcal K^\epsilon}$-equivalent
if and only if the corresponding nonzero solutions~$U$ and~$\tilde U$
of the equation~$\mathcal P^\epsilon_C$ are $G^{\rm ess}_{\mathcal P^\epsilon_C}$-equivalent.
\end{corollary}

\begin{proof}
Since the class~$\mathcal P^\epsilon$ is disjointedly semi-normalized with respect to the linear superposition of solutions,
the essential point symmetry group~\smash{$G^{\rm ess}_{\mathcal P^\epsilon_C}$} of the equation~$\mathcal P^\epsilon_C$
coincides with the stabilizer subgroup of $\pi_*G^\sim_{\mathcal P^\epsilon}$ with respect to~$\mathcal P^\epsilon_C$.
Then the lemma's statement follows from the proof of Lemma~\ref{FP:Lemma1b} for the particular value $\tilde C=C$
or, more specifically, from the equivalence of the equations~\eqref{eq:TransOfU} and~\eqref{eq:TransOfB}.
\end{proof}

The above assertions on equations from the class~$\mathcal K^\epsilon$ constitute the theoretical basis
for analyzing the group classification problem for this class
with respect to its equivalence group~$G^\sim_{\mathcal K^\epsilon}$.
In view of Lemma~\ref{FP:Lemma1a},
if equations~$\mathcal P^\epsilon_C$ and~$\mathcal P^\epsilon_{\tilde C}$ are $G^\sim_{\mathcal P^\epsilon}$-inequivalent,
then any equations $\mathcal K^\epsilon_B\in\mathfrak K_C$ and $\mathcal K^\epsilon_{\tilde B}\in\mathfrak K_{\tilde C}$
are $G^\sim_{\mathcal K^\epsilon}$-inequivalent.
Lemma~\ref{FP:Lemma1b} implies that it suffices to consider only the image of a single representative~$\mathcal P^\epsilon_C$
of each $G^\sim_{\mathcal P^\epsilon}$-orbit.
In other words, $C$ should run through the set of all potentials of the canonical $G^\sim_{\mathcal P^\epsilon}$-inequivalent equations
that admit extensions of the essential kernel algebra of the class~$\mathcal P^\epsilon$,
see Theorem~\ref{thm:FP:ClassificationsE0AndP}.
Corollary~\ref{cor:FPSimilarity1} shows how to distinguish inequivalent Kolmogorov equations in the set~$\mathfrak K_C$.
It is necessary to find a complete set of $G^{\rm ess}_{\mathcal P^\epsilon_C}$-inequivalent solutions~$U^\gamma$ of the equation~$\mathcal P^\epsilon_C$,
where $\gamma$ indexes this set,
and then to map the equation~$\mathcal P^\epsilon_C$ by each of the point transformations~$\Phi_{U^\gamma}\in\mathfrak T_C$.

The procedure for group classification of the class~$\mathcal F^\epsilon$ up to the $G^\sim_{\mathcal F^\epsilon}$-equivalence
can be described similarly. It is based on the following construction.
Consider a heat equation~$\mathcal P^\epsilon_C$ with potential~$C=C(t,x)$
and the set~$\check{\mathfrak T}_C$ of point transformations
that are parameterized by solutions of the heat equation~$\mathcal P^{-\epsilon}_{-C}$
and map~$\mathcal P^\epsilon_C$ to Fokker--Planck equations from the class~$\mathcal F^\epsilon$,
\[
\check{\mathfrak T}_C=\big\{\Psi_U\colon\,\tilde t=t,\ \tilde x=x,\ \tilde u=U(t,x)u \mid U\neq0\colon\,U_t=-\epsilon U_{xx}-CU\big\}.
\]
By~$\mathfrak F_C$ we denote the set of Fokker--Planck equations that are the images of~$\mathcal P^\epsilon_C$ under such maps,
\[
\mathfrak F_C:=(\check {\mathfrak T}_C)_*\mathcal P^\epsilon_C=
\big\{\mathcal F^\epsilon_B\mid B=-2\epsilon U_x/U,\ U\neq0\colon\,U_t=-\epsilon U_{xx}-CU\big\}.
\]

\begin{lemma}
(i) For any equation $\mathcal F^\epsilon_B\in\mathcal F^\epsilon$, there exists a function~$C$ of~$(t,x)$
such that $\mathcal F^\epsilon_B\in\mathfrak F_C$.

(ii) If reduced Kolmogorov equations $\mathcal F^\epsilon_B\in\mathfrak F_C$ and $\mathcal F^\epsilon_{\tilde B}\in\mathfrak F_{\tilde C}$
are $G^\sim_{\mathcal F^\epsilon}$-equivalent,
then the equations~$\mathcal P^\epsilon_C$ and~$\mathcal P^\epsilon_{\tilde C}$
are $G^\sim_{\mathcal P^\epsilon}$-equivalent.

(iii) If $\tilde C=\mathscr T_*C$ for some $\mathscr T\in G^\sim_{\mathcal P^\epsilon}$,
then the transformation $\breve{\mathscr T}\in G^\sim_{\mathcal F^\epsilon}$ with $\varpi_*\breve{\mathscr T}=\varpi_*\mathscr T$
and the identity $u$-component maps $\mathfrak F_C$ onto~$\mathfrak F_{\tilde C}$, $\mathfrak F^\epsilon_{\tilde C}=\breve{\mathscr T}_*\mathfrak F_C$.

(iv) Fokker--Planck equations
$\mathcal F^\epsilon_B=(\Psi_U)_*\mathcal P^\epsilon_C$ and
\smash{$\mathcal F^\epsilon_{\tilde B}=(\Psi_{\tilde U})_*\mathcal P^\epsilon_C$} from~$\mathfrak F_C$
are $G^\sim_{\mathcal F^\epsilon}$-equiv\-a\-lent
if and only if the corresponding nonzero solutions~$U$ and~$\tilde U$ of the equation~$\mathcal P^{-\epsilon}_{-C}$ are
$G^{\rm ess}_{\mathcal P^{-\epsilon}_{-C}}$-equivalent.
\end{lemma}

Summing up the results of this section, we obtain the following joint description
of the group classifications of the classes~$\bar{\mathcal K}$, $\mathcal K^\epsilon$, $\bar{\mathcal F}$ and $\mathcal F^\epsilon$
up to the equivalences generated by the corresponding equivalence groups.
Recall that $\epsilon$ takes a fixed value in $\{-1,1\}$, and in the classes~$\bar{\mathcal K}$ and~$\bar{\mathcal F}$
we can choose any of these two values.

\begin{theorem}\label{thm:GroupClassificationOfbarKbarFKepsFeps}
Let $\mathcal L$ be one of the classes~$\bar{\mathcal K}$, $\mathcal K^\epsilon$, $\bar{\mathcal F}$ and $\mathcal F^\epsilon$.
A complete list of $G^\sim_{\mathcal L}$-inequivalent essential Lie-symmetry extensions of the essential kernel algebra
$\mathfrak g^{\cap\rm ess}_{\mathcal L}=\langle u\p_u\rangle$ of~$\mathcal L$
is exhausted by the cases with $A=\epsilon$ if $\mathcal L\in\{\bar{\mathcal K},\bar{\mathcal F}\}$
and $B=2\epsilon\breve\epsilon U_x/U$,
where
$\breve\epsilon= 1$ for $\mathcal L\in\{\bar{\mathcal K},\mathcal K^\epsilon\}$ and
$\breve\epsilon=-1$ for $\mathcal L\in\{\bar{\mathcal F},\mathcal F^\epsilon\}$,
the function~$U$ of $(t,x)$ runs through a complete set of
\smash{$G^{\rm ess}_{\mathcal P^{\epsilon\breve\epsilon}_{\breve\epsilon C}}$}-inequivalent nonzero solutions
of the equations $\mathcal P^{\epsilon\breve\epsilon}_{\breve\epsilon C}$,
and the potential~$C$ runs through the values listed in Theorem~\ref{thm:FP:ClassificationsE0AndP}
subject to the conditions of Remark~\ref{rem:Case1WithinPeps}.
Depending on~$C$,  basis vector fields of essential Lie invariance algebras besides $u\p_u$ are the following:
\begin{enumerate}\itemsep=0ex
\item $C=C(x)\colon$\quad $\hat P^t_U$;
\item $C=\mu x^{-2}\colon$\quad $\hat P^t_U$, \ $\hat D_U$, \ $\hat\Pi_U$;
\item $C=0\colon$\quad $\hat P^t_U$, \ $\hat D_U$, \ $\hat\Pi_U$, \ $\hat P^x_U$, \ $\hat G_U$.
\end{enumerate}
Here
\begin{gather*}
\hat P^t_U:=\p_t-R(t,x)u\p_u,\quad
\hat D_U:=2t\p_t+x\p_x-\big(2tR(t,x)+\tfrac12\epsilon xB(t,x)\big)u\p_u,\\
\hat\Pi_U:=4t^2\p_t+4tx\p_x-\big(\epsilon x^2+2t+4t^2R(t,x)+2\epsilon txB(t,x)\big)u\p_u,\\
\hat P^x_U:=\p_x-\tfrac12\epsilon B(t,x)u\p_u,\quad
\hat G_U:=2t\p_x-\epsilon\big(x+tB(t,x)\big)u\p_u,
\end{gather*}
and $R:=\breve\epsilon U_t/U$.
\end{theorem}

\subsection{Constructing an explicit classification list is a no-go problem}\label{FP:sec:GroupFoliation}

In principle, the only difference of the group classification of the class~$\mathcal K^\epsilon$
up to the $G^\sim_{\mathcal K^\epsilon}$-equivalence from that up to the general point equivalence
lies in the necessity of finding, in a closed form,
all \smash{$G^{\rm ess}_{\mathcal P^\epsilon_C}$}-inequivalent solutions
of each selected equation~$\mathcal P^\epsilon_C$ from the class~$\mathcal P^\epsilon$
instead of a single particular solution thereof.
(For the class~$\mathcal F^\epsilon$,
the equation~$\mathcal P^\epsilon_C$ is replaced by the equation $\mathcal P^{-\epsilon}_{-C}$.)
At the same time, it is the biggest obstacle to accomplish the classification procedure
and present an explicit classification list.
The problem of finding all \smash{$G^{\rm ess}_{\mathcal P^\epsilon_C}$}-inequivalent solutions of an equation~$\mathcal P^\epsilon_C$
is equivalent to representing the entire solution set of this equation in a closed form%
\footnote{%
Integral representations, which are known, e.g., for the general solutions of the equations~$\mathcal P^\epsilon_C$
with $C=\mu x^{-2}$, are not appropriate for using in the course of presenting a group classification list.
}
and singling out canonical representatives among \smash{$G^{\rm ess}_{\mathcal P^\epsilon_C}$}-equivalent solutions.

In general, the problem of finding solutions of a differential equation that are inequivalent with respect to its point symmetry group
can be considered with the help of the group foliation method,
see~\cite[Chapter~VII]{Ovsiannikov1982} and~\cite{AncoAliWolf2011,MartinaSheftelWinternitz2001}.%
\footnote{%
We revisit basic notions of the theory of group foliations using the notation of Section~\ref{FP:sec:TheorBack}.
A~system of differential equations is called \emph{automorphic} with respect to a group~$G$ of point transformations 
in the space coordinatized by the independent and dependent variables of this system (or, briefly, \emph{$G$-automorphic}) 
if its solution set is the $G$-orbit of its fixed solution. 
Each $G$-automorphic system can be rewritten in terms of differential invariants of~$G$ in a special way. 
If a system of differential equations~$E$ with independent variables~$x$ and dependent variables~$u$
admits a group~$G$ as its point symmetry group, 
then there exists a group foliation of~$E$ with respect to~$G$, 
which is a (finite) set of classes 
$\mathcal L^i=\{\mathcal L^i_{\theta^i}\mid\theta^i\in\mathcal S^i\}$, $i=1,\dots,N$, of (systems of) differential equations 
with the same independent and dependent variables
such that for any $i\in\{1,\dots,N\}$ and for any $\theta^i\in\mathcal S^i$ the system $\mathcal L^i_{\theta^i}$ is $G$-automorphic 
and the solution set of~$E$ is the (disjoint) union of the solution sets of~$\mathcal L^i_{\theta^i}$, $\theta^i\in\mathcal S^i$, $i=1,\dots,N$. 
For each $i\in\{1,\dots,N\}$, 
the set~$\mathcal S^i$ is the solution set of an auxiliary system~$\mathsf S^i$ for the arbitrary-element tuple $\theta^i$ of the class~$\mathcal L^i$, 
which is called the \emph{resolving system} for the $i$th component of the group foliation.
}
Let us discuss the equation~$\mathcal P_C:=\mathcal P^{+1}_C$ with the potential $C(t,x)=\mu x^{-2}$, where $\mu\in\mathbb R\setminus\{0\}$, within this framework.

The group \smash{$G^{\rm ess}_{\mathcal P_C}=:G$} is constituted by the point transformations of the form~\eqref{eq:SymGroupPmux2}.
Prolonging these transformations to derivatives
and invariantizing~\cite{FelsOlver1999} the coordinate function~$u_t$ of the underlying jet space
under the normalization condition $(\tilde t,\tilde x,\tilde u,\tilde u_{\tilde x})=(0,1,1,0)$,
we find the first-order normalized differential invariant~$I_{10}$ of this group,
\begin{gather*}
I_{10}=x^2\frac{u_t}u-x^2\frac{u_x^2}{u^2}-x\frac{u_x}{u}.
\end{gather*}
All differential invariants of the group~$G$ can be expressed in terms of~$I_{10}$ by means the
operators~$\mathrm D^{\rm i}_1$ and~$\mathrm D^{\rm i}_2$ of invariant differentiation,
\[
\mathrm D^{\rm i}_1=x^2\left(\mathrm D_t-\frac{2u_x}u\mathrm D_x\right),\quad
\mathrm D^{\rm i}_2=x\mathrm D_x \quad\text{with}\quad
[\mathrm D^{\rm i}_1,\mathrm D^{\rm i}_2]=2(I_{02}+I_{10})\mathrm D^{\rm i}_2-2\mathrm D^{\rm i}_1,
\]
where $I_{02}$ is the invariantization of~$u_{xx}$, $I_{02}=x^2(u_{xx}-u_t)/u$.
In particular, a functional basis of such invariants of order less than or equal to two 
is constituted by~$I_{10}$ and%
\footnote{%
In contrast to the normalized differential invariants~$I_{ij}$, $(i,j)\in\mathbb N\times\mathbb N_0$ 
with $\mathbb N_0:=\mathbb N\cup\{0\}$, which are obtained by invariantizing, 
under the above normalization condition, the jet variables~$u_{ij}$ associated with the derivatives $\p^{i+j}u/\p t^i\p x^j$, 
the notation $\hat I_{ij}$ is used for differentiated differential invariants having a similar association.%
}
\[
\hat I_{20}:=\mathrm D^{\rm i}_1 I_{10},\quad \hat I_{11}:=\mathrm D^{\rm i}_2 I_{10},\quad
I_{02}=\frac{[\mathrm D^{\rm i}_1,\mathrm D^{\rm i}_2]I_{10}+2\hat I_{20}}{2\hat I_{11}}-I_{10}.
\]

The highest rank for $G$-automorphic systems is equal to 
the number of independent variables of the equation~$\mathcal P_C$, which is two.
The order~$k$ of automorphic systems to be considered 
in the course of constructing a group foliation for~$\mathcal P_C$ 
is equal to two, which follows from the following conditions.
It should not be less than the order of the equation~$\mathcal P_C$
whereas the cardinality of a functional basis of differential invariants of~$G$ 
whose order is not less than~$k$ should be greater than or equal to
the joint number of independent and dependent variables in~$\mathcal P_C$, which is equal to three.

Thus, each of such rank-two systems~$\mathcal A^2_F$ takes the form 
\begin{gather}\label{FP:eq:A21}
I_{02}=-\mu, \quad \hat I_{11}=F(I_{10},\hat I_{20}),
\end{gather}
where $F$ is a smooth function of~$(I_{10},\hat I_{20})$ satisfying the compatibility condition for~$\mathcal A^2_F$. 
The first equation in~\eqref{FP:eq:A21} is the representation of the equation~$\mathcal P_C$ in terms of differential invariants of~$G$.
Treating now the third-order differential invariants
\[
\hat I_{30}:=(\mathrm D^{\rm i}_1)^2I_{10},\quad \hat I_{21}:=\mathrm D^{\rm i}_2\mathrm D^{\rm i}_1I_{10},\quad \hat I_{12}:=(\mathrm D^{\rm i}_2)^2I_{10},
\quad \hat I_{03}:=\mathrm D^{\rm i}_2I_{02}
\]
as functions of $(I_{10},\hat I_{20})$,
we rewrite the operators of invariant differentiation on solutions of~$\mathcal A^2_F$~as
\[
\mathrm D^{\rm i}_1=\hat I_{20}\p_{I_{10}}+\hat I_{30}\p_{\hat I_{20}},\quad \mathrm D^{\rm i}_2=F\p_{I_{10}}+\hat I_{21}\p_{\hat I_{20}}.
\]
The third-order syzygy
\begin{gather*}
\mathrm D^{\rm i}_1I_{02}-\mathrm D^{\rm i}_2(\hat I_{11}+I_{02})
=-\hat I_{20}-5\hat I_{11}+2(I_{02}+I_{10})^2+4(I_{10}-I_{02})
\end{gather*}
can be seen as a cross-differentiation of the equations of the automorphic system~$\mathcal A^2_F$, when restricted to its solution set.
Thus, the left hand side of this equality is $-\hat I_{12}=-\mathrm D^{\rm i}_2\hat I_{11}$, which upon further expansion gives
\[
-FF_1-\hat I_{21}F_2
=-5F-\hat I_{20}+2I_{10}^2+4(1-\mu)I_{10}+2\mu(\mu+2),
\]
where $F_1:=\p F/\p I_{10}$ and $F_2:=\p F/\p \hat I_{20}$. Assuming first that $F_2\neq0$, we obtain the following expressions
for~$\hat I_{21}$ and~$\hat I_{12}$ in terms of~$F$:
\begin{gather*}
\hat I_{21}=\frac{5F+\hat I_{20}-2I_{10}^2-4(1-\mu)I_{10}-2\mu(\mu+2)-FF_1}{F_2},\\
\hat I_{12}=5F+\hat I_{20}-2I_{10}^2-4(1-\mu)I_{10}-2\mu(\mu+2).
\end{gather*}
A similar expression for~$\hat I_{30}$ can be found by representing $\hat I_{21}$~as
\[
\hat I_{21}=\mathrm D^{\rm i}_2\mathrm D^{\rm i}_1\hat I_{10}=
(\mathrm D^{\rm i}_1\mathrm D^{\rm i}_2-2(I_{02}+I_{10})\mathrm D^{\rm i}_2+2\mathrm D^{\rm i}_1)\hat I_{10}=
\hat I_{20}F_1+\hat I_{30}F_2-2(I_{02}+I_{10})F+2\hat I_{20},
\]
and hence
\[
\hat I_{30}=\frac{(5{-}F_1)F+\hat I_{20}-2(I_{10}+1-\mu)^2+2(1-4\mu)}{F_2^2}-
\frac{\hat I_{20}(F_1+2)-2(I_{10}-\mu)F}{F_2},
\]
while $\hat I_{03}$ vanishes on solutions of~$\mathcal A^2_F$.
The compatibility conditions in pairs of the above equations for third-order differential invariants,
\begin{gather*}
\begin{split}
\mathrm D^{\rm i}_1\hat I_{21}-\mathrm D^{\rm i}_2\hat I_{30}&=-2\hat I_{30}+2\hat I_{21}(I_{10}+I_{02}),\\
\end{split}\\
\mathrm D^{\rm i}_2\hat I_{21}-\mathrm D^{\rm i}_1\hat I_{12}=4\hat I_{21}-4\hat I_{12}(I_{10}+I_{02})-2\hat I_{03}\hat I_{11}-4\hat I_{20}-2\hat I_{11}^2-4\hat I_{11}(I_{10}+I_{02}),
\end{gather*}
when restricted to the solution set of the automorphic system~$\mathcal A^2_F$, give the same \emph{resolving} equation,
\begin{gather}\label{FP:eq:A2}
\begin{split}
&(FF_y-2\bar F)^2F_{zz}-2F(FF_y-\bar F)F_zF_{yz}+F^2F_z^2F_{yy}\\
&+(-4\bar Fy-2F^2+14Fy-4yz-18z)F_z^3
+(\breve FF_y+6Fy+22F+2\breve F+2\bar F)F_z^2=0,
\end{split}
\end{gather}
where \ $y:=I_{10}-\mu$, \ $z:=\hat I_{20}$, \ $\bar F:=5F+z-2y^2-4y-8\mu$ \ and \ $\breve F:=-4F-2y^2-4y-8\mu$.

If $F_2=0$, then $\hat I_{11}$ becomes a function~$\hat F$ of~$I_{10}$ only, and $\hat I_{20}$ is treated as a function~$\check F$ of~$I_{10}$ as well, 
which leads to the rank-one automorphic system~$\mathcal A^1_\theta$ with~$\theta=(\hat F,\check F)$,
\begin{gather}\label{FP:eq:A11}
I_{02}=-\mu, \quad \hat I_{11}=\hat F(I_{10}),\quad \hat I_{20}=\check F(I_{10}).
\end{gather}

Abusing notations, we rewrite the operators of invariant differentiation on solutions of~$\mathcal A^1_\theta$~as
$\mathrm D^{\rm i}_1=\check F\p_{I_{10}}$ and
$\mathrm D^{\rm i}_2=\hat F\p_{I_{10}}.$
The syzygies
\begin{gather*}
\mathrm D^{\rm i}_1I_{02}-\mathrm D^{\rm i}_2(\hat I_{11}+I_{02})
=-\hat I_{20}-5\hat I_{11}+2(I_{02}+I_{10})^2+4(I_{10}-I_{02}),\\
\mathrm D^{\rm i}_2I_{20}-\mathrm D^{\rm i}_1\hat I_{11}
=2\hat I_{20}-2(I_{02}+I_{10})\hat I_{11}
\end{gather*}
on the solution set of~$\mathcal A^1_\theta$ become
\begin{gather}\label{FP:eq:A1}
\begin{split}
&\hat F\hat F'=5\hat F+\check F-2Q(y),\\
&\hat F\check F'-\check F\hat F'=2\check F-2y\hat F,
\end{split}
\end{gather}
where the prime denotes the differentiation with respect to~$y$, and $Q(y):=y^2+2y+4\mu$.
It is reduced by substituting~$\hat F$ from the first equation to the second one to
\begin{gather*}
\hat F^2\hat F''-2(\hat F+Q(y))\hat F'+(6y+14)\hat F-4Q(y)=0.
\end{gather*}

Every value of the function~$F$ or of the pair $(\hat F,\check F)$,
which satisfies the equations~\eqref{FP:eq:A2} or~\eqref{FP:eq:A1}, respectively,
gives the orbit of a solution of~$\mathcal P_C$, and different values correspond to different orbits.
In other words, the problem of finding all $G$-inequivalent solutions of the equation~$\mathcal P_C$
is equivalent to solving the system~\eqref{FP:eq:A21} or the system~\eqref{FP:eq:A11} for each of such values,
which is of course a no-go problem.

\section[Group classifications of equations with time-independent coefficients]
{Group classifications of equations with time-independent\\ coefficients}\label{sec:FPGCTimeIndependent}

All the canonical (up to the general point equivalence) extensions of the essential kernel algebras
within the classes~$\bar{\mathcal K}$ and~$\bar{\mathcal F}$ correspond to equations 
with constant diffusion coefficients and time-independent drifts,
see Theorem~\ref{thm:FPClassificationFP}.
Besides, physically meaningful Fokker--Planck equations, in particular those mentioned in the introduction,
have time-independent coefficients.
Moreover, in contrast to the classes~$\bar{\mathcal K}$ and~$\bar{\mathcal F}$,
complete closed-form lists of inequivalent essential Lie-symmetry extensions
within the classes~$\bar{\mathcal F}'$ and~$\bar{\mathcal K}'$ can be constructed in an explicit form.
The group classification technique to be applied to the latter classes
coincides with that used in Section~\ref{sec:GroupClassificationsWrtEquivGroupsKF},
but now the class~$\mathcal P'^\epsilon$ is to be mapped instead of the class~$\mathcal P^\epsilon$.
In Section~\ref{sec:GroupClassificationsWrtEquivGroupsKF} we describe in detail the procedure
of obtaining Lie-symmetry extensions within the class~$\bar{\mathcal K}$,
so below we focus on Fokker--Planck equations from the class~$\bar{\mathcal F}'$
and then just formulate the classification result for the class~$\bar{\mathcal K}'$.

\subsection[Group classifications of heat equations with time-independent potentials]
{Group classifications of heat equations with time-independent potentials}\label{sec:GroupClassificationP'}

The equations in the class~$\mathcal P'^\epsilon$ that admit extensions of the essential kernel algebra of~$\mathcal P'^\epsilon$
were described using no equivalence, e.g., in~\cite{Gungor2018}.
The list of such equations modulo the $\mathcal G^\sim_{\mathcal P'^\epsilon}$-equivalence
is presented in Theorem~\ref{thm:FP:ClassificationsE0AndP},
where Case~1 should be interpreted as the general case within the class~$\mathcal P'^\epsilon$.
Below we develop another approach and solve the group classification problem
for this class up to the~$G^\sim_{\mathcal P'^\epsilon}$-equivalence, 
simultaneously obtaining the group classification of the class~$\mathcal E'_0$ up to the $G^\sim_{\mathcal E'_0}$-equivalence.

\begin{theorem}\label{FP:thm:GroupClassificationOfTimeIndepPotentials}
The (essential) kernel algebras of the classes~$\mathcal P'^\epsilon$ and~$\mathcal E'_0$
coincide with~$\langle \p_t,u\p_u\rangle$.
A complete list of $G^\sim_{\mathcal P'^\epsilon}$-inequivalent (resp.\ $G^\sim_{\mathcal E'_0}$-inequivalent)
essential Lie-symmetry extensions in the class~$\mathcal P'^\epsilon$ (resp.\ the class~$\mathcal E'_0$)
is exhausted by the cases given in Table~\ref{tab:FPStationaryPotentialsGC}.
\begin{table}[!ht]
\footnotesize
\begin{center}
\caption{Complete group classifications of the classes~$\mathcal P'^\epsilon$ and~$\mathcal E'_0$
up to the $G^\sim_{\mathcal P'^\epsilon}$- and $G^\sim_{\mathcal E'_0}$-equivalences, respectively.%
\label{tab:FPStationaryPotentialsGC}}
\def\arraystretch{1.6} 
\begin{tabular}{|c|c|l|}
\hline
no. & $C$ & \hfil Basis elements of~$\mathfrak g^{\rm ess}_C$ besides $\p_t$ and $u\p_u$\\
\hline
0   & $C(x)$                     & --\\
1   & $\mu x^{-2}$               & $D(t)$, $D(t^2)$ \\
1a  & $\mu x^{-2}+\epsilon x^2$  & $D(\cos4t)$, $D(\sin4t)$\\
1b  & $\mu x^{-2}-\epsilon x^2$  & $D({\rm e}^{4t})$, $D({\rm e}^{-4t})$\\
2   & $0$                        & $P(1)$, $P(t)$, $D(t)$, $D(t^2)$\\
2a  & $\epsilon x^2$             & $P(\cos2t)$, $P(\sin2t)$, $D(\cos4t)$, $D(\sin4t)$\\
2b  & $-\epsilon x^2$            & $P({\rm e}^{2t})$, $P({\rm e}^{-2t})$, $D({\rm e}^{4t})$, $D({\rm e}^{-4t})$\\
2c  & $x$                        & $\bar P(1)$, $\bar P(t)$, $D(t)-\frac32\epsilon\bar P(t^2)$, $D(t^2)-\epsilon\bar P(t^3)$\\
\hline
\end{tabular}
\end{center}
Here $\mu\ne0$,
$D(f)=f\p_t+\frac12f_tx\p_x-(\frac18\epsilon f_{tt}x^2+\frac14f_t)u\p_u$,
$P(f)=f\p_x-\frac12\epsilon f_txu\p_u$, $\bar P(f)=P(f)+Fu\p_u$, and $F$ is a fixed antiderivative of~$f=f(t)$.
\end{table}
\end{theorem}

\begin{proof}
Case~1 of Theorem~\ref{thm:FP:ClassificationsE0AndP} with the general value $C=C(x)$
is the general case in the class~$\mathcal P'^\epsilon$.
To obtain all $G^\sim_{\mathcal P'^\epsilon}$-inequivalent Lie-symmetry extensions in the class~$\mathcal P'^\epsilon$,
we take the following simple steps:
\begin{itemize}
\item
For each of Cases~2 and~3 of Theorem~\ref{thm:FP:ClassificationsE0AndP}, we determine
its counterpart with the most general time-independent potential, say~$\bar C$, in its $G^\sim_{\mathcal P^\epsilon}$-orbit.
\item
Using the equivalence transformations in~$G^\sim_{\mathcal P'^\epsilon}$, we simplify the equation~$\mathcal P^\epsilon_{\bar C}$,
obtaining $G^\sim_{\mathcal P'^\epsilon}$-inequivalent Lie-symmetry extension cases in the class~$\mathcal P'^\epsilon$.
\end{itemize}
Thus, the general equation in the $G^\sim_{\mathcal P^\epsilon}$-orbit of the heat equation with the potential~$\mu/x^2$,
where~$\mu\ne0$, has the potential
\begin{gather*}
\bar C(t,x)=\frac\mu{(x+\zeta)^2}
-\frac\epsilon4\sigma\left(\frac1\sigma\right)_{tt}(x+\zeta)^2
+\frac\epsilon2\zeta_{tt}x
+\theta,
\end{gather*}
where $\sigma$, $\zeta$ and~$\theta$ are arbitrary smooth functions of~$t$ with $\sigma\neq0$, cf.\ Proposition~\ref{FP:pro:EquivalenceGroupHeatPotential}.
The time-independence of this potential implies that it is at most as general as
\begin{gather*}
\bar C(t,x)=\frac{\mu}{(x+\beta)^2}+\alpha(x+\beta)^2+\gamma
\end{gather*}
for some constants~$\alpha$, $\beta$ and~$\gamma$.
Then, using the transformation from~$G^\sim_{\mathcal P'^\epsilon}$
with $c_1=|\alpha|^{1/4}$, $c_2=0$, $c_3=|\alpha|^{1/4}\beta$, $c_4=1$ and $c_5=-\gamma$, see Corollary~\ref{cor:EquivGroupOfP'eps},
we can reduce $\bar C(x)$ to~$\mu/x^2+x^2\sgn\alpha$.

The $\mathcal G^\sim_{\mathcal P'^\epsilon}$-orbit of the equation~$\mathcal P'^\epsilon_0$ consists of the equations
with at most quadratic potentials,
and a complete set of $G^\sim_{\mathcal P'^\epsilon}$-inequivalent equations among them is exhausted
by the equations~$P'^\epsilon_C$ with $C\in\{0,x^2,-x^2,x\}$.

Another way to obtain the group classification of $\mathcal P'^\epsilon$ up to the $G^\sim_{\mathcal P'^\epsilon}$-equivalence
is to extend Cases~2 and~3 of Theorem~\ref{thm:FP:ClassificationsE0AndP} by the inverses of the admissible transformations
listed in Theorem~\ref{FP:thm:EquivGroupoidHeatStationary}.
The corresponding essential Lie invariance algebras are obtained via pushing forward the algebras~$\mathfrak g^{\rm ess}_{\mu/x^2}$
by these inverses.

The $G^\sim_{\mathcal P'^\epsilon}$- and $G^\sim_{\mathcal E'_0}$-equivalences are consistent
via the wide subset~$\mathcal M_{\mathcal E'_0}$ of~\smash{$\mathcal G(G^\sim_{\mathcal E'_0})$},
cf.\ the paragraph before Corollary~\ref{cor:EquivGroupOfP'eps}
and the proof of Corollary~\ref{FP:cor:EquivGroupoidOfE'AndE'_0}.
\end{proof}

\subsection[Group classifications of Fokker-Planck equations with time-independent coefficients]
{Group classifications of Fokker--Planck equations\\ with time-independent coefficients}\label{sec:GroupClassificationF'}

First of all, the group classification of the class~$\bar{\mathcal F}'$ up to the $G^\sim_{\bar{\mathcal F}'}$-equivalence reduces to
the group classification of the class~$\mathcal F'^\epsilon$ up to the $G^\sim_{\mathcal F'^\epsilon}$-equivalence,
cf. the discussion at the beginning of Section~\ref{sec:GroupClassificationsWrtEquivGroupsKF}.
The procedure of the latter classification is similar to that described
in Section~\ref{sec:GroupClassificationsWrtEquivGroupsKF}
for the classes~$\mathcal F^\epsilon$ and~$\mathcal K^\epsilon$,
but for~$\mathcal F'^\epsilon$ we can solve the corresponding problem completely by finding inequivalent solutions
of the ordinary differential equations $\epsilon W_{xx}+C(x)W=-\lambda W$,
where the coefficient~$C$ runs through the set of potentials of canonical representatives
of the Lie-symmetry extensions in the class~$\mathcal P'^\epsilon$.

For each time-independent potential~$C=C(x)$, we consider the family~$\check{\mathfrak T}_C'$ of point transformations
$\Psi_U$: $\tilde t=t$, $\tilde x=x$, $\tilde u=U(t,x)u$ mapping the equation~$\mathcal P'^\epsilon_C=\mathcal P^\epsilon_C$
to Fokker--Planck equations~$\mathcal F'^\epsilon_B=\mathcal F^\epsilon_B$ with time-independent drifts~$B=B(x)$.
The parameter function~$U$ satisfies the equation~$\mathcal P^{-\epsilon}_{-C}$,
and $B=-2\epsilon U_x/U$.
In view of time-independence of drifts $B$ within the class~$\mathcal F'^\epsilon$,
the function~$U$ takes the form $U(t,x)=V(t)W(x)$.
Substituting this expression for~$U$ into the equation~$\mathcal P^{-\epsilon}_{-C}$,
we separate the variables and obtain that, up to a negligible nonzero constant multiplier,
$V(t)={\rm e}^{\lambda t}$, $\lambda\in\mathbb R$,
while the function~$W$ satisfies the equation~$\mathcal W_{C,\lambda}$: $\epsilon W_{xx}+CW=-\lambda W$.
Thus,
\[
\check{\mathfrak T}'_C:=\big\{\Psi_{\lambda,W}\colon\tilde t=t,\ \tilde x=x,\ \tilde u={\rm e}^{\lambda t}W(x)u\mid
 \epsilon W_{xx}+CW=-\lambda W,\,  W(x)\ne0,\,\lambda\in\mathbb R\big\}.
\]

The image of this family of point transformations acting on the equation~$\mathcal P^\epsilon_C$
is the set~$\mathfrak F'_C$ of reduced Fokker--Planck equations with drifts of the form
$B=-2\epsilon U_x/U$, where $U$ runs through the set of nonzero
solutions of the equation $\mathcal P^{-\epsilon}_{-C}$ in separated variables, $U(t,x)={\rm e}^{\lambda t}W(x)$,
\[
\mathfrak F'_C:=\big\{\mathcal F^\epsilon_B\mid B=-2\epsilon W_x/W,\ W=W(x)\ne0\colon\epsilon W_{xx}+CW=-\lambda W,\,\lambda\in\mathbb R\big\}.
\]

\begin{lemma}\label{FP:Lemma2a}
Given Fokker--Planck equations $\mathcal F^\epsilon_B\in\mathfrak F_{C}'$
and $\mathcal F^\epsilon_{\tilde B}$ being $G^\sim_{\mathcal F'^\epsilon}$-equivalent to~$\mathcal F^\epsilon_B$,
there exists a potential~$\tilde C$ such that
the equation~$\mathcal P^\epsilon_{\tilde C}$ is $G^\sim_{\mathcal P'^\epsilon}$-equivalent to~$\mathcal P^\epsilon_C$,
and $\mathcal F^\epsilon_{\tilde B}\in\mathfrak F'_{\tilde C}$.
\end{lemma}

\begin{proof}
Since the equation $\mathcal F^\epsilon_{\tilde B}$ is $G^\sim_{\mathcal F'^\epsilon}$-equivalent to~$\mathcal F^\epsilon_B$,
there exists the equivalence transformation~$\mathscr T\in G^\sim_{\mathcal F'^\epsilon}$ such that
$(\varpi_*\mathscr T)^*\tilde B=c_1^{-1}B$ for a nonzero constant~$c_1$.
It is easy to see that the function~$\tilde C$ with~$(\varpi_*\mathscr T)^*\tilde C=c_1^{-2}C$ satisfies the requirements of the lemma.
\end{proof}

\begin{lemma}\label{FP:Lemma2b}
If $\tilde C=\mathscr T_*C$ for some $\mathscr T\in G^\sim_{\mathcal P'^\epsilon}$,
then the transformation $\breve{\mathscr T}\in G^\sim_{\mathcal F'^\epsilon}$ with $\varpi_*\breve{\mathscr T}=\varpi_*\mathscr T$
and the identity $u$-component maps $\mathfrak F'_C$ onto~$\mathfrak F'_{\tilde C}$,
$\mathfrak F'_{\tilde C}=\breve{\mathscr T}_*\mathfrak F'_C$.
\end{lemma}

\begin{proof}
For an arbitrary Fokker--Planck equation $\mathcal F^\epsilon_{\tilde B}\in\mathfrak F'_{\tilde C}$,
we have $\tilde B=-2\epsilon\tilde U_{\tilde x}/\tilde U$ for some nonzero solution~$\tilde U$ of~$\mathcal P^{-\epsilon}_{-\tilde C}$
of the form $\tilde U(\tilde t,\tilde x)={\rm e}^{\tilde\lambda \tilde t}\tilde W(\tilde x)$.
As an element of~$G^\sim_{\mathcal P'^\epsilon}$,
the transformation~$\mathscr T$ has the form~\eqref{eq:EquivtransOfP'eps}.
Thus, we have $\tilde U(\tilde t,\tilde x)=c_4{\rm e}^{c_5t}U(t,x)$
for some solution~$U$ of~$\mathcal P^{-\epsilon}_{-C}$ of the form $U(t,x)={\rm e}^{\lambda t}W(x)$.
Setting $B:=-2\epsilon U_x/U=-2\epsilon W_x/W$, we obtain $\tilde B(\tilde x)=c_1^{-1}B(x)$.
This means that the Fokker--Planck equation~$\mathcal F^\epsilon_B$ with the drift~$B=-2\epsilon W_x/W$,
which belongs to~$\mathfrak F'_C$, is mapped to~\smash{$\mathcal F^\epsilon_{\tilde B}$} by the equivalence transformation
obtained via prolonging the point transformation
$\tilde t=c_1^2t+c_2$, $\tilde x=c_1x+c_3$, $\tilde u=u$ of~$(t,x,u)$ as $\tilde B=c_1^{-1}B$.
This defines the required transformation \smash{$\breve{\mathscr T}\in G^\sim_{\mathcal F'^\epsilon}$},
which does not depend on the choice of~$\tilde B$.
Hence \smash{$\mathfrak F'_{\tilde C}\subseteq\breve{\mathscr T}_*\mathfrak F'_C$}.
Reverting the argumentation we derive the inverse inclusion.
\end{proof}

\begin{lemma}\label{FP:lem:FPSimilarity2}
The Fokker--Planck equations
$\mathcal F^\epsilon_B:=(\Psi_{\lambda,W})_*\mathcal P^\epsilon_C$ and
\smash{$\mathcal F^\epsilon_{\tilde B}:=(\Psi_{\tilde\lambda,\tilde W})_*\mathcal P^\epsilon_C$} from~$\mathfrak F'_C$
are $G^\sim_{\mathcal F'^\epsilon}$-equivalent if and only if
the associated solutions $U(t,x)={\rm e}^{\lambda t}W(x)$ and $\tilde U(\tilde t,\tilde x)=\smash{{\rm e}^{\tilde\lambda t}\tilde W(x)}$
of the equation~$\mathcal P^{-\epsilon}_{-C}$ are $H_C$-equivalent,
where $H_C:=G^{\rm ess}_{\mathcal P^{-\epsilon}_{-C}}\cap\pi_*G^\sim_{\mathcal P'^{-\epsilon}}$.
\end{lemma}

\begin{proof}
Let Fokker--Planck equations~\smash{$\mathcal F^\epsilon_B,\mathcal F^\epsilon_{\tilde B}\in\mathfrak F'_C$}
be $G^\sim_{\mathcal F'^\epsilon}$-equivalent,
that is, there exists an equivalence transformation~$\mathscr T$ of $\mathcal F'^\epsilon$,
which is of the form~\eqref{eq:EquivTransF'eps}, such that $\tilde B=\mathscr T_* B$.
The last equality implies that the associated functions~$U$ and~$\tilde U$ satisfy
$\tilde U(\tilde t,\tilde x)=f(t)U(t,x)$ with a nonzero function~$f$ of~$t$.
This function can be found from the representations with separated variables for~$U$ and~$\tilde U$,
$f(t)={\rm e}^{(\lambda-c_1^2\tilde\lambda)t-\tilde\lambda c_2}$.
The point transformation
\[
\breve{\mathscr T}\colon\quad\tilde t=c_1^2t+c_2,\quad \tilde x=c_1x+c_3, \quad \tilde u={\rm e}^{(\lambda-c_1^2\tilde \lambda)t-\tilde\lambda c_2} u,\quad
\tilde C=\frac1{c_1^2}\big(C+(\lambda-c_1^2\tilde\lambda)\big)
\]
belongs to~$G^\sim_{\mathcal P'^{-\epsilon}}$.
Since its projection~$\pi_*\breve{\mathscr T}$ maps the solution~$U$ of~$\mathcal P'^{-\epsilon}_{-C}$ to the solution~$\tilde U$ thereof,
then $\pi_*\breve{\mathscr T}\in G^{\rm ess}_{\mathcal P'^{-\epsilon}_{-C}}$.
Thus, the solutions $U$ and~$\tilde U$ are $H_C$-equivalent.

Conversely, let $U$ and~\smash{$\tilde U$} be $H_C$-equivalent solutions of the equation~\smash{$\mathcal P'^{-\epsilon}_{-C}$}.
Then there exist constants~$c_1,\dots,c_5$ with $c_1c_4\neq0$ such that
$\tilde U(\tilde t,\tilde x)=c_4{\rm e}^{c_5t}U(t,x)$, where $\tilde t=c_1^2t+c_2$ and $\tilde x=c_1x+c_3$.
The corresponding drifts~$B$ and~$\tilde B$ are related via $\tilde B(\tilde x)=B(x)/c_1$,
and thus the equations~$\mathcal F^\epsilon_B$ and~$\mathcal F^\epsilon_{\tilde B}$ are $G^\sim_{\mathcal F'^\epsilon}$-equivalent.
\end{proof}

Having at our disposal the group classification of the class~$\mathcal P'^\epsilon$
and Lemmas~\ref{FP:Lemma2a}--\ref{FP:lem:FPSimilarity2},
we solve the group classification problems for the classes~$\mathcal F'^\epsilon$ and~$\bar{\mathcal F}'$
with respect to the corresponding equivalence groups using the mapping method.

\begin{theorem}\label{FP:ClassifyFPStat}
The (essential) kernel algebras of the classes~$\mathcal F'^\epsilon$ and~$\bar{\mathcal F}'$
coincide with~$\langle \p_t,u\p_u\rangle$.
A complete list of $G^\sim_{\mathcal F'^\epsilon}$-inequivalent (resp.\ $G^\sim_{\bar{\mathcal F}'}$-inequivalent)
essential Lie-symmetry extensions in the class~$\mathcal F'^\epsilon$ (resp.\ the class~$\bar{\mathcal F}'$)
is exhausted by the cases given in Table~\ref{tab:FPStationaryDriftsGC}.
\begin{table}[!ht]
\footnotesize
\def\arraystretch{1.5}
\begin{center}
\caption{The group classifications of the classes~$\mathcal F'^\epsilon$, $\bar{\mathcal F}'$, $\mathcal K'^\epsilon$ and~$\bar{\mathcal K}'$
up to the $G^\sim_{\mathcal F'^\epsilon}$-, $G^\sim_{\bar{\mathcal F}'}$-,
$G^\sim_{\mathcal K'^\epsilon}$- and $G^\sim_{\bar{\mathcal K}'}$-equivalences, respectively.
\label{tab:FPStationaryDriftsGC}}
\begin{tabular}{|c|l|l|}
\hline
no.  &\hfil $W$ & \hfil Basis elements of~$\mathfrak g^{\rm ess}_B$ besides $\p_t$ and $u\p_u$\\
\hline
0    & $W(x)$ & --\\[1.5ex]
1    & $\sqrt{|x|}\cos(\alpha\ln|x|)$, $\sqrt{|x|}$, $\sqrt{|x|}\ln|x|$,
     & $D_{B,0}(t)$, $D_{B,0}(t^2)$\\
     & $\sqrt{|x|}\left(a_1|x|^\alpha+a_2|x|^{-\alpha}\right)$ & \\
$1'$ & $\sqrt{|x|}\,\mathcal Z_\beta(x)$, $\sqrt{|x|}\,\tilde{\mathcal Z}_\gamma(x)$ & $D_{B,1}(t)$, $D_{B,1}(t^2)$ \\
$1''$& $\sqrt{|x|}\,\mathcal C_\beta(x)$, $\sqrt{|x|}\,\tilde{\mathcal C}_\gamma(x)$ & $D_{B,-1}(t)$, $D_{B,-1}(t^2)$ \\[1.5ex]
1a   & $|x|^{-1/2}\mathop{\rm Re}\Big((a_1-a_2{\rm i})\mathrm W_{-\frac14\kappa{\rm i},\frac12\nu}({\rm i}x^2)\Big)$
     & $D_{B,\kappa}(\cos4t)$, $D_{B,\kappa}(\sin4t)$\\[1.5ex]
1b   & $|x|^{-1/2}\Big(a_1\mathrm W_{\frac14\kappa,\frac12\nu}(x^2)+a_2\mathrm W_{-\frac14\kappa,\frac12\nu}(-x^2)\Big)$
     & $D_{B,\kappa}({\rm e}^{4t})$, $D_{B,\kappa}({\rm e}^{-4t})$\\[1.5ex]
2    & $1$, $x$                          & $P_B(1)$, $P_B(t)$, $D_{B, 0}(t)$, $D_{B, 0}(t^2)$\\
$2'$ & $\cos x$                          & $P_B(1)$, $P_B(t)$, $D_{B,1}(t)$, $D_{B,1}(t^2)$\\
$2''$& ${\rm e}^x$, $\cosh x$, $\sinh x$ & $P_B(1)$, $P_B(t)$, $D_{B,-1}(t)$, $D_{B,-1}(t^2)$\\[1.5ex]
2a   & $|x|^{-1/2}\mathop{\rm Re}\Big((a_1-a_2{\rm i})\mathrm W_{-\frac14\kappa{\rm i},\frac14}({\rm i}x^2)\Big)$
     & $P_B(\sin2t)$, $P_B(\cos2t)$, $D_{B,\kappa}(\cos4t)$, $D_{B,\kappa}(\sin4t)$\\[1.5ex]
2b   & $|x|^{-1/2}\Big(a_1\mathrm W_{\frac14\kappa,\frac14}(x^2)+a_2\mathrm W_{-\frac14\kappa,\frac14}(-x^2)\Big)$
     & $P_B({\rm e}^{2t})$, $P_B({\rm e}^{-2t})$, $D_{B,\kappa}({\rm e}^{4t})$, $D_{B,\kappa}({\rm e}^{-4t})$\\[1.5ex]
2c   & $a_1\mathrm {Ai}(x)+a_2\mathrm {Bi}(x)$ & $\bar P_B(1)$, $\bar P_B(t)$, $D_{B,0}(t)+\frac32\bar P_B(t^2)$, $D_{B,0}(t^2)+\bar P_B(t^3)$\\
\hline
\end{tabular}
\end{center}
{Here $A=\epsilon$ for the classes~$\bar{\mathcal F}'$ and~$\bar{\mathcal K}'$; $B=2\epsilon\breve\epsilon W_x/W$ with 
$\breve\epsilon=-1$ for the classes~$\mathcal F'^\epsilon$ and~$\bar{\mathcal F}'$ and
$\breve\epsilon=1$ for the classes~$\mathcal K'^\epsilon$ and~$\bar{\mathcal K}'$;
$\kappa\in\mathbb R$,
$\nu\in\mathbb R_{\geqslant0}\cup{\rm i}\mathbb R_{>0}$,
$\alpha,\gamma\in\mathbb R_{>0}$, $\beta\in\mathbb R_{\geqslant0}$, $\alpha,\beta,\nu\neq\frac12$;
$(a_1,a_2)\ne(0,0)$ and $(a_1,a_2)$ are defined up to a nonzero multiplier;
$\mathcal Z_\beta$, $\mathcal C_\beta$,
\smash{$\tilde{\mathcal Z}_\gamma$} and \smash{$\tilde{\mathcal C}_\gamma$} are cylinder functions,
see the proof of Theorem~\ref{FP:ClassifyFPStat} for details;
$\operatorname{Ai}$ and $\operatorname{Bi}$ are the Airy functions of the first and the second kind, respectively;
$\mathrm W_{k,m}$ is a Whittaker function;
$D_{B,\kappa}(f)=f\p_t+\frac12f_tx\p_x-\epsilon\big(\frac18f_{tt}x^2+\frac14\epsilon f_t-\kappa f+\frac14f_txB(x)\big)u\p_u$,
$P_B(f)=f\p_x-\frac12\epsilon\big(f_tx+fB(x)\big)u\p_u$, $\bar P_B(f)=P_B(f)-\epsilon Fu\p_u$, and $F$ is a fixed antiderivative of~$f=f(t)$.}
\end{table}
\end{theorem}

\begin{proof}
Recall that the class~$\mathcal F'^\epsilon$ is related to its superclass~$\bar{\mathcal F}'$
via a subset~$\mathcal M$ of~\smash{$\mathcal G^{G^\sim_{\bar{\mathcal F}'}}$}.
If the projection of an equivalence transformation of the class~$\bar{\mathcal F}'$
to the space with the coordinates $(t,x,u,B)$ relates two equations from the class~$\mathcal F'^\epsilon$,
then in view of Propositions~\ref{FP:pro:EquivGroupStationaryFP} and~\ref{pro:EquivGroupOfF'eps},
this projection belongs to the group~$G^\sim_{\mathcal F'^\epsilon}$.
This means that the $G^\sim_{\mathcal F'^\epsilon}$- and the $G^\sim_{\bar{\mathcal F}'}$-equivalences are consistent
via~$\mathcal M$.
The essential kernel algebras of the classes~$\bar{\mathcal F}'$ and~$\mathcal F'^\epsilon$ are the same,
$\mathfrak g^{\cap\rm ess}_{\bar{\mathcal F}'}=\mathfrak g^{\cap\rm ess}_{\mathcal F'^\epsilon}=\langle \p_t,u\p_u \rangle$.
These two facts together with Proposition~\ref{pro:glinInE0} imply that
the group classification of~$\bar{\mathcal F}'$ up to the $G^\sim_{\bar{\mathcal F}'}$-equivalence reduces
to the group classification of~$\mathcal F'^\epsilon$ up to the $G^\sim_{\mathcal F'^\epsilon}$-equivalence.

For the latter classification, we again use the procedure described at the end of Section~\ref{FP:sec:MapMethod}.
The application of this procedure to the class~$\mathcal F'^\epsilon$ is based on Lemmas~\ref{FP:Lemma2a}--\ref{FP:lem:FPSimilarity2}.
To each potential~$C$ listed in Table~\ref{tab:FPStationaryPotentialsGC},
we relate the set~$\mathfrak F'_C$ of Fokker--Planck equations that are similar to the equation~$\mathcal P^\epsilon_C$.
This set is defined to be the image of the equation~$\mathcal P^\epsilon_C$ with respect to the family~$\check{\mathfrak T}'_C$ of point transformations.
In view of Lemma~\ref{FP:lem:FPSimilarity2}, $G^\sim_{\mathcal F'^\epsilon}$-inequivalent equations from the set~$\mathfrak F'_C$ are associated
with $H_C$-inequivalent particular solutions
of the equation \smash{$\mathcal P'^{-\epsilon}_{-C}$}: $U_t=-\epsilon U_{xx}-C(x)U$
that are of the form~$U(t,x)={\rm e}^{\lambda t}W(x)$.
The group~$H_C$ from Lemma~\ref{FP:lem:FPSimilarity2} induces
a subgroup~$H^\sim_{\mathcal W_C}$ of the equivalence group of the class~$\mathcal W_C$
of linear ordinary differential equations~$\mathcal W_{C,\lambda}$: $\epsilon W_{xx}+C(x)W=-\lambda W$,
which are parameterized by~$\lambda\in\mathbb R$.
Here~$C$ is a fixed function of~$x$.
In other words, instead of describing $H_C$-inequivalent
special solutions of the equation $\mathcal P'^{-\epsilon}_{-C}$,
we can classify solutions of equations in the class~$\mathcal W_C$ up to the $H^\sim_{\mathcal W_C}$-equivalence.%
\footnote{
Given a class~$\mathcal L$ of (systems of) differential equations, 
a solution~$f$ of a system $\mathcal L_\theta\in\mathcal L$ is called $G^\sim_{\mathcal L}$-equivalent to
a solution~$\tilde f$ of a system $\mathcal L_{\tilde\theta}\in\mathcal L$ if there exists $\mathscr T\in G^\sim_{\mathcal L}$ 
such that $\tilde\theta=\mathscr T_*\theta$ and $(\pi_*\mathscr T)^*\tilde f=f$.
}
Therefore, to obtain a group classification list for the class~$\mathcal F'^\epsilon$ modulo the $G^\sim_{\mathcal F'^\epsilon}$-equivalence,
we should find an exhaustive list of $H^\sim_{\mathcal W_C}$-inequivalent solutions of equations in~$\mathcal W_C$.
A classification list for the class~$\mathcal F'^\epsilon$ is constituted by the Fokker--Planck equations~$\mathcal F^\epsilon_B$ whose drifts
are constructed with the help of the found solutions~$W$, $B=-2\epsilon W_x/W$.\looseness=-1

It is convenient to denote $\kappa:=\epsilon\lambda$ and $\nu:=\frac12\sqrt{1-4\epsilon\mu}$.
Thus,
$\nu\in\mathbb R_{\geqslant0}$ if $4\epsilon\mu\leqslant1$ and
$\nu\in{\rm i}\mathbb R_{>0}$ if $4\epsilon\mu>1$.
Thereafter $a_1$ and~$a_2$ are arbitrary real constants with $(a_1,a_2)\ne(0,0)$,
which are defined up to a nonzero multiplier when they appear in expressions for~$B$.
The numeration below is associated with the corresponding cases of Table~\ref{tab:FPStationaryPotentialsGC},
and the proof goes on a case-by-case basis. 
For basic results on involved special functions, see, e.g., \cite{AbramowitzStegun1964,Kamke1977}. 

\medskip\par\noindent{\bf0.}
If $C$ is a general function of~$x$, then $W$ is such a function as well.

\medskip\par\noindent{\bf1.}
$C(x)=\mu x^{-2}$ with $\mu\ne0$.
Consider first the case $\kappa\neq0$.
The general solution of the equation~$\mathcal W_{C,\lambda}$: $x^2W_{xx}+(\kappa x^2+\epsilon\mu)W=0$ is
$W(x)=\sqrt{|x|}\,\mathsf Z_{|\nu|}\big(\sqrt{|\kappa|}\,x\big)$,
where the cylinder function~$\mathsf Z_{|\nu|}$ is
$\mathcal Z_\nu$, $\mathcal C_\nu$, \smash{$\tilde{\mathcal Z}_{|\nu|}$} or~\smash{$\tilde{\mathcal C}_{|\nu|}$}
if $4\epsilon\mu\leqslant1$ and $\kappa>0$, $4\epsilon\mu\leqslant1$ and $\kappa<0$,
$4\epsilon\mu>1$ and $\kappa>0$ or $4\epsilon\mu>1$ and $\kappa<0$, respectively.
Here~$\mathcal Z_\nu$ and~$\mathcal C_\nu$ are linear combinations
of linearly independent Bessel and of linearly independent modified Bessel functions, respectively.
The cylinder function~$\tilde{\mathcal Z}_{|\nu|}$ is an arbitrary linear combination of the (linearly independent)
modifications of the Hankel functions~\smash{$\mathrm H^{(1)}_\nu$} and~\smash{$\mathrm H^{(2)}_\nu$},
\begin{gather*}
\tilde{\mathcal Z}_{|\nu|}(x)=a_1\tilde{\mathrm H}^{(1)}_\nu(x)+a_2\tilde{\mathrm H}^{(2)}_\nu(x)\\
\phantom{\tilde{\mathcal Z}_{|\nu|}(x)}=\frac{a_1}2\big({\rm e}^{-\frac12\nu\pi}\mathrm H^{(1)}_{\nu}(x)+{\rm e}^{\frac12\nu\pi}\mathrm H^{(2)}_{\nu}(x)\big)
+\frac{a_2}{2{\rm i}}\big({\rm e}^{-\frac12\nu\pi}\mathrm H^{(1)}_{\nu}(x)-{\rm e}^{\frac12\nu\pi}\mathrm H^{(2)}_{\nu}(x)\big),
\end{gather*}
and the cylinder function~$\tilde{\mathcal C}_{|\nu|}$ is an arbitrary linear combination of the modified Bessel function~$\mathrm K_{\nu}$
and the modification~$\tilde {\mathrm I}_\nu$ of the modified Bessel function~$\mathrm I_\nu$ (for $\nu\neq0$),
\[
\tilde{\mathcal C}_{|\nu|}(x)=a_1\tilde {\mathrm I}_\nu(x)+a_2\mathrm K_\nu(x)=
\frac{a_1\pi{\rm i}}{2\sin(\nu\pi)}\big(\mathrm I_{\nu}(x)+\mathrm I_{-\nu}(x)\big)+a_2\mathrm K_{\nu}(x).
\]
It is not possible to define a numerically satisfactory companion to $\mathrm K_\nu$, which remains finite as $\nu\to0$.
All the above functions are real-valued and represent the general solution of~$\mathcal W_{C,\lambda}$
for each related values of parameters, see~\cite{Dunster1990} for more details
on the functions~\smash{$\tilde{\mathrm H}^{(1)}_\nu$}, \smash{$\tilde{\mathrm H}^{(2)}_\nu$}, \smash{$\tilde {\mathrm I}_\nu$} and~$\mathrm K_{\nu}$.
The group~$H^\sim_{\mathcal W_C}$ consists of the transformations
$\tilde x=c_1x$, $\tilde W=c_2W$, $\tilde\lambda=c_1^{-2}\lambda$, where $c_1$ and~$c_2$ are nonzero arbitrary constants.
Therefore, the arguments of the above cylinder functions can be scaled to~$x$
using transformations from~\smash{$H^\sim_{\mathcal W_C}$}.

If $\kappa=0$, then the equation~$\mathcal W_{C,\lambda}$ degenerates to the Euler equation $x^2W_{xx}+\epsilon\mu W=0$.
Depending on the value of $\sgn(1-4\epsilon\mu)$, $-1$, $0$ or~$1$,
its general solution assumes the form
\begin{gather*}
W(x)=\sqrt{|x|}\big(a_1\cos(|\nu|\ln|x|)+a_2\sin(|\nu|\ln|x|)\big),\quad
W(x)=\sqrt{|x|}\big(a_1+a_2\ln|x|\big)\quad\mbox{or}\\
W(x)=\sqrt{|x|}\big(a_1|x|^\nu+a_2|x|^{-\nu}\big),
\end{gather*}
and up to the $H^\sim_{\mathcal W_C}$-equivalence, each of the first two values of~$W$
reduces to an element of $\big\{\sqrt{|x|}\cos(|\nu|\ln|x|)\big\}$ or $\big\{\sqrt{|x|},\sqrt{|x|}\ln{|x|}\big\}$,
respectively, whereas the last value cannot be simplified.
Note that the values $W(x)=|x|^\sigma$ with $\sigma\in\mathbb R\setminus\{0,1\}$
and $W(x)=\sqrt{|x|}\cos(|\nu|\ln|x|)$ with $\nu\in{\rm i}\mathbb R\setminus\{0\}$
give the drifts from Cases~2 and~3 of Theorem~\ref{thm:FPClassificationFP}.

\medskip\par\noindent{\bf1a.}
$C(x)=\mu x^{-2}+\epsilon x^2$ with $\mu\ne0$.
The equation~$\mathcal W_{C,\lambda}$ is \mbox{$x^2W_{xx}+(x^4+\kappa x^2+\epsilon\mu)W=0$}.
As a pair of its real-valued linearly independent solutions, we can take the real and imaginary parts of
$|x|^{-1/2}\mathrm W_{-\frac14\kappa{\rm i},\frac12\nu}({\rm i}x^2)$,
where $\mathrm W_{k,m}$ is a Whittaker function, and $\nu\ne1/2$.
If $\nu\in2\mathbb N_0+1$, the function~$W(x)$ can be represented in terms of regular and irregular Coulomb functions,
\[
W(x)=|x|^{-1/2}\left(a_1 F_{\frac{\nu-1}2}\left(-\kappa,\tfrac12 x^2\right)+a_2 G_{\frac{\nu-1}2}\left(-\kappa,\tfrac12 x^2\right)\right).
\]
The group~$H^\sim_{\mathcal W_C}$, which is constituted by the point transformations
$\tilde x=\pm x$, $\tilde W=c_1W$, $\tilde\lambda=\lambda$, where $c_1$ is a nonzero arbitrary constant,
does not allow us to gauge the parameters of~$W$.

\medskip\par\noindent{\bf1b.}
$C(x)=\mu x^{-2}-\epsilon x^2$ with $\mu\ne0$.
The equation~$\mathcal W_{C,\lambda}$ is $x^2W_{xx}+(-x^4+\kappa x^2+\epsilon\mu)W=0$.
Its (real-valued) general solution can be represented in the form
\[W(x)=|x|^{-1/2}\big(a_1\mathrm W_{\frac14\kappa,\frac12\nu}(x^2)+a_2\mathrm W_{-\frac14\kappa,\frac12\nu}(-x^2)\big)\]
for any value of~$\nu$, and $\nu\ne1/2$.
The group~$H^\sim_{\mathcal W_C}$ is the same as that in Case~1a, and hence no gauging of parameters in~$W$ is possible.

\medskip\par\noindent{\bf2.}
$C=0$. The group~\smash{$H^\sim_{\mathcal W_C}$} consists of the point transformations
$\tilde x=c_1x+c_2$, $\tilde W=c_3W$, $\tilde\lambda=c_1^{-2}\lambda$,
where $c_1$, $c_2$, and $c_3$ are arbitrary constants with $c_1c_3\neq0$.
The general solution of the equation~$\mathcal W_{C,\lambda}$: $W_{xx}+\kappa W=0$ is
$W(x)=a_1x+a_2$ if $\kappa=0$,
$W(x)=a_1{\rm e}^{\sqrt{-\kappa}x}+a_2{\rm e}^{-\sqrt{-\kappa}x}$ if $\kappa<0$ and
$W(x)=a_1\sin(\sqrt\kappa x)+a_2\cos(\sqrt\kappa x)$ if $\kappa>0$.
Up to the $H^\sim_{\mathcal W_C}$-equivalence, each of these values of~$W$
reduces to an element of $\{1,x\}$, of $\{{\rm e}^x,\sinh x,\cosh x\}$ and to $\cos x$, respectively.

\medskip\par\noindent{\bf2a, 2b.}
$C(x)=\pm x^2$. These cases are specifications of Cases~1a and~1b for $\nu=1/2$ with the same group~$H^\sim_{\mathcal W_C}$.

\medskip\par\noindent $\boldsymbol{2\rm c.}$
$C(x)=x$. We replace this potential by the $G^\sim_{\mathcal P'^\epsilon}$-equivalent potential $C(x)=-\epsilon x$.
Then the equation~$\mathcal W_{C,\lambda}$ takes the form  $W_{xx}-(x-\kappa)W=0$,
and its general solution is $W(x)=a_1\operatorname{Ai}(x-\kappa)+a_2\operatorname{Bi}(x-\kappa)$.
Here $\operatorname{Ai}$ and $\operatorname{Bi}$ are the Airy functions of the first and the second kinds.
The group~\smash{$H^\sim_{\mathcal W_C}$} is constituted by the point transformations
$\tilde x=c_1x+c_2$, $\tilde W=c_3W$, $\tilde\lambda=c_1^{-2}\lambda$,
where $c_1$, $c_2$ and~$c_3$ are arbitrary constants with $c_1c_3\neq0$.
Therefore, up to the \smash{$H^\sim_{\mathcal W_C}$}-equivalence, we can set
$W(x)=a_1\operatorname{Ai}(x)+a_2\operatorname{Bi}(x)$.
\end{proof}

\begin{remark}\label{rem:AddEquivalencesFPStat}
It is evident from the classification procedure that there are additional equivalences
between classification cases listed in Table~\ref{tab:FPStationaryDriftsGC},
cf.\ also Theorem~\ref{thm:FPClassificationFP}.
More specifically, the Fokker--Planck equations corresponding to Cases~1, $1'$, $1''$, 1a and~1b
(resp.\ Cases~2, $2'$, $2''$, 2a, 2b and~2c)
with the same value of~$\nu:=\frac12\sqrt{1-4\epsilon\mu}$ (see the proof of Theorem~\ref{FP:ClassifyFPStat})
are pairwise $\mathcal G^\sim_{\mathcal F'^\epsilon}$-equivalent.
It is easy to construct transformations realizing the additional equivalences
between equations related to classification list.
Let
$\mathcal F^\epsilon_{B_1}=(\Psi_{\lambda_1,W_1})_*\mathcal P^\epsilon_C$ and
$\mathcal F^\epsilon_{B_2}=(\Psi_{\lambda_2,W_2})_*\mathcal P^\epsilon_C$ be two such equations,
where $C(x)=\mu x^{-2}$ with $\mu\ne0$ in Cases~1$*$ and $C=0$ in Cases~2$*$.
Then the transformation $\Phi:=\Psi_{\lambda_2,W_2}\circ(\Psi_{\lambda_1,W_1})^{-1}$
maps $\mathcal F^\epsilon_{B_1}$ to~$\mathcal F^\epsilon_{B_2}$,
$\mathcal F^\epsilon_{B_2}=\Phi_*\mathcal F^\epsilon_{B_1}$.
The canonical representatives of the corresponding $\mathcal G^\sim_{\mathcal F'^\epsilon}$-equivalence class
are presented in Theorem~\ref{thm:FPClassificationFP}.
\end{remark}

\begin{remark}
Cases 2, $2'$, $2''$, 2a and~2b of Table~\ref{tab:FPStationaryDriftsGC}
are counterparts of Cases~1, $1'$, $1''$, 1a and~1b
for~$\nu=\frac12$ with further Lie-symmetry extensions (see the proof of Theorem~\ref{FP:ClassifyFPStat})
but Case~2c does not admit such an interpretation.
It is a consequence of the fact
that the set $\mathrm t\big(\mathrm s^{-1}(C)\big)$ with $C=0$ in $\mathcal G^\sim_{\mathcal P'^\epsilon}$
is wider than the similar set with $C(x)=\mu/x^2$, where $\mu\neq0$.
\end{remark}

\subsection[Group classifications of Kolmogorov equations with time-independent coefficients]
{Group classifications of Kolmogorov equations\\ with time-independent coefficients}\label{sec:GroupClassificationK'}

We only formulate the assertions constituting the theoretical basis of the procedure of classifying
reduced Kolmogorov equations with time-independent drifts;
their proofs are analogous to those of Lemmas~\ref{FP:Lemma2a}--\ref{FP:lem:FPSimilarity2}.
For each time-independent potential~$C=C(x)$, by~$\mathfrak T'_C$ and~$\mathfrak K'_C$
we denote the family of the point transformations mapping the equation~$\mathcal P^\epsilon_C$ to such Kolmogorov equations
and the set of corresponding images, respectively,
\begin{gather*}
\mathfrak T'_C=\big\{\Phi_{\lambda,W}\colon\tilde t=t,\ \tilde x=x,\ \tilde u={\rm e}^{\lambda t}u/W(x)\mid
\lambda\in\mathbb R,\,W\ne0\colon\,\epsilon W_{xx}+CW=-\lambda W\big\},
\\
\mathfrak K'_C:=(\mathfrak T'_C)_*\mathcal P^\epsilon_C=\big\{\mathcal K^\epsilon_B\mid
B=2\epsilon W_x/W,\ W=W(x)\ne0\colon\epsilon W_{xx}+CW=-\lambda W,\,\lambda\in\mathbb R \big\}.
\end{gather*}

\begin{lemma}
(i) Given Kolmogorov equations $\mathcal K^\epsilon_B\in\mathfrak K'_C$
and \smash{$\mathcal K^\epsilon_{\tilde B}$} being $G^\sim_{\mathcal K'^\epsilon}$-equivalent to~$\mathcal K^\epsilon_B$,
there exists a function~$\tilde C$ of~$x$ such that
the equation~$\mathcal P^\epsilon_{\tilde C}$ is $G^\sim_{\mathcal P'^\epsilon}$-equivalent to the equation~$\mathcal P^\epsilon_C$,
and $\mathcal K^\epsilon_{\tilde B}\in\mathfrak K_{\tilde C}'$.

(ii) If $\tilde C=\mathscr T_*C$ for some \smash{$\mathscr T\in G^\sim_{\mathcal P'^\epsilon}$},
then the transformation \smash{$\breve{\mathscr T}\in G^\sim_{\mathcal K'^\epsilon}$} with $\varpi_*\breve{\mathscr T}=\varpi_*\mathscr T$
and the identity $u$-component maps $\mathfrak K'_C$ onto~$\mathfrak K_{\tilde C}'$, $\mathfrak K_{\tilde C}'=\breve{\mathscr T}_*\mathfrak K'_C$.

(iii) The Kolmogorov equations
$\mathcal K^\epsilon_B:=(\Phi_{\lambda,W})_*\mathcal P^\epsilon_C$ and
$\mathcal K^\epsilon_{\tilde B}:=(\Phi_{\tilde \lambda,\tilde W})_*\mathcal P^\epsilon_C$ from $\mathfrak K'_C$
are $G^\sim_{\mathcal K'^\epsilon}$-equivalent if and only if
the associated solutions ${\rm e}^{\lambda t}W(x)$ and \smash{${\rm e}^{\tilde\lambda t}\tilde W(x)$}
of the equation~$\mathcal P^\epsilon_C$ are $G^{\rm ess}_{\mathcal P^\epsilon_C}\cap\pi_*G^\sim_{\mathcal P'^\epsilon}$-equivalent.
\end{lemma}

The group classification of the class~$\bar {\mathcal K}'$ up to the $G^\sim_{\bar {\mathcal K}'}$-equivalence
reduces to
the group classification of the class~$\mathcal K'^\epsilon$ up to the $G^\sim_{\mathcal K'^\epsilon}$-equivalence,
and the latter classification is very similar to that done in Theorem~\ref{FP:ClassifyFPStat}.
Its result can be summarized by Table~\ref{tab:FPStationaryDriftsGC} as well, cf. Theorem~\ref{thm:FPClassificationFP}.

\begin{theorem}\label{FP:ClassifyKStat}
The essential kernel algebras of the classes~$\mathcal K'^\epsilon$ and~$\bar{\mathcal K}'$
coincide with~$\langle \p_t,u\p_u\rangle$.
A complete list of $G^\sim_{\mathcal K'^\epsilon}$-inequivalent (resp.\ $G^\sim_{\bar{\mathcal K}'}$-inequivalent)
essential Lie-symmetry extensions in the class~$\mathcal K'^\epsilon$ (resp.\ the class~$\bar{\mathcal K}'$)
is exhausted by the cases given in Table~\ref{tab:FPStationaryDriftsGC}.
\end{theorem}

\section{Conclusion}\label{sec:FPConclusion}

The mapping method of group classification is an efficient tool for solving group classification problems.
Its simplest version is related to similar classes of differential equations,
where the respective equivalences are obviously consistent and
there is a clear one-to-one correspondence between the associated kernel algebras and between classification lists.

Another simple version of the method is the reduction of the group classification of a class~$\bar{\mathcal L}$
to that of its subclass~$\mathcal L$ with the help of mapping~$\bar{\mathcal L}$ onto~$\mathcal L$ by
a wide family of admissible transformations from the action groupoid of the equivalence group of~$\bar{\mathcal L}$.
If the stabilizer subgroup of this group with respect to~$\mathcal L$ coincides with the equivalence group of~$\mathcal L$,
then any classification list for~$\mathcal L$ is a classification list for~$\bar{\mathcal L}$.

In the present paper, we have extended the mapping method of group classification to weakly similar classes.
But the weaker a relation between classes is, the less interrelated their group classifications are.
Thus, in general there is no relation between the kernel algebras of (genuinely) weakly similar classes.
This urged us to reformulate the classical statement of the group classification problem in Section~\ref{FP:sec:GroupClassProblem},
replacing the kernel algebra of the class under study
by a family of uniform subalgebras of the maximal Lie invariance algebras of systems from this class
and thus revisiting the entire framework of group classification of classes of differential equations.
Given weakly similar classes~$\mathcal L$ and~$\tilde{\mathcal L}$,
each of such families for~$\mathcal L$ is associated with such a family for~$\tilde{\mathcal L}$, and vice versa.
We have suggested various techniques within the framework of the mapping method of group classification
depending on normalization properties of involved classes of differential equations
and on the equivalences to be used in the course of group classification.

Before discussing the obtained results for Kolmogorov and Fokker--Planck equations,
we note that the classes in the pairs
$(\bar{\mathcal K},\bar{\mathcal F})$, $(\mathcal K^\epsilon,\mathcal F^{-\epsilon})$,
$(\bar{\mathcal K}',\bar{\mathcal F}')$ and $(\mathcal K'^\epsilon,\mathcal F'^{-\epsilon})$
are the adjoint of each other in the sense that the adjoint of any equation in the first class
belongs to the second class, and vice versa.
The adjointness leads to an intimate connection between the equivalence groupoids of the first and the second classes,
between their equivalence groups as well as between the maximal Lie invariance algebras of equations from these classes,
and thus the studies of these objects for the first and the second classes are similar.
This is why we have considered in detail only one of the classes in each pair for particular purposes.
We follow the same convention in the rest of the conclusion as well.

The fact underlying the application of the mapping method to group classification
of the classes~$\bar{\mathcal K}$ and~$\mathcal K^\epsilon$
up to the $G^\sim_{\bar{\mathcal K}}$- and $G^\sim_{\mathcal K^\epsilon}$-equivalences, respectively,
is that these classes are weakly similar to the class~$\mathcal P^\epsilon$ of heat equations with potentials.
Moreover, the group classification of the class~$\bar{\mathcal K}$ is reduced
to the group classification of the class~$\mathcal K^\epsilon$ via mapping~$\bar{\mathcal K}$ onto~$\mathcal K^\epsilon$
by a wide family of admissible transformations from the action groupoid of~$G^\sim_{\bar{\mathcal K}}$.
This reduction is particularly convenient since the class~$\mathcal K^\epsilon$
is weakly similar to the class~$\mathcal P^\epsilon$ in a specific way,
namely, point transformations establishing the similarity can be chosen to have the identity $(t,x)$-components.
Unfortunately, it appeared that the explicit listing of all Lie-symmetry extension cases in the class~$\mathcal K^\epsilon$
up to the $G^\sim_{\mathcal K^\epsilon}$-equivalence is actually a ``no-go'' problem
because it is equivalent to finding all $G_{\mathcal P^\epsilon_C}$-inequivalent solutions of the heat equations~$\mathcal P^\epsilon_C$,
where $C$ runs through the list from Theorem~\ref{thm:FP:ClassificationsE0AndP}.
This claim has additionally been interpreted by means of the group foliation method in Section~\ref{FP:sec:GroupFoliation}.
On the other hand, the mapping method allowed us at least to describe the solution of the group classification problems
for~$\bar{\mathcal K}$ and~$\mathcal K^\epsilon$ in Theorem~\ref{thm:GroupClassificationOfbarKbarFKepsFeps}.

Although the classification procedures for the classes~$\bar{\mathcal K}$ and~$\bar{\mathcal K}'$ are analogous,
the mapping method has demonstrated much better results for the class~$\bar{\mathcal K}'$,
which is weakly similar to the class~$\mathcal P'^\epsilon$.
More specifically, we have reduced the group classification of the class~$\bar{\mathcal K}'$ to that of the class~$\mathcal K'^\epsilon$
via mapping~$\bar{\mathcal K}'$ onto~$\mathcal K'^\epsilon$
by a wide family of admissible transformations from the action groupoid of the equivalence group~$G^\sim_{\bar{\mathcal K}'}$ of~$\bar{\mathcal K}'$.
The class~$\mathcal K'^\epsilon$ is weakly similar to the class~$\mathcal P'^\epsilon$
with respect to point transformations with the identity $(t,x)$-components.
While the technical crux of the mapping method for the class~$\mathcal K^\epsilon$ lies in solving the equations $U_t=\epsilon U_{xx}+C(t,x)U$
for the values of the potential~$C$ collected in Theorem~\ref{thm:FP:ClassificationsE0AndP},
for the class~$\mathcal K'^\epsilon$ the corresponding equations are $\epsilon U_{xx}+(C(x)+\lambda)U=0$,
whose explicit general solutions can be found for all associated potentials and $\lambda\in\mathbb R$
in terms of elementary or special functions.
This allows us to present an explicit classification list for the class~$\mathcal K'^\epsilon$ up to the $G^\sim_{\mathcal K'^\epsilon}$-equivalence
(resp.\ for the class~$\bar{\mathcal K}'$ up to the $G^\sim_{\bar{\mathcal K}'}$-equivalence)
in Table~\ref{tab:FPStationaryDriftsGC}.

However nice this result is, it also indirectly confirms that the difficulty level of problems concerning admissible transformations
in the superclass~$\bar{\mathcal K}$ and Lie symmetries of equations therein is too high.
Firstly, let us discuss Lie symmetries of equations in~$\bar{\mathcal K}$.
There are only four (resp.\ three) families of Lie-symmetry extension cases in the class~and~$\bar{\mathcal K}$ (resp.\ $\bar{\mathcal K}'$)
modulo its equivalence groupoid and they are parameterized by at most one real parameter, but there are already eleven
(in an ``optimistical'' count, which became nineteen in a ``pessimistic'' count, see Table~2)
$G^\sim_{\bar{\mathcal K}'}$-inequivalent Lie-symmetry extension cases in the class~$\bar{\mathcal K}'$.
Furthermore, in some of these cases the corresponding values of arbitrary elements are expressed in terms of special functions
and parameterized by up to three real parameters.
Recall that arbitrary-element values expressed in terms of special functions arose in~\cite{VaneevaPopovychSophocleous2009} as well,
where the mapping method was applied to another pair of weakly similar classes.
Moreover, $G^\sim_{\bar{\mathcal K}'}$-inequivalent Lie-symmetry extension cases are also $G^\sim_{\bar{\mathcal K}}$-equivalent,
and therefore all of them are contained in a classification list of the class~$\bar{\mathcal K}$
with respect to its equivalence group~$G^\sim_{\bar{\mathcal K}}$
but this is only a small part of the list.
In fact, this list is much wider since Lie-symmetry extension cases with time-dependent drifts are therein as well
in view of the weakness of the $G^\sim_{\bar{\mathcal K}}$-equivalence,
and as discussed above, it cannot be presented in an explicit form.

To discuss the classification of admissible transformations of the class~$\mathcal K'^\epsilon$ (resp.\ $\bar{\mathcal K}'$),
let us look first at its counterpart for the class~$\mathcal P'$ in accordance with the mapping technique.
Theorem~\ref{FP:thm:EquivGroupoidHeatStationary} states
that the class~$\mathcal P'$ possesses the minimal self-consistent generating
(up to the $G^\sim_{\mathcal P'}$-equivalence and the linear superposition of solutions)
set~$\mathcal B$ of admissible transformations, which consists of three families.
The admissible transformations from~$\mathcal B$ are precisely those
that induce all independent additional equivalences between the $G^\sim_{\mathcal P'}$-inequivalent Lie-symmetry extension cases
in the class~$\mathcal P'$ that are listed in Table~\ref{tab:FPStationaryPotentialsGC}.
An analogous generating set~$\tilde{\mathcal B}$ for the class~$\mathcal K'^\epsilon$ is constituted by the admissible transformations
inducing all independent additional equivalences between the $G^\sim_{\mathcal K'^\epsilon}$-inequivalent Lie-symmetry extension cases
in the class~$\mathcal K'^\epsilon$ that are collected in Table~\ref{tab:FPStationaryDriftsGC};
cf.\ Remark~\ref{rem:AddEquivalencesFPStat}.
The set~$\tilde{\mathcal B}$ is definitely minimal for the class~$\mathcal K'^\epsilon$ but, as with Lie symmetries,
is significantly more voluminous than its counterpart~$\mathcal B$.
Each element of~$\mathcal B$ is associated with a family of elements of~$\tilde{\mathcal B}$.
Moreover, many elements of~$\tilde{\mathcal B}$ correspond to admissible identity transformations in~$\mathcal P'$.
It is then clear that no generating set of admissible transformations of the class~$\bar{\mathcal K}$
can be presented in an explicit form.

Justifying our interest in the subject of the paper in the introduction, we have mentioned
the class~$\mathcal L$ of reaction--diffusion equations $\mathcal L_{c,g}$: $u_t=cu_x^{-2}u_{xx}+g(u)$ with $c\ne0$,
which was under study in~\cite{OpanasenkoBoykoPopovych2020}.
Using results obtained in the present paper,
we can solve the group classification problem for the class~$\mathcal L$ up to the $G^\sim_{\mathcal L}$-equivalence.
More specifically, we can translate the group classification list for the class~$\mathcal K'^\epsilon$
that is given in Table~\ref{tab:FPStationaryDriftsGC} to that for the class~$\mathcal L$.
Thus, the constant arbitrary element~$c$ of~$\mathcal L$ can be gauged to~$\epsilon$ by scaling equivalence transformations of~$\mathcal L$,
and the gauged subclass of equations with $c=\epsilon$ is similar to the class~$\mathcal K'^\epsilon$
via the hodograph transformation~$\tilde t=t$, $\tilde x=u$, $\tilde u=x$, $g=-B$.
The Lie symmetries of equations from the class~$\mathcal L$ should be classified within the revisited framework of group classification,
see Section~\ref{FP:sec:GroupClassProblem}.
The uniform family $\{\mathfrak g^{\rm lin}_{\mathcal K'^\epsilon_B}\mid B_t=0\}$
of subalgebras of the algebras~$\mathfrak g_{\mathcal K'^\epsilon_B}$ that are related to the linear superposition of solutions
is mapped by the hodograph transformation onto the uniform family
of subalgebras $\mathfrak g^{\rm unf}_{\mathcal L_{\epsilon,g}}=\langle f(t,u)\p_x\rangle$ of the algebras~$\mathfrak g_{\mathcal L_{\epsilon,g}}$,
where for each fixed value of the arbitrary element~$g$, the parameter function $f=f(t,u)$ runs through the solution set of the equation $f_t=\epsilon f_{uu}-g(u)f_u$.
A complete list of $G^\sim_{\mathcal L}$-inequivalent essential Lie-symmetry extensions in~$\mathcal L$
is given by the modification of Table~2 with $g=-B$ and permuted~$x$ and~$u$.

\section*{Acknowledgments}
The authors are grateful to the anonymous reviewer for valuable remarks.
The research of ROP was supported by the Austrian Science Fund (FWF), projects P28770 and P33594,
while SO acknowledges the support of NSERC.
Also ROP expresses his gratitude to Memorial University of Newfoundland for hospitality
during his stay in St.~John's, while SO to the University of Vienna.

\footnotesize

\end{document}